\documentclass[11pt,a4paper]{article}
\usepackage[utf8x]{inputenc}
\usepackage{hyperref}
\usepackage{index}
\usepackage{tocbibind}
\usepackage{fullpage}
\usepackage[frenchb]{babel}

\usepackage{bbm}
\usepackage[bbgreekl]{mathbbol}
\usepackage{nccrules}
\usepackage{tocbibind}
\usepackage[intlimits]{amsmath}
\usepackage{amsthm}
\usepackage{amssymb}
\usepackage{mathrsfs}
\usepackage{stmaryrd}
\usepackage[all]{xy}
\newdir{ >}{{}*!/-10pt/\dir{>}}

\newcommand{\Z}{\ensuremath{\mathbb{Z}}}
\newcommand{\Q}{\ensuremath{\mathbb{Q}}}
\newcommand{\R}{\ensuremath{\mathbb{R}}}
\newcommand{\C}{\ensuremath{\mathbb{C}}}

\newcommand{\A}{\ensuremath{\mathbb{A}}}

\newcommand{\Tr}{\ensuremath{\mathrm{tr}\,}}

\newcommand{\Qform}[1]{\ensuremath{\langle #1 \rangle}}
\newcommand{\Resprod}{\ensuremath{{\prod}'}}

\newcommand{\dd}{\ensuremath{\,\mathrm{d}}}

\newcommand{\angles}[1]{\ensuremath{\langle #1 \rangle}}
\newcommand{\mes}{\ensuremath{\mathrm{mes}}}
\newcommand{\sgn}{\ensuremath{\mathrm{sgn}}}

\newcommand{\identity}{\ensuremath{\mathrm{id}}}

\newcommand{\Hom}{\ensuremath{\mathrm{Hom}}}
\newcommand{\End}{\ensuremath{\mathrm{End}}}
\newcommand{\rightiso}{\ensuremath{\stackrel{\sim}{\rightarrow}}}

\newcommand{\Ker}{\ensuremath{\mathrm{Ker}\,}}

\newcommand{\Image}{\ensuremath{\mathrm{Im}\,}}

\newcommand{\Ad}{\ensuremath{\mathrm{Ad}\,}}
\newcommand{\ad}{\ensuremath{\mathrm{ad}\,}}

\newcommand{\Gm}{\ensuremath{\mathbb{G}_\mathrm{m}}}

\newcommand{\Res}{\ensuremath{\mathrm{Res}}}

\newcommand{\GL}{\ensuremath{\mathrm{GL}}}
\newcommand{\SO}{\ensuremath{\mathrm{SO}}}
\newcommand{\Spin}{\ensuremath{\mathrm{Spin}}}
\newcommand{\U}{\ensuremath{\mathrm{U}}}

\newcommand{\PGL}{\ensuremath{\mathrm{PGL}}}

\newcommand{\so}{\ensuremath{\mathfrak{so}}}
\newcommand{\spin}{\ensuremath{\mathfrak{spin}}}
\newcommand{\syp}{\ensuremath{\mathfrak{sp}}}

\newcommand{\Sp}{\ensuremath{\mathrm{Sp}}}

\newcommand{\Spt}{\ensuremath{\overline{\mathrm{Sp}}}}
\newcommand{\Mp}{\ensuremath{\widetilde{\mathrm{Sp}}}}
\newcommand{\MMp}[1]{\ensuremath{\widetilde{\mathrm{Sp}}^{(#1)}}}
\newcommand{\inv}{\ensuremath{\mathrm{inv}}}

\theoremstyle{plain}
\newtheorem{proposition}{Proposition}[section]
\newtheorem{lemma}[proposition]{Lemme}
\newtheorem{theorem}[proposition]{Théorème}
\newtheorem{corollary}[proposition]{Corollaire}

\theoremstyle{definition}
\newtheorem{definition}[proposition]{Définition}
\newtheorem{definition-theorem}[proposition]{Définition-Théorème}
\newtheorem{definition-proposition}[proposition]{Définition-Proposition}
\newtheorem{hypothesis}[proposition]{Hypothèse}

\newtheorem{remark}[proposition]{Remarque}

\newtheorem{notation}[proposition]{Notation}

\newcommand{\bmu}{\ensuremath{\bbmu}}
\newcommand{\noyau}{\ensuremath{\boldsymbol{\varepsilon}}} % l'élément non trivial dans le noyau du revêtement métaplectique
\newcommand{\Css}{\ensuremath{\mathscr{C}_\text{ss}}} % Classes de conjugaison semi-simples
\newcommand{\Cssst}{\ensuremath{\mathscr{C}_\text{ss}^\text{st}}} % Classes de conjugaison stables semi-simples
\newcommand{\Cssgeo}{\ensuremath{\mathscr{C}_\text{ss}^\text{géo}}} % Classes de conjugaison géométriques semi-simples
\newcommand{\CVP}{\ensuremath{\stackrel{\mathrm{VP}}{\longleftrightarrow}}} % Correspondance par valeurs propres
\newcommand{\Lag}{\ensuremath{\mathrm{Lagr}}} % La grassmannienne lagrangienne
\newcommand{\rev}{\ensuremath{\mathbf{p}}} % Le symbole pour revêtements
\newcommand{\asp}{\ensuremath{\dashrule[.7ex]{2 2 2 2}{.4}}} % Le symbole pour anti-spécifiques

\title{Transfert d'intégrales orbitales pour le groupe métaplectique}
\author{Wen-Wei Li}
\date{}
\makeindex

\begin{document}

% Version longue en français
%\begin{abstract}
%  On met en place un formalisme de l'endoscopie pour le groupe métaplectique $\Mp(2n,F)$ sur un corps local $F$, dont les groupes orthogonaux impairs déployés $\SO(2n'+1)\times\SO(2n''+1)$ avec $n'+n''=n$ occupent les rôles des groupes endoscopiques elliptiques. En posant une définition du facteur de transfert, on démontre le transfert d'intégrales orbitales anti-spécifiques de $\Mp(2n,F)$ vers $\SO(2n'+1) \times \SO(2n''+1)$ et le lemme fondamental pour les unités des algèbres de Hecke dans le cas $F$ non archimédien de caractéristique résiduelle assez grande par rapport à $n$. Cela généralise les travaux d'Adams et Renard dans le cas $F=\R$.

%  La démonstration se fait par la méthode de descente de Harish-Chandra. On démontre qu'aux voisinages des éléments semi-simples, on se ramène à la situation de l'endoscopie des groupes classiques et de l'endoscopie non standard sur l'algèbre de Lie. On démontre que le facteur de transfert se descend en les bons facteurs de transfert, cela permet de conclure grâce au travail de Ngô Bao Châu.
%\end{abstract}

\maketitle

\begin{abstract}
  We set up a formalism of endoscopy for metaplectic groups. By defining a suitable transfer factor, we prove an analogue of the Langlands-Shelstad transfer conjecture of orbital integrals over any local field of characteristic zero, as well as the fundamental lemma for units of the Hecke algebra in the unramified case. This generalizes prior works of Adams and Renard in the real case and serves as a first step to study the Arthur-Selberg trace formula for metaplectic groups.
\end{abstract}

\tableofcontents

\section{Introduction}
La formule des traces d'Arthur-Selberg est l'un des outils les plus puissants pour la théorie moderne des formes automorphes. Cette approche est surtout féconde lorsque l'on compare les formules des traces de deux groupes réductifs. Pour ce faire, il faut mettre la formule des traces sous une forme ``stable''. La théorie de l'endoscopie, inventée par Langlands et ses collaborateurs, donne un plan pour résoudre ce problème pour les groupe réductifs.

D'autre part, il existe une famille de revêtements non linéaires $\Mp(2n,F)$ des groupes symplectiques $\Sp(2n,F)$ sur un corps local $F$, qui s'appellent les groupes métaplectiques. À un caractère additif non trivial $\psi: F \to \mathbb{S}^1$ est associée une représentation admissible $\omega_\psi$ de $\Mp(2n,F)$, qui s'appelle la représentation de Weil. Bien que le revêtement métaplectique soit traditionnellement un revêtement à deux feuillets, pour des raisons techniques nous ferons agrandir le revêtement métaplectique de sorte que $\rev: \Mp(2n,F) \to \Sp(2n,F)$ est un revêtement à huit feuillets. Autrement dit, $\Ker(\rev) = \bmu_8 := \{z \in \C^\times: z^8 = 1 \}$. Cela n'affecte pas les résultats que l'on cherche.

Si l'on envisage d'établir et puis de stabiliser la formule des traces pour $\Mp(2n,F)$, le premier pas est d'étudier le transfert local des intégrales orbitales (\ref{prop:transfert}). Cependant, on ne peut pas adapter littéralement la théorie de l'endoscopie car $\Mp(2n,F)$ n'est pas un groupe linéaire algébrique; en particulier il n'a pas de L-groupe. L'un des objets de cet article est de mettre en place un tel formalisme.

Les représentations de $\Mp(2n,F)$ qui nous intéressent sont celles telles que la multiplication par chaque $\noyau \in \Ker(\rev) = \bmu_8$ agit par $\noyau \cdot \identity$ ; ces représentations sont dites spécifiques. Par exemple, la représentation de Weil $\omega_\psi$ est spécifique. Pour l'étude des représentations spécifiques, il suffit de considérer les fonctions telles que $f(\noyau \tilde{x})=\noyau^{-1} f(\tilde{x})$ pour tout $\noyau \in \bmu_8$; ces fonctions sont dites anti-spécifiques. Ces notions se généralisent à tout revêtement. La distinction entre objets spécifiques et anti-spécifiques est superficielle pour $\Mp(2n,F)$ (voir \ref{prop:passage-specifique}).

Notre approche se modèle sur l'endoscopie pour les groupes réductifs. Notons $G := \Sp(2n)$, $\tilde{G} := \Mp(2n,F)$. Tout d'abord il faut trouver des bonnes définitions pour:
\begin{enumerate}
  \item les groupes endoscopiques elliptiques $H$ de $\tilde{G}$,
  \item la correspondance de classes de conjugaison semi-simples entre $H$ et $G$,
  \item une notion de conjugaison stable sur $\tilde{G}$,
  \item le facteur de transfert $\Delta$. 
\end{enumerate}
Une fois que ceci sera fait, on pourra définir l'intégrale orbitale endoscopique
\begin{gather}\label{eqn:integrale-endoscopique-intro}
  J_{H,\tilde{G}}(\gamma, f) = \sum_\delta \Delta(\gamma,\tilde{\delta}) J_{\tilde{G}}(\tilde{\delta}, f)
\end{gather}
d'une fonction anti-spécifique $f$; les notations sont analogues à celles pour l'endoscopie des groupes réductifs et on les expliquera dans \S\ref{sec:conj-transfert}. Par la suite, on peut formuler le transfert de fonctions $f \mapsto f^H$ qui fait concorder $J_{H,\tilde{G}}(\cdot,f)$ et l'intégrale orbitale stable $J_H^\text{st}(\cdot, f^H)$ sur $H$. Comme pour l'endoscopie pour les groupes réductifs, le transfert doit être explicite pour les fonctions sphériques dans le cas non ramifié (\ref{hyp:non-ramifie}). De tels énoncés sont connus sous le nom de ``lemme fondamental''.

Esquissons nos réponses aux questions ci-dessus.
\begin{enumerate}
  \item Soit $F$ une extension finie de $\Q_p$, $p>2$. Selon un résultat de Savin \cite{Sa88}, l'algèbre d'Iwahori-Hecke spécifique (ou anti-spécifique) de $\tilde{G}$ est isomorphe à l'algèbre d'Iwahori-Hecke de $\SO(2n+1)$, le groupe orthogonal impair déployé. Cela suggère que l'on doit regarder $\Sp(2n,\C)$ comme le groupe dual de $\tilde{G}$; de telles évidences existent aussi pour le cas $F=\R$ \cite{ABPTV07}. En poursuivant cette philosophie, on définit une donnée endoscopique elliptique de $\Mp(2n)$ comme une paire $(n',n'') \in \Z^2_{\geq 0}$ telle que $n'+n''=n$; le groupe endoscopique associé est $H_{n',n''} := \SO(2n'+1) \times \SO(2n''+1)$. Contrairement à l'endoscopie pour $\SO(2n+1)$, on distingue les données $(n',n'')$ et $(n'',n')$.

  \item Soit $\gamma=(\gamma',\gamma'') \in H_{n',n''}(F)$ semi-simple ayant valeurs propres
$$ \underbrace{a'_1, \ldots, a'_{n'}, 1, (a'_{n'})^{-1}, \ldots, (a'_{1})^{-1}}_{\text{provenant de } \gamma'}, \underbrace{a''_1, \ldots, a''_{n''}, 1, (a''_{n''})^{-1}, \ldots, (a''_{1})^{-1}}_{\text{provenant de } \gamma''}. $$
On dit que $\delta \in G(F)$ correspond à $\gamma$ s'il est semi-simple avec valeurs propres
$$ a'_1, \ldots, a'_{n'}, (a'_{n'})^{-1}, \ldots, (a'_{1})^{-1}, -a''_1, \ldots, -a''_{n''}, -(a''_{n''})^{-1}, \ldots, -(a''_{1})^{-1}. $$
Cela induit une application entre classes de conjugaison semi-simples géométriques.

  \item Il y a aussi une définition \textit{ad hoc} de stabilité: deux éléments semi-simples réguliers dans $\tilde{G}$ sont stablement conjugués si leurs images dans $G(F)$ sont stablement conjugués et si $\Tr \omega_\psi^+ - \Tr \omega_\psi^-$ prend la même valeur, où $\omega_\psi^\pm$ sont les deux morceaux irréductibles de la représentation de Weil. C'est aussi la voie poursuivie dans \cite{Ad98, Ho07}.

  \item Le facteur de transfert est plus subtil. Lorsque $F=\R$ et $n''=0$, Adams a défini un facteur de transfert $\Delta$ sur l'ensemble des éléments semi-simples réguliers dans $\tilde{G}$ et il est égal à $\Tr \omega_\psi^+ - \Tr \omega_\psi^-$. Plus généralement, pour un groupe endoscopique $H$ quelconque, le facteur de transfert est défini dans cet article comme un produit $\Delta = \Delta'\Delta''\Delta_0$, où $\Delta', \Delta''$ sont fabriqués à partir des caractères $\Tr \omega_\psi^\pm$ et $\Delta_0$ est un terme relativement simple qui est stablement invariant. Le facteur $\Delta_0$ coïncide avec le facteur défini par Renard \cite{Re99}. Il n'y a pas de facteur $\Delta_{IV}$ comme en \cite{LS1}, car nous avons normalisé les intégrales orbitales.
\end{enumerate}

Le facteur de transfert satisfait aux propriétés suivantes.
\begin{itemize}
  \item[\bfseries Spécificité] (\ref{prop:facteur-specificite}): on exige cette propriété de sorte que l'intégrale orbitale endoscopique \eqref{eqn:integrale-endoscopique-intro} est bien définie.
  \item[\bfseries Propriété de cocycle] (\ref{prop:propriete-cocycle}): c'est une condition naturelle pour l'endoscopie, qui affecte des signes aux classes de conjugaison dans une classe de conjugaison semi-simple régulière stable.
  \item[\bfseries Descente parabolique] (\ref{prop:descente-parabolique}): cela réduit le calcul du facteur de transfert aux éléments elliptiques; c'est aussi la principale raison pour laquelle on travaille sur le revêtement à huit feuillets.
  \item[\bfseries Normalisation] (\ref{prop:normalisation-waldspurger}): dans le cas non ramifié, le facteur vaut $1$ pour les éléments à réductions régulières qui se correspondent.
  \item[\bfseries Symétrie] (\ref{prop:symetrie}): qui relie les facteurs de transfert pour $H_{n',n''}$  et $H_{n'',n'}$ . Cette symétrie est réalisée par la multiplication par une image réciproque canonique dans $\tilde{G}$ de $-1 \in G(F)$, ce que l'on désigne encore par $-1$. L'usage d'un tel élément est loisible car on travaille avec le revêtement à huit feuillets.
  \item[\bfseries Formule du produit] (\ref{prop:Delta-forumule-produit}): elle servira à stabiliser les termes elliptiques réguliers dans la formule des traces.
\end{itemize}
Les quatre premières propriétés et la descente semi-simple caractérisent le facteur de transfert dans le cas non ramifié (cf. \cite{Ha93}).

Dans le cas $F=\R$, J. Adams \cite{Ad98} a établi le relèvement de caractères entre $\tilde{G}$ et $\SO(2n+1)$ et D. Renard \cite{Re98} a démontré le transfert d'intégrales orbitales. Pour les groupes endoscopiques $H_{n',n''}$ en général, le transfert d'intégrales orbitales est établi par Renard dans le cas réel; son formalisme paraît différent, mais il est équivalent au nôtre, pour l'essentiel. Pour $F$ non archimédien, $n'=1$ et $n''=0$, J. Schultz a établi le relèvement de caractères entre $\tilde{G}$ et $\SO(3)$ dans sa thèse \cite{Sc98}.

Indiquons brièvement notre approche. À la suite de Langlands, Shelstad \cite{LS2} et Waldspurger, on applique la méthode de descente semi-simple de Harish-Chandra pour réduire le transfert à l'algèbre de Lie. Le revêtement disparaît et on se ramène à des situations composées de l'endoscopie pour les groupes unitaires et symplectiques, ainsi qu'une situation ``non standard'' étudiée dans \cite{Wa08}, à savoir le transfert entre les algèbres de Lie de $\Sp(2n)$ et $\SO(2n+1)$. Grâce aux travaux de Ngô Bao Châu \cite{Ng08}, le transfert est maintenant établi dans chaque situation ci-dessus. Le noyau technique de cet article est donc de prouver que notre facteur $\Delta$ se descend en les bons facteurs aux algèbres de Lie. On demande de plus que les facteurs ainsi descendus soient normalisés dans le cas non ramifié.

Récapitulons la structure de cet article.
\begin{itemize}
  \item Dans \S 2, on recueille les définitions et propriétés de base du groupe métaplectique. L'usage de cocycles est minimaliste. Il y a aussi des discussions de revêtements non linéaires en général.
  \item Dans \S 3, on paramètre explicitement les classes de conjugaison dans les groupes classiques. La classification est bien connue. Remarquons que notre convention diffère que celle de \cite{Wa01} (voir \ref{rem:c-normalisation}). Faute d'avoir une référence complète, on y reproduit toutes les démonstrations.
  \item Dans \S 4, on rappelle les formules du caractère de la représentation de Weil dues à Maktouf \cite{Mak99} en suivant l'approche de T. Thomas \cite{Th08}. Leurs approches reposent sur le modèle de Schrödinger du groupe métaplectique. Ces formules servent aussi à caractériser le scindage au-dessus d'un parabolique de Siegel (\ref{prop:caractere-Levi}). Pour traiter le cas non ramifié, il faudra aussi étudier le caractère via le modèle latticiel.
  \item Dans \S 5, les fondations de l'endoscopie sont mises en place. On établit aussi des résultats utiles pour les articles qui feront suite.
  \item La section \S 6 traite le transfert archimédien. On réconcilie le formalisme de Renard avec le nôtre. On prouve que nos facteurs de transfert coïncident et le transfert archimédien en résulte.
  \item La section \S 7 est consacrée à la descente semi-simple. La descente des termes $\Delta'$, $\Delta''$ repose sur des formules de Maktouf et le calcul de l'indice de Weil de la forme de Cayley (\ref{prop:calcul-qX}). D'autre part, la descente du terme $\Delta_0$ est une manipulation des symboles locaux et des formes quadratiques.
  \item La section \S 8 reprend les arguments pour l'endoscopie des groupes réductifs (cf. \cite{Wa08}); on établit le transfert non archimédien et le lemme fondamental pour les unités par la méthode de descente.
\end{itemize}
Dans \S 6-\S 7, on travaillera avec le revêtement métaplectique à $\mathbf{f}$ feuillets avec $8|\mathbf{f}$. Dans \S 8, on supposera $\mathbf{f}=8$.

Enfin, signalons un autre formalisme de l'endoscopie pour $\Mp(2n)$ proposé par Renard dans \cite{Re99} pour le cas $F=\R$. Grosso modo, les données endoscopiques elliptiques sont toujours en bijection avec les paires $(n',n'') \in \Z_{\geq 0}^2$ telles que $n'+n''=n$, mais les groupes endoscopiques sont $\Mp(2n') \times \Mp(2n'')$. Le facteur de transfert est $\Delta_0$. Nous étudierons une variante de ce formalisme en détail dans \S\ref{sec:trans-arch}, dont les définitions marchent dans le cas non archimédien sans modification. En particulier, on peut parler du transfert d'intégrales orbitales à la Renard.

Remarquons que l'on peut déduire le transfert à l'aide du transfert à la Renard composé avec le transfert de $\Mp(2k)$ vers $\SO(2k+1)$ (associé à la paire $(k,0)$) pour $k = n'$ et $k = n''$. Vu la définition du facteur de transfert $\Delta = \Delta' \Delta'' \Delta_0$, cette approche paraît raisonnable. En effet, le transfert pour $F=\R$ sera démontré de cette façon dans \S\ref{sec:trans-arch}. Réciproquement, si l'on peut montrer que chaque intégrale orbitale stable sur $\SO(2m+1)$ provienne de $\Mp(2m)$ via transfert, alors le transfert à la Renard résulte de notre formalisme. On espère revenir un jour sur cette question.

\paragraph{Remerciements}
J'exprime ma forte gratitude à Jean-Loup Waldspurger pour m'avoir proposé ce sujet et pour ses nombreuses remarques cruciales pendant ce travail. Je remercie également David Renard pour des discussions sur le cas réel. Enfin, je remercie le referee pour d’utiles remarques.

\subsection*{Conventions}
On note $\mathbb{S}^1$ le groupe $\{z \in \C^\times : |z|=1 \}$. Si $\mathbf{f} \in \Z$, on note $\bmu_\mathbf{f}$\index{$\bmu_\mathbf{f}$} le groupe $\{z \in \C^\times : z^\mathbf{f}=1 \}$. Si $A$ est une algèbre centrale semi-simple de dimension finie sur un corps $F$, on note la trace et la norme réduite par $\Tr_{A/F}$ et $N_{A/F}$ respectivement.

\paragraph{Corps locaux}
Soit $F$ un corps local non archimédien, on note $\mathfrak{o}_F$ l'anneau des entiers, $\mathfrak{p}_F$ l'idéal maximal de $\mathfrak{o}_F$ et $\varpi_F$ une uniformisante choisie. On prend toujours la valuation $v$ sur $F$ telle que $v(\varpi_F)=1$. Le conducteur d'un caractère additif\index{$\psi$} $\psi: F \to \mathbb{S}^1$ non-trivial est le plus grand sous $\mathfrak{o}_F$-module $\mathfrak{a}$ de $F$ tel que $\psi|_\mathfrak{a}=1$. Le symbole de Hilbert quadratique pour le corps local $F$ est noté par $(\cdot,\cdot)_F$. Le groupe de Galois absolu de $F$ est noté $\Gamma_F$.

\paragraph{Groupes réductifs}
Soit $F$ un corps et $M$ un $F$-groupe réductif. La composante connexe de $M$ est notée $M^0$. Soit $R$ une $F$-algèbre commutative, on note l'ensemble de $R$-points de $M$ par $M(R)$. Lorsque $M$ est un groupe classique, on confond systématiquement $M(F)$ et $M$.

Supposons que $M$ agit algébriquement sur une $F$-variété $X$. Pour $x \in X(F)$, on note $M^x \subset M$ son fixateur et $M_x := (M^x)^0$. Par exemple, $M$ agit sur lui-même par conjugaison et on obtient ainsi les commutants.

Supposons $M$ connexe et soit $m \in M(F)$. La classe de conjugaison contenant $m$ est notée $\mathcal{O}(m)$. On dira que $m_1, m_2 \in M(F)$ sont géométriquement conjugués s'ils sont conjugués par $M(\bar{F})$. La classe de conjugaison géométrique contenant $m \in M(F)$ est notée par $\mathcal{O}^\text{geo}(m)$. L'ensemble de classes de conjugaison (resp. conjugaison géométrique) semi-simples dans $M$ est noté par $\Css(M)$ (resp. $\Cssgeo(M)$). L'ensemble des éléments semi-simples dans $M(F)$ est noté $M(F)_\text{ss}$. On dira qu'un élément $m \in M(F)_\text{ss}$ est régulier si $M_m$ est un tore, on dira qu'il est fortement régulier si de plus $M^m = M_m$. L'ouvert de Zariski des éléments semi-simples réguliers dans $M$ est noté $M_\text{reg}$.

Soient $m_1, m_2 \in M(F)_\text{ss}$, on dira qu'ils sont stablement conjugués s'il existe $x \in M(\bar{F})$ tel que $x^{-1}m_1 x = m_2$ et $x \sigma(x)^{-1} \in M_{m_1}(\bar{F})$ pour tout $\sigma \in \Gamma_F$. La classe de conjugaison stable contenant $m$ est notée par $\mathcal{O}^\text{st}(m)$; l'ensemble de classes de conjugaison stable semi-simples dans $M$ est notée par $\Cssst(M)$. Si $m$ est fortement régulier, alors $\mathcal{O}^\text{geo}(m)=\mathcal{O}^\text{st}(m)$. On dit qu'une fonction $\phi $ est stablement invariante si $\mathcal{O}^\text{st}(x)=\mathcal{O}^\text{st}(y)$ implique $\phi(x)=\phi(y)$.

\paragraph{Éléments compacts}
Soit $F$ un corps local non archimédien de caractéristique résiduelle $p$. Soient $M$ un $F$-groupe réductif connexe et $\delta \in M(F)$. On dit que $\delta$ est compact si l'ensemble $\delta^\Z$ est d'adhérence compacte dans $M(F)$. On dit que $\delta$ est topologiquement unipotent si $\lim_{n \to \infty} \delta^{p^n} = 1$. Soit $X \in \mathfrak{m}(F)$, notons $T$ le plus grand $F$-tore central dans le commutant de sa partie semi-simple $X_s$. On dit que $X$ est topologiquement nilpotent si $|x^*(X_s)|_F < 1$ pour tout $x^* \in X^*(T)$. Si $p$ est assez grand (voir \cite{Wa08} 4.4 pour une borne explicite), l'exponentielle fournit un homéomorphisme de l'ensemble des éléments topologiquement nilpotents sur l'ensemble des éléments topologiquement unipotents. Tout élément compact $\delta$ admet une décomposition de Jordan topologique\index{décomposition de Jordan topologique} $\delta=\exp(X)\eta=\eta\exp(X)$, où $X$ est topologiquement nilpotent et $\eta$ est d'ordre fini premier à $p$. Cette décomposition est unique, $\eta$ et $\exp(X)$ appartiennent à l'adhérence de $\delta^\Z$. Les détails se trouvent, par exemple, dans \cite{Wa08} 5.2.

\paragraph{Formes quadratiques}
Dans ce texte, on ne considère que les formes quadratiques non  dégénérées. Soient $A$ un anneau commutatif avec $\frac{1}{2}$ et $a_1, \ldots, a_m \in A^\times$, on note $\Qform{a_1, \ldots, a_m}$ la $A$-forme quadratique sur $A^m$ définie par
$$ (x_1,\ldots,x_m) \mapsto a_1 x_1^2 + \ldots a_m x_m^2. $$

On note par $\mathbb{H}$ la forme hyperbolique de rang $2$. On abrège souvent une forme quadratique $(V,q)$ par $q$. Si $F$ est un corps local et $q$ est une $F$-forme quadratique, on note $\det q$ le déterminant de $q$ et $s(q)$ l'invariant de Hasse de $q$. On note la somme orthogonale des formes $q_1$ et $q_2$ par $q_1 \oplus q_2$. Pour un caractère additif non-trivial $\psi: F \to \mathbb{S}^1$ et une $F$-forme quadratique $q$, notons $\gamma_\psi(q)$\index{$\gamma_\psi(\cdot)$} l'indice de Weil défini dans \cite{Weil64}. Il induit un homomorphisme du groupe de Witt $W(F)$ vers $\bmu_8$.

\section{Le groupe métaplectique}

\subsection{Revêtements de groupes réductifs}
Soient $F$ un corps local et $m \in \Z_{\geq 1}$. Soient $M$ un $F$-groupe algébrique et $\rev: \tilde{M} \to M(F)$ une extension centrale de groupes topologiques telle que $\Ker(\rev) \simeq \bmu_m$; on l'appelle un revêtement\index{revêtement} à $m$ feuillets. Il y a une notion naturelle d'équivalence pour de tels revêtements. Le groupe $\tilde{M}$ est localement compact. De plus, il est totalement discontinu si $F$ est non archimédien. Nous fixons toujours une identification $\Ker(\rev)=\bmu_m$.

Supposons $M$ réductif. Soit $P$ un sous-groupe parabolique défini sur $F$ dont $U$ est le radical unipotent. Alors il existe un unique scindage $s: U(F) \to \tilde{M}$ pour $\rev$, qui est invariant par conjugaison par $P(F)$ (\cite{MW94}, appendice A).

\paragraph{Objets spécifiques}
Soit $C_c^\infty(\tilde{M})$ l'algèbre de fonctions lisses à support compact sur $\tilde{M}$, munie du produit de convolution. Il y a une décomposition
$$ C_c^\infty(\tilde{M}) = \bigoplus_{\chi \in \Hom(\bmu_m, \C^\times)} C_{c,\chi}^\infty(\tilde{M}) $$
selon l'action par translation par $\bmu_m$. Idem pour l'espace des fonctions de Schwartz $\mathcal{S}(\tilde{M})$ lorsque $F$ est archimédien.

Soit $\chi \in \Hom(\bmu_m, \C^\times)$. Une fonction dans $C_{c,\chi}^\infty(\tilde{M})$ (resp. $\mathcal{S}_\chi(\tilde{M})$) est dite $\chi$-équivariante\index{$\chi$-équivariante!fonction}. Ceci permet de définir la notion de distributions $\chi$-équivariantes\index{$\chi$-équivariante!distribution}. Une représentation $\pi$ de $\tilde{M}$ est dite $\chi$-équivariante si $\pi|_{\bmu_m}$ est une somme de $\chi$\index{$\chi$-équivariante!représentation}. Le caractère d'une représentation de $\tilde{M}$, pourvu qu'il soit bien défini, est $\chi$-équivariante si et seulement si sa représentation l'est.

Notons $\chi_-: \bmu_m \hookrightarrow \C^\times$ le plongement standard. Pour $\chi = \chi_-$ (resp. $\chi = \chi_-^{-1}$), les objets $\chi$-équivariants sont abrégés comme spécifiques (resp. anti-spécifiques)\index{spécifique!fonction}\index{anti-spécifique!fonction}\index{spécifique!distribution}\index{anti-spécifique!distribution}\index{spécifique!représentation}\index{anti-spécifique!représentation}. Les objets spécifiques (resp. anti-spécifiques) sont affectés de l'indice $-$ (resp. $\asp$). Par exemple, on définit les espaces \index{$C_{c,-}^\infty, C_{c,\asp}^\infty$} $C_{c,-}^\infty(\tilde{M})$ et $C_{c,\asp}^\infty(\tilde{M})$. Lorsque $F$ est archimédien, on définit de la même manière les espaces $\mathcal{S}_{-}(\tilde{M})$ et $\mathcal{S}_{\asp}(\tilde{M})$\index{$\mathcal{S}_{-}, \mathcal{S}_{\asp}$}.

\paragraph{Pousser-en-avant}
Soit $m|m'$. On peut pousser-en-avant l'extension centrale $1 \to \bmu_m \to \tilde{M} \to M(F) \to 1$ via $\bmu_m \hookrightarrow \bmu_{m'}$ et on obtient ainsi un revêtement à $m'$-feuillets $\rev': \tilde{M}' \to M(F)$. On le note aussi par $\tilde{M}' = \tilde{M} \times_{\bmu_m} \bmu_{m'}$. La restriction de $\tilde{M}'$ à $\tilde{M}$ identifie les objets spécifiques sur $\tilde{M}'$ à ceux sur $\tilde{M}$ (par exemple les fonctions, les représentations etc...). De même pour les objets anti-spécifiques.

Tout revêtement que l'on rencontrera dans cet article provient d'un revêtement à deux feuillets. Indiquons un passage entre objets spécifiques et anti-spécifiques dans telles situations.

\begin{proposition}\label{prop:passage-specifique}
  Soient $m \in 2\Z_{\geq 1}$, $\rev: \tilde{M} \to M(F)$ un revêtement à deux feuillets et $\tilde{M}' := \tilde{M} \times_{\bmu_2} \bmu_m$. Alors il existe un caractère continu $\xi: \tilde{M}' \to \bmu_{m/2}$ tel que $\xi([\tilde{m},\noyau])=\noyau^{-2}$.

  L'application $\pi \mapsto \pi \otimes \xi$ est une bijection des représentations spécifiques sur les représentations anti-spécifiques. L'application $f \mapsto f \xi$ induit un isomorphisme d'algèbres $C_{c,-}^\infty(\tilde{M}') \rightiso C_{c,\asp}^\infty(\tilde{M}')$; elle induit un isomorphisme $\mathcal{S}_-(\tilde{M}') \rightiso \mathcal{S}_{\asp}(\tilde{M}')$ si $F$ est archimédien.
\end{proposition}
\begin{proof}
  On vérifie que $\xi$ est bien défini et continu. Le reste en résulte immédiatement.
\end{proof}

\begin{remark}
  On aura parfois besoin de considérer $\tilde{M}' := \tilde{M} \times_{\bmu_m} \C^\times$, qui est une extension de $M(F)$ par $\C^\times$. Les définitions ci-dessus sont pareilles pour $\tilde{M}'$ et on a toujours ladite équivalence entre objets spécifiques/anti-spécifiques.
\end{remark}

Si $F$ est un corps global et $\A$ son anneau d'adèles, alors les terminologies précédentes s'adaptent aux revêtements de $M(k)$ où $k$ est une sous-algèbre de $\A$ munie de la topologie induite.

Nous adoptons systématiquement la convention de désigner un élément dans $\tilde{M}$ par $\tilde{m}$, etc., et sa projection dans $M(F)$ par $m$, etc.

\subsection{La représentation de Weil et le groupe métaplectique local}
Soit $F$ un corps local de caractéristique nulle. On fixe un caractère additif non trivial $\psi: F \to \mathbb{S}^1$.

Soient $n > 0$ et $(W, \angles{\cdot|\cdot})$ un $F$-espace symplectique de dimension $2n$. On supprime souvent la forme $\angles{\cdot|\cdot}$ quand on parle d'un tel espace.

Le groupe de Heisenberg $H(W)$ associé à $(W, \angles{\cdot|\cdot})$ est l'espace $W \times F$ muni du produit
$$ (w,t) \cdot (w', t') = \left(w+w', t+t' + \frac{\angles{w|w'}}{2} \right). $$
Le centre de $H(W)$ est $\{0\} \times F \simeq F$, on l'identifie à $F$.

Notons $\Sp(W)$ le groupe symplectique associé à $(W,\angles{\cdot|\cdot})$. Il agit sur $H(W)$ par
$$ g \cdot (w,t) = (g(w), t). $$

Le théorème de Stone-von Neumann affirme qu'il existe une et une seule représentation lisse irréductible $(\rho_\psi, S_\psi)$ de $H(W)$ à caractère central $\psi$, à isomorphisme près. De plus, une telle représentation est admissible et unitarisable.

Soit $g \in \Sp(W)$. La représentation
$$\rho_\psi^g: h \mapsto \rho_\psi(g \cdot h) $$
vérifie encore les propriétés du théorème de Stone-von Neumann, d'où un opérateur d'entrelacement $M[g]: S_\psi: \to S_\psi$ tel que
$$ M[g] \circ \rho_\psi = \rho_\psi^g \circ M[g]. $$
L'opérateur $M[g]$ est unique à une constante multiplicative près, donc $g \mapsto M[g]$ est un homomorphisme $\Sp(W) \to \PGL(S_\psi)$. Si l'on remplace $(\rho_\psi, S_\psi)$ par sa version unitaire, on peut supposer que $M[g]$ est une isométrie.

Posons\index{$\Spt_\psi(W)$}
$$ \Spt_\psi(W) := \{ (g,M) \in \Sp(W) \times \GL(S_\psi) : M \circ \rho_\psi = \rho_\psi^g \circ M \}.  $$
Cela fournit une extension centrale de $\Sp(W)$ par $\C^\times$ et $\Spt_\psi(W)$ admet une structure naturelle de groupe localement compact (\cite{Weil64}, \S 35). Notons\index{$\MMp{2}(W)$} $\MMp{2}(W)$ le groupe dérivé de $\Spt_\psi(W)$.

Si $F=\C$, le revêtement $\rev: \MMp{2}(W) \to \Sp(W)$ est scindé et on identifie $\MMp{2}(W)$ à $\bmu_2 \times \Sp(W)$; sinon, $p$ est l'unique revêtement non trivial à deux feuillets de $\Sp(W)$. C'est pourquoi nous avons supprimé l'indice $\psi$. La représentation de Weil attachée à $\psi$\index{la représentation de Weil} est la composée $\omega_\psi: \MMp{2}(W) \hookrightarrow \Spt_\psi(W) \to \GL(S_\psi)$. Elle se décompose en deux morceaux non-isomorphes irréductibles, l'un dit pair ($+$) et l'autre impair ($-$) \index{$\omega_\psi^\pm$}:
$$ \omega_\psi = \omega_\psi^+ \oplus \omega_\psi^- , $$
où les représentations $\omega_\psi^\pm$ sont admissibles et unitarisables. Cela permet de définir ses caractères comme distributions spécifiques sur $\MMp{2}(W)$.

Soit $\mathbf{f} \in 2\Z_{\geq 1}$. Posons\index{$\MMp{\mathbf{f}}(W)$}
$$ \MMp{\mathbf{f}}(W) := \MMp{2}(W) \times_{\bmu_2} \bmu_{\mathbf{f}} .$$
Notons le revêtement $\MMp{\mathbf{f}}(W) \to \Sp(W)$ par la même lettre $p$. On obtient ainsi une famille de revêtements indexée par $2\Z_{\geq 1}$ telle que $\MMp{\mathbf{f}}(W) \subset \MMp{\mathbf{f}'}(W)$ si et seulement si $\mathbf{f}|\mathbf{f}'$. Regardons $\omega_\psi^\pm$ comme représentations spécifiques sur les $\MMp{\mathbf{f}}(W)$. Idem pour leurs caractères.

Dans le cas $F=\C$, on a un scindage canonique $\MMp{\mathbf{f}}(W) \simeq \bmu_\mathbf{f} \times \Sp(W)$. Dans le cas trivial $W=\{0\}$, nos définitions entraînent que $\MMp{\mathbf{f}}(W)= \bmu_{\mathbf{f}}$.

\subsection{Sous-groupes et réseaux hyperspéciaux}
Supposons que $F$ est non archimédien de caractéristique résiduelle $p>2$. Soit $L \subset W$ un $\mathfrak{o}_F$-réseau, on définit son réseau dual comme
$$ L^* := \{w \in W : \forall m \in L,\; \angles{w|m} \in \mathfrak{o}_F \}. $$
Un réseau $L$ dans $W$ est dit autodual si $L^*=L$. De tels réseaux existent toujours.

\begin{theorem}
  Si $L$ est un réseau dans $W$ tel que $L^*=L$ ou $L^*=\mathfrak{p}_F L$, alors $K_L := \text{Stab}_{\Sp(W)}(L)$ est un sous-groupe hyperspécial de $\Sp(W)$. Cela fournit une correspondance biunivoque entre les sous-groupes hyperspéciaux et les réseaux $L$ tels que $L^*=L$ ou $L^*=\mathfrak{p}_F L$.
\end{theorem}

Un tel réseau $L$ détermine un modèle de $\Sp(W)$ sur $\mathfrak{o}_F$. On définit le réseau hyperspécial dans $\syp(W)$ associé à $L$ par $\mathfrak{k}_L := \syp(W,\mathfrak{o}_F)$. Si l'on remplace $\angles{\cdot|\cdot}$ par $\varpi_F \angles{\cdot|\cdot}$ alors cela a pour effet d'échanger les réseaux $L$ avec $L^*=L$ et ceux avec $L^* = \varpi_F L$, à homothéthie près; pourtant $\Sp(W)$ ne change pas.

Observons aussi que $\Sp(W)$ agit transitivement sur les réseaux autoduaux. Nous avons fixé une forme symplectique $\angles{\cdot|\cdot}$ sur $W$, cela a l'effet de distinguer une classe de conjugaison canonique de sous-groupes hyperspéciaux de $\Sp(W)$, à savoir ceux associés aux réseaux autoduaux. On ne considère que des tels sous-groupes hyperspéciaux dans cet article. 

\subsection{Modèles de la représentation de Weil}
Nous donnerons deux constructions pour la représentation $(\rho_\psi, S_\psi)$ et les opérateurs d'entrelacement $M[g]$ dans les définitions précédentes. Commençons par une construction générale.

Soit $A \subset W$ un sous-groupe tel que $A_F := A \times F$ est un sous-groupe abélien maximal dans $H(W)$; ceci est équivalent à $A = A^\perp$ où $A^\perp := \{w \in W : \forall a \in A, \; \psi(\angles{a,w})=1 \}$. Puisque $A_F/(\{0\} \times \Ker(\psi))$ est abélien, il existe un caractère $\psi_A: A_F \to \mathbb{S}^1$ qui prolonge $1 \times \psi$ sur $\{0\} \times F$. On définit
$$ (\rho_A, S_A) := \text{Ind}_{A_F}^{H(W)}(\psi_A) .$$
\begin{theorem}
  $(\rho_A, S_A)$ est une représentation lisse irréductible à caractère central $\psi$.
\end{theorem}

\subsubsection{Le modèle de Schrödinger}\index{modèle de Schrödinger}
Les détails se trouvent dans \cite{Th06, Th08}.

Un sous-espace $\ell \subset W$ est dit un lagrangien\index{lagrangien} dans $W$ si $\ell$ est totalement isotrope de dimension maximale pour $\angles{\cdot|\cdot}$. Notons\index{$\Lag(W)$} $\Lag(W)$ l'ensemble des lagrangiens dans $W$.

Dans la construction ci-dessus, prenons pour $A=\ell \in \Lag(W)$. Dans ce cas-là $\ell_F = \ell \times F$ comme groupes topologiques, et on peut prendre $\psi_\ell = 1 \times \psi$. Soit $(\rho_\ell, S_\ell)$ la représentation ainsi obtenue. On fixe une mesure de Haar sur $\ell$ et on prend la mesure autoduale sur $W$ par rapport à $\psi(\angles{\cdot|\cdot})$. La représentation $(\rho_\ell, S_\ell)$ s'identifie à l'espace des vecteurs lisses de l'induite compacte $\text{ind}_{\ell_F}^{H(W)}(\psi_\ell)$.

Soit $\ell'$ un autre lagrangien muni d'une mesure de Haar. P. Perrin a défini un opérateur d'entrelacement canonique
$$ \mathscr{F}_{\ell', \ell}: S_\ell \to S_{\ell'} .$$
C'est essentiellement une transformation de Fourier partielle convenablement normalisée.

Soit $g \in \Sp(W)$, alors on obtient une isométrie par transport de structures:
$$ g_*: S_\ell \to S_{g\ell}. $$

Définissons\index{$M_\ell[\cdot]$}
$$ M_\ell[g] := \mathscr{F}_{\ell, g\ell} \circ g_*  = g_* \circ \mathscr{F}_{g^{-1}\ell, \ell}. $$

On vérifie que $(g, M_\ell[g]) \in \Spt_\psi(W)$. Grosso modo, la représentation $\tilde{\omega}_\psi$ sur $\Spt_\psi(W)$ ne peut pas être définie sur $\Sp(W)$ car $\mathscr{F}_{\ell'', \ell'} \circ \mathscr{F}_{\ell', \ell}$ n'est pas égale à $\mathscr{F}_{\ell'', \ell}$, mais diffère par un nombre complexe de module $1$ ($\ell, \ell', \ell'' \in \Lag(W)$ quelconques). Pour expliciter  cette obstruction, récapitulons des propriétés de l'indice de Maslov telle qu'elles sont énoncées  par T. Thomas \cite{Th06}.

Étant donnés $\ell_1, \ldots, \ell_m \in \Lag(W)$ ($m \geq 3$), On définit une $F$-forme quadratique $\tau(\ell_1, \ldots, \ell_m)$. Elle s'appelle l'indice de Maslov\index{indice de Maslov}. Soit $[\tau(\ell_1, \ldots, \ell_m)]$ sa classe dans le groupe de Witt $W(F)$; cette classe satisfait aux propriétés suivantes.
\begin{enumerate}
  \item \textbf{Invariance symplectique}. Pour tout $g \in \Sp(W)$,
    $$ \tau(\ell_1, \ldots, \ell_m) \simeq \tau(g\ell_1, \ldots, g\ell_m). $$
  \item \textbf{Additivité symplectique}. Soient $W_1, W_2$ deux $F$-espaces quadratiques et $W := W_1 \oplus W_2$. Si $\ell_1, \ldots, \ell_m \in \Lag(W_1)$ et $\ell'_1, \ldots, \ell'_m \in \Lag(W_2)$, alors
    $$ \tau(\ell_1 \oplus \ell'_1, \ldots, \ell_m \oplus \ell'_m) \simeq \tau(\ell_1, \ldots, \ell_m) \oplus \tau(\ell'_1, \ldots, \ell'_m). $$
  \item \textbf{Symétrie diedrale}.
    \begin{gather*}
      \tau(\ell_1, \ldots, \ell_m) \simeq  \tau(\ell_2, \ldots, \ell_m, \ell_1), \\
      [\tau(\ell_1, \ldots, \ell_m)] = -[\tau(\ell_m, \ldots, \ell_1)].
    \end{gather*}
  \item \textbf{Condition de chaîne}. Pour tout $3 \leq k < m$, on a
    $$ [\tau(\ell_1, \ldots, \ell_m)] = [\tau(\ell_1, \ldots, \ell_k)] + [\tau(\ell_1, \ell_k, \ldots, \ell_m)]. $$
\end{enumerate}

Dans le cas $m=3$, l'espace $\tau(\ell_1, \ell_2, \ell_3)$ est Witt équivalent à l'indice de Maslov défini par Kashiwara \cite{LV80}. Vu la condition de chaîne, cela détermine $[\tau(\ell_1, \ldots, \ell_m)] \in W(F)$ pour $m \geq 3$ quelconque. La dimension de $\tau$ est aussi calculée:

\begin{proposition}[\cite{Th06}]\label{prop:Maslov-dimension}\index{indice de Maslov!formule de dimension}
  Regardons les lagrangiens $\ell_1, \ldots, \ell_m$ comme indexés par $\Z/m\Z$. Alors:
  $$ \dim \tau(\ell_1, \ldots, \ell_m) = \frac{(m-2)\dim W}{2} - \sum_{i \in \Z/m\Z} \dim(\ell_i \cap \ell_{i+1}) + 2 \dim \bigcap_{i \in \Z/m\Z} \ell_i.$$
\end{proposition}

\begin{theorem}[G. Lion, P. Perrin]
  Soient $\ell_1, \ldots, \ell_m \in \Lag(W)$ ($m \geq 3$). Alors
  $$ \mathscr{F}_{\ell_1, \ell_m} \circ \cdots \circ \mathscr{F}_{\ell_2, \ell_1} = \gamma_\psi(-\tau(\ell_1,\ldots,\ell_m)) \cdot \identity_{S_{\ell_1}}. $$
\end{theorem}

\begin{corollary}
  Soit $\ell \in \Lag(W)$, alors pour tout $x, y \in \Sp(W)$ on a
  $$ M_\ell[x] \cdot M_\ell[y] = \gamma_\psi(\tau(\ell,y\ell,xy\ell)) M_\ell[xy]. $$
\end{corollary}

On en déduit un scindage au-dessus d'un sous-groupe parabolique de Siegel.

\begin{proposition}\label{prop:scindage-Schrodinger-Spt}\index{$\sigma_\ell(\cdot)$}
  Soit $\ell \in \Lag(W)$. Notons $P_\ell$ le sous-groupe de $\Sp(W)$ qui stabilise $\ell$, alors $\sigma_\ell: x \mapsto (x,M_\ell[x])$ fournit un scindage de $\rev: \Spt_\psi(W) \to \Sp(W)$ au-dessus de $P_\ell(F)$.
\end{proposition}
\begin{proof}
  Vu la définition de la topologie sur $\Spt_\psi(W)$ (\cite{Weil64}, \S 35), la continuité de $\sigma_\ell$ est immédiate. Il suffit que $\tau(\ell,x\ell,xx'\ell)=0$ pour tous $x,x' \in P_\ell(F)$. Or dans ce cas $\tau(\ell,x\ell,xx'\ell)=\tau(\ell,\ell,\ell)$, et \ref{prop:Maslov-dimension} montre que sa dimension est zéro.
\end{proof}

\begin{definition}\label{def:-1}\index{-1}
  L'élément $\sigma_\ell(-1)$ dans $\Spt_\psi(W)$ est central d'ordre $2$. Il s'envoie sur $-1$ dans $\Sp(W)$. On le note abusivement par $-1$. Cette convention sera justifiée par le fait qu'elle ne dépend pas de $\ell$ (\ref{prop:-1-action}) et qu'elle est compatible avec tout scindage de $p$ dont nous ferons usage (\ref{prop:-1-relevement}, \;\ref{prop:-1-latticiel}).
\end{definition}

On abrège souvent $(-1) \cdot 
\tilde{x}$ par $-\tilde{x}$, pour tout $\tilde{x} \in \Spt_\psi(W)$. L'indépendance du choix de $\ell$ résulte de la proposition suivante.

\begin{proposition}\label{prop:-1-action}
  L'élément $-1 \in \Spt_\psi(W)$ agit par $\pm \identity$ sur $S_\psi^\pm$.
\end{proposition}
\begin{proof}
  Fixons $\ell \in \Lag(W)$. On se ramène à prouver que $M_\ell[-1]$ agit par $\pm \identity$ sur $S_\psi^\pm$, ce qui résulte des formules explicites dans \cite{MVW87} chapitre 2, II.6.
\end{proof}

Il sera démontré que $-1$ vit dans le revêtement $\MMp{\mathbf{f}}(W)$ pourvu que $8|\mathbf{f}$ (\ref{prop:scindage-Schrodinger}). On obtiendra aussi une caractérisation de $-1$ par la valeur du caractère de la représentation de Weil.

\subsubsection{Le modèle latticiel}\label{sec:modele-latticiel}\index{modèle latticiel}
Dans cette sous-section, $F$ est supposé non archimédien de caractéristique résiduelle $p>2$. Fixons un $\mathfrak{o}_F$-réseau $L \subset W$ tel que $L=L^\perp$ et prenons $A=L$ dans la construction générale de $(\rho_A, S_A)$. De tels réseaux existent toujours.

Prenons $L_F = L \times F$. Puisque $p>2$, $L_F$ est un sous-groupe de $H(W)$ et on peut prendre $\psi_L(a,t) = \psi(t)$. D'où une représentation lisse $(\rho_L, S_L)$ de $H(W)$.

On choisit un système de représentants $R \subset W$ de l'espace discret $W/L$. Pour tout $r \in R$, définissons une fonction localement constante à support compact $f_r: H(W) \to \C$ par
$$
  f_r((r'+a,t)) = \begin{cases} \psi\left(t+\frac{\angles{r'|a}}{2}\right), & \text{si } r'=r \\ 0, & \text{sinon} \end{cases}
$$
où $r' \in R$ et $a \in L$. Ces fonctions forment une base de $S_L$. Si l'on munit $S_L$ du produit hermitien $(f|g) := \sum_{\dot{w} \in W/L} f(w,0)\overline{g(w,0)}$ en rappelant que $W/L$ est discret, alors $\{f_r\}_{r \in R}$ est une base orthonormée.

\begin{proposition}\index{$M_L[\cdot]$}
  Posons $K := \text{Stab}_{\Sp(W)}(L)$. Pour $x \in K$, soit $M_L[x]: S_L \to S_L$ l'opérateur unitaire $g(\cdot) \mapsto g(x^{-1}(\cdot))$. Alors $x \mapsto (x, M_L[x])$ est un homomorphisme injectif continu de $K$ dans $\Spt_\psi(W)$. Ceci fournit un scindage de $\rev: \MMp{2}(W) \to \Sp(W)$ au-dessus de $K$.
\end{proposition}
\begin{proof}
  Pour la deuxième assertion, voir \cite{MVW87}, Chapitre 2, II.10.
\end{proof}

Le sous-groupe ouvert compact $K$ est toujours hyperspécial. Nous identifions désormais $K$ comme un sous-groupe ouvert compact de $\MMp{2}(W)$.

\begin{proposition}
  Soient $\{f_r\}_{r \in R}$ les fonctions définies précédemment. Pour $r, r' \in R$ et $x \in K$, on a
  $$ (M_L[x]f_r)(r', 0) = \begin{cases} \psi\left(\dfrac{\angles{r|x^{-1}(r')-r}}{2} \right), & \text{si } x^{-1}(r')-r \in L \\ 0, & \text{sinon}. \end{cases}$$
  Autrement dit, si $r' \in R$ représente la classe $x^{-1}(r)$ mod $L$, alors on a
  $$ (M_L[x]f_r) = \psi\left(\frac{\angles{r|x^{-1}(r')}}{2} \right) f_{r'}. $$
\end{proposition}

\begin{remark}
  Lorsque $\psi$ est de conducteur $\mathfrak{o}_F$, on a $M^\perp = M^*$ pour tout réseau $M \subset W$; en particulier, $L$ est autodual.
\end{remark}

\begin{remark}\label{rem:H(L)-invariant}
  Supposons $\psi$ de conducteur $\mathfrak{o}_F$. Posons
  $$ H(L) := L \times \mathfrak{o}_F. $$
  C'est un sous-groupe ouvert compact de $H(W)$. Soit $(\rho_\psi, S_\psi)$ une représentation satisfaisant aux énoncés du théorème de Stone-von Neumann. En utilisant le modèle latticiel associé a un réseau autodual $L$, on montre que
  $$ \dim_\C S_\psi^{H(L)} = 1. $$
  En effet, cet espace engendré par la fonction caractéristique de $L$. Soit $s_L \in S_\psi^{H(L)}$, $s_L \neq 0$, alors $M_L[x]$ est caractérisé par $M_L[x](s_L)=s_L$.
\end{remark}

  Montrons une compatibilité entre le modèle de Schrödinger et le modèle latticiel qui sera utile. Supposons qu'il existe $\ell,\ell' \in \Lag(W)$ tels que $W=\ell\oplus\ell'$ et $L=(\ell\cap L) \oplus (\ell'\cap L)$. Soit $x \in \Sp(W)$ tel que $x\ell=\ell$ et $x\ell'=\ell'$.

\begin{proposition}\label{prop:Levi-latticiel-compatible}
  Conservons les hypothèses ci-dessus. Soit $T$ un $F$-tore maximal déployé dans $P_\ell \cap P_{\ell'}$, alors $M_\ell[x] = M_L[x]$ pour $x \in T(F) \cap K$.
\end{proposition}
\begin{proof}
  Du point de vue du modèle de Schrödinger, $S_\psi^{H(L)}$ est engendré $s_L := \mathbbm{1}_{L \cap \ell'}$ (\cite{MVW87} chapitre 2, II. 10). On a $M_L[x](s_L)=s_L$ car $x \in K$. D'après la formule explicite de $M_\ell[x]$ (\cite{MVW87} chapitre 2, II. 6), on a aussi $M_\ell[x](s_L)=s_L$, d'où l'assertion.
\end{proof}

\subsubsection{La construction de Lion-Perrin}\index{construction de Lion-Perrin}

Le modèle de Schrödinger conduit à une construction du revêtement à deux feuillets $\rev: \MMp{2}(W) \to \Sp(W)$, étudiée systématiquement par Lion et Perrin \cite{LP81}. 

Pour $V$ un $F$-espace vectoriel de dimension finie, posons
$$ \mathfrak{o}(V) := (\bigwedge^{\text{max}} V \setminus \{0\})/F^{\times 2}. $$

C'est un torseur sous $F^\times/F^{\times 2}$; ici nous adoptons la convention $\bigwedge^0 \{0\} = F$ de sorte que $\mathfrak{o}(\{0\}) = F^\times/F^{\times 2}$. Un élément dans $\mathfrak{o}(V)$ est dit une orientation de $V$.

Soient $\ell_1, \ell_2 \in \Lag(W)$ et $e_i \in \mathfrak{o}(\ell_i)$ ($i=1,2$). Les $F$-espaces vectoriels $\ell_1/(\ell_1 \cap \ell_2)$ et $\ell_2/(\ell_1 \cap \ell_2)$ sont en dualité par rapport à $\angles{\cdot|\cdot}$, d'où l'accouplement
$$\angles{\cdot|\cdot}: \mathfrak{o}(\ell_1/(\ell_1 \cap \ell_2)) \times \mathfrak{o}(\ell_2/(\ell_1 \cap \ell_2)) \to F^\times/F^{\times 2}.$$

Fixons une orientation $e \in \mathfrak{o}(\ell_1 \cap \ell_2)$ et choisissons $\bar{e}_i \in \mathfrak{o}(\ell_i/(\ell_1 \cap \ell_2))$ de sorte que $\bar{e}_i \wedge e = e_i$ ($i=1,2$). Définissons
$$ A_{\ell_1, \ell_2} := \angles{\bar{e}_1|\bar{e}_2} \in F^\times/F^{\times 2}. $$
C'est indépendant du choix de $e$.

Fixons maintenant $\ell \in \Lag(W)$ et $e \in \mathfrak{o}(\ell)$. Pour tout $g \in \Sp(W)$, munissons $g\ell$ de l'orientation transportée. On définit\index{$m_g(\cdot)$}
$$ m_g(\ell) := \gamma_\psi(1)^{\frac{\dim W}{2} - \dim g\ell\cap\ell - 1} \gamma_\psi(A_{g\ell,\ell}).$$
Cela ne dépend pas du choix de l'orientation $e$. On construit $\MMp{2}(W)$ comme l'ensemble des
$$ (g,t), \qquad g \in \Sp(W), \; t: \Lag(W) \to \C^\times $$
tels que
\begin{itemize}
  \item $t(\ell)^2 = m_g(\ell)^2$ pour tout $\ell$;
  \item $t(\ell') = \gamma_\psi(\tau(\ell, g\ell, g\ell', \ell')) t(\ell)$ pour tout $\ell, \ell'$.
\end{itemize}

La multiplication dans $\MMp{2}(W)$ est définie par $(g,t)(g',t') = (gg', tt' \cdot c_{g,g'})$ où
$$ c_{g,g'}(\ell) = \gamma_\psi(\tau(\ell, g\ell, gg' \ell)) . $$

On définit $\rev: \MMp{2}(W) \to \Sp(W)$ par $\rev(g,t) = g$. En utilisant les propriétés de l'indice de Maslov, on vérifie que la multiplication est bien définie et associative, et que $\rev$ est un revêtement à deux feuillets. L'élément neutre est $(1, \mathbbm{1})$ où $\mathbbm{1}$ est la fonction constante de valeur $1$.

Soit $\ell \in \Lag(W)$, on définit la fonction d'évaluation $\mathrm{ev}_\ell$ par\index{$\mathrm{ev}(\cdot)$}
\begin{equation}\label{eqn:ev-def}
  \forall (g,t) \in \MMp{2}(W), \quad \mathrm{ev}_\ell(g,t) = t(\ell).
\end{equation}

Étant fixé un lagrangien $\ell \in \Lag(W)$, on a un plongement de $\MMp{2}(W)$ dans $\Spt_\psi(W)$ par le modèle de Schrödinger:
$$ (g,t) \mapsto (g, t(\ell) M_\ell[g]). $$
L'image ne dépend pas du choix de $\ell$ et cela identifie $\MMp{2}(W)$ au sous-groupe dérivé de $\Spt_\psi(W)$.

Notons $P=P_\ell$ le stabilisateur de $\ell$; c'est le sous-groupe parabolique de Siegel associé à $\ell$. Le résultat suivant affirme que le revêtement $\MMp{8}(W)$ est suffisamment grand pour réaliser le modèle de Schrödinger.

\begin{proposition}\label{prop:scindage-Schrodinger}
  Le scindage $\sigma_\ell: P(F) \to \Spt_\psi(W)$ dans \ref{prop:scindage-Schrodinger-Spt} est à valeurs dans $\MMp{8}(W)$. En particulier, l'élément $-1$ dans \ref{def:-1} appartient à $\MMp{8}(W)$.
\end{proposition}
\begin{proof}
  On a $\sigma_\ell(g)=(g, M_\ell[g])$. Il existe $t: \Lag(W) \to \C^\times$ tel que $(g,t) \in \MMp{2}(W)$. On voit que $\sigma_\ell(g)$ et l'image de $(g,t)$ dans $\Spt_\psi(W)$ diffèrent par $t(\ell) \in \C^\times$. D'après la construction de Perrin-Lion, $t(\ell)$ est un produit des indices de Weil $\gamma_\psi(\cdot)$, qui appartiennent à $\bmu_8$. D'où l'assertion.
\end{proof}

\begin{remark}
  Soit $P_\ell=MU$ une décomposition de Lévi, alors le scindage est unique sur $U(F)$ (\cite{MVW87}, appendice 1). Nous donnerons une caractérisation de $\sigma_\ell|_M$ en termes du caractère de $\omega_\psi$ (\ref{rem:caractere-Levi}).
\end{remark}

\subsection{Le cas global}\label{sec:cas-global}
Dans cette section $F$ est un corps de nombres. Notons $\A$ l'anneau d'adèles associé à $F$.

Fixons un caractère non-trivial $\psi: \A/F \to \mathbb{S}^1$, regardé aussi comme un caractère de $\A$ avec décomposition en composantes locales
$$ \psi = \bigotimes_v \psi_v. $$

Fixons un $F$-espace symplectique $(W, \angles{\cdot|\cdot})$ défini sur $\mathfrak{o}_F$, notons $L$ l'ensemble de $\mathfrak{o}_F$-points de $W$, c'est un $\mathfrak{o}_F$-réseau dans $W$. Pour toute place $v$ de $F$, on construit les objets suivants:
\begin{align*}
  (W_v, \angles{\cdot|\cdot}) & := (W, \angles{\cdot|\cdot}) \otimes_F F_v ; \\
  H(W_v): & \; \text{le groupe de Heisenberg}; \\
  (\rho_v, S_v): & \; \text{représentation irréductible lisse à caractère central } \psi_v; \\
  \Spt_\psi(W_v), \MMp{\mathbf{f}}(W_v), \rev_v: & \; \text{les revêtements} \quad (\mathbf{f} \in 2\Z_{\geq 1}); \\
  \omega_{\psi_v}: & \; \text{la représentation de Weil}.
\end{align*}

Pour presque toute place finie $v$, $\psi_v$ est de conducteur $\mathfrak{o}_v$ et le complété $L_v$ de $L$ est autodual. Pour une telle $v$, posons $K_v := \text{Stab}(L_v)$ et $s_v \in S_v$ le vecteur correspondant à la fonction caractéristique de $L_v$ pour le modèle latticiel (cf. \ref{rem:H(L)-invariant}).

On définit $(\rho_\psi, S_\psi) = \bigotimes'_v (\rho_v, S_v)$, produit restreint par rapport aux vecteurs $s_v$. On définit le groupe de Heisenberg adélique $H(W, \A)$ par rapport à $L$, alors $(\rho_\psi, S_\psi)$ est une représentation de $H(W,\A)$ à caractère central $\psi$. Le groupe $\Sp(W,\A)$ agit sur $H(W,\A)$ de façon naturelle.

On sait construire le produit restreint $\Resprod_v \MMp{\mathbf{f}}(W_v)$ par rapport aux $K_v$. Posons\index{$\MMp{\mathbf{f}}(W,\A)$}
\begin{align*}
  \mathbf{N} & := \left\{ (\noyau_v) \in \bigoplus_v \Ker(\rev_v) = \bigoplus_v \bmu_\mathbf{f} : \prod \noyau_v = 1 \right\}, \\
  \MMp{\mathbf{f}}(W,\A) & := \Resprod_v \MMp{\mathbf{f}}(W_v)/\mathbf{N}.
\end{align*}
On écrit une classe dans $\MMp{\mathbf{f}}(W,\A)$ comme $[\tilde{x}_v]_v$, où $\tilde{x}_v \in \MMp{\mathbf{f}}(W_v)$ pour toute $v$.

Les projections locales $(\rev_v)_v$ fournissent un revêtement $\rev: \MMp{\mathbf{f}}(W,\A) \to \Sp(W,\A)$ et on a $\Ker(\rev) = \bmu_\mathbf{f}$. On définit la représentation de Weil adélique par $\omega_\psi := \bigotimes'_v \omega_{\psi_v}$; elle est spécifique. On retrouve les avatars locaux $\rev_v: \MMp{\mathbf{f}}(W_v) \to \Sp(W_v)$ comme les fibres de $\rev$ au-dessus des $\Sp(W_v)$.

Par ailleurs, on peut formuler une variante du théorème de Stone-von Neumann pour $H(W,\A)$, ce qui permet de définir
$$ \Spt_\psi(W, \A) := \{(g,M) \in \Sp(W,\A) \times \GL(S_\psi) : \rho_\psi^g \circ M = M \circ \rho_\psi \} $$
comme dans le cas local. C'est une extension centrale de $\Sp(W,\A)$ par $\C^\times$. On définit $\MMp{2}(W,\A)$ comme le groupe dérivé de $\Spt_\psi(W,\A)$. La représentation de Weil $\omega_\psi$ provient de la projection sur $\GL(S_\psi)$, et on a
$$ \Spt_\psi(W,\A) = \MMp{2}(W,\A) \times_{\bmu_2} \C^\times .$$
On définit les $\MMp{\mathbf{f}}(W,\A)$ en posant $\MMp{\mathbf{f}}(W,\A) := \MMp{2}(W,\A) \times_{\bmu_2} \bmu_\mathbf{f}$.

\begin{remark}
  Les constructions dans cette section ne dépendent pas du choix de $L$. En effet, si $L$, $L'$ sont deux tels réseaux, alors $L_v = L'_v$ pour presque toute place finie $v$.
\end{remark}

Pour formuler la théorie des représentations automorphes, il faudra un scindage canonique de $p$ au-dessus de $\Sp(W)$. L'existence d'un tel scindage est dû à Weil. C'est unique et à image dans $\MMp{2}(W,\A)$ car $\Sp(W)$ est engendré par ses commutateurs. Donnons une construction explicite. Fixons $\ell \in \Lag(W)$, on construit les opérateurs $M_{\ell_v}[x]$ pour tout $x \in \Sp(W)$ et toute place $v$ de $F$ où $\ell_v := \ell \otimes_F F_v$.

\begin{proposition} \label{prop:relevement-rationnel-explicite}
  L'application
  \begin{align*}
    i: \Sp(W) & \to \Spt_\psi(W,\A) \\
    x & \mapsto (x, \bigotimes_v M_{\ell_v}[x])
  \end{align*}
  est un homomorphisme bien définie.
\end{proposition}
\begin{proof}
  Prenons $\ell' \in \Lag(W)$ tel que $W = \ell \oplus \ell'$. Pour presque toute place finie $v$, on a $L_v = L_v \cap \ell_v \oplus L_v \cap \ell'_v$. D'après les formules explicites pour $M_{\ell_v}[x]$ et $s_v$ (\cite{MVW87} chapitre 2, II.6, II.10) avec la décomposition de Bruhat, $M_{\ell_v}[x]$ fixe $s_v$ pour presque tout $v$, donc $\bigotimes_v M_{\ell_v}[x] \in \GL(S_\psi)$ et $i$ est bien défini.

  Montrons que $i$ est un homomorphisme. Soient $x,y \in \Sp(W)$, alors
  $$ M_{\ell_v}[x] M_{\ell_v}[y] = \gamma_{\psi_v}(\tau(\ell, x\ell, xy\ell)) M_{\ell_v}[xy]. $$
  L'espace quadratique $\tau(\ell, x\ell, xy\ell)$ est défini sur $F$. Grâce à la réciprocité de Weil (\cite{Weil64} \S 30, Proposition 5), le produit $\prod_v \gamma_{\psi_v}(\tau(\ell, x\ell, xy\ell))$ vaut $1$, cela permet de conclure.
\end{proof}

\begin{corollary}\label{prop:-1-relevement}
  Pour toute place $v$, soit $-1_v \in \MMp{8}(W_v)$ l'élément considéré dans \ref{def:-1}. Alors $i(-1)$ est égal à $[-1_v]_v$.
\end{corollary}
\begin{proof}
  Si l'on plonge $\MMp{8}(W,\A)$ dans $\Spt_\psi(W)$, alors $[-1_v]_v$ s'identifie à
  $$(-1, \bigotimes_v M_{\ell_v}[-1]), $$
  qui est égal à $i(-1)$ par ladite proposition.
\end{proof}

\section{Classes de conjugaison semi-simples dans les groupes classiques}

\subsection{Formes hermitiennes}
Dans cette section, $F$ est un corps parfait de caractéristique $\neq 2$.

\paragraph{Formes pour les anneaux à involutions}
Fixons $A$ une $F$-algèbre et $\tau$ une $F$-involution de $A$ (i.e. un antiautomorphisme de carré l'identité). On dit qu'une telle paire $(A,\tau)$ est une $F$-algèbre à involution\index{algèbre étale à involution}. Lorsque $A$ est étale, on supprime souvent l'involution et on exprime l'algèbre à involution par $A/A^\#$, où $A^\#$ est la sous-algèbre fixée par $\tau$. C'est loisible car les données $A,A^\#$ déterminent $\tau$.

Soit $M$ un $A$-module projectif à droite de type fini. Lorsque $A$ est commutatif, les modules à gauche et à droite se confondent. Une application bi $F$-additive $q: M \times M \to A$ est dite une forme sesquilinéaire si pour tout $m,n \in M$ et $a,b \in A$, on a
$$ q(ma|nb) = \tau(a)q(m|n)b. $$

Une telle application $q$ équivaut à un homomorphisme
\begin{align*}
  g_q: & M \to M^* := \Hom_A(M,A) \\
  & m \mapsto q(m|-)
\end{align*}
où on regarde le dual $M^*$ comme un $A$-module à droite par $(fa)(m) = \tau(a)f(m)$ pour tout $a \in A$, $m \in M$.

Une forme sesquilinéaire $q$ est dite non-dégénérée si $g_q$ est un isomorphisme. Soit $\epsilon=\pm 1$, on dit que $q$ est $\epsilon$-hermitienne si $q$ est non-dégénérée et $q(m|n)=\epsilon \tau(q(n|m))$ pour tout $m,n \in M$; cela équivaut à la commutativité du diagramme
$$\xymatrix{
  M \ar[r]^{g_q} \ar[d]_{\epsilon \cdot \varpi} & M^* \ar@{=}[d] \\
  M^{**} \ar[r]^{g_q^*} & M^*
}$$
où $\varpi: M \to M^{**}$ est donnée par $\varpi(m)(\lambda) = \angles{\lambda,m}$ pour tous $m \in M$,$\lambda \in M^*$. Il y a une notion naturelle d'isométries entre ces formes. Notons la catégorie des $(A,\tau)$-formes $\epsilon$-hermitienne par $\mathfrak{Herm}^\epsilon(A,\tau)$.

Voici quelques cas spéciaux que nous utiliserons plus tard.
\begin{itemize}
  \item $A=F$, $\tau=\identity$, $\epsilon=1$: les $F$-formes quadratiques.
  \item Idem, mais $\epsilon=-1$: les $F$-formes symplectiques.
  \item $A=E$ une extension quadratique de $F$, $\tau$ l'involution associée, $\epsilon=1$: les $E/F$-formes hermitiennes.
  \item Idem, mais $\epsilon=-1$: les $E/F$-formes anti-hermitiennes.
\end{itemize}

\paragraph{Pousser-en-avant des formes}
Supposons donnés une inclusion $(A_1,\tau_1) \hookrightarrow (A_2, \tau_2)$ et un homomorphisme $F$-linéaire $t: A_2 \to A_1$ tel que $t \circ \tau_2 = \tau_1 \circ t$. Soit $q$ une $(A_2,\tau_2)$-forme $\epsilon$-hermitienne, on en déduit une $(A_1, \tau_1)$-forme $\epsilon$-hermitienne $t_* q$ sur $M$ (regardé comme un $A_1$-module) définie par
$$ (t_*q)(m|n) =  t(q(m|n)). $$

\paragraph{Formes en catégories}
Afin de classifier les classes de conjugaison, travaillons dans un cadre plus abstrait. Une référence possible est \cite{Knus91} II.
\begin{definition}[cf. \cite{Knus91} II 2] % ou \cite{Ba05}...
  Une catégorie $F$-additive avec dualité est un triplet $(\mathcal{C}, *, \varpi)$ où $\mathcal{C}$ est une catégorie $F$-additive, $*$ est un foncteur $\mathcal{C}^\text{op} \to \mathcal{C}$ et $\varpi: \identity_{\mathcal{C}} \rightiso **$ est un isomorphisme de foncteurs.
  
  Une forme $\epsilon$-hermitienne dans $(\mathcal{C},*,\varpi)$ est une paire $(M,\phi)$ où $M$ est un objet dans $\mathcal{C}$ et $\phi$ est un isomorphisme $M \rightiso M^*$ tel que le diagramme
  $$\xymatrix{
    M \ar[r]^{\phi} \ar[d]_{\epsilon\varpi_M} & M^* \ar@{=}[d] \\
    M^{**} \ar[r]^{\phi^*} & M^*
  }$$
  est commutatif. On dit aussi qu'une forme $\epsilon$-hermitienne est hermitienne si $\epsilon=1$ et anti-hermitienne si $\epsilon=-1$. On note la catégorie des formes $\epsilon$-hermitienne dans $\mathcal{C}$ par $\mathfrak{Herm}^\epsilon(\mathcal{C})$.

  Soient $(\mathcal{C}_1, *_1, \varpi_1)$, $(\mathcal{C}_2, *_2, \varpi_2)$ deux catégories $F$-additives avec dualité. Un morphisme entre elles est une paire $(G,\eta)$, où $G$ est un foncteur $F$-additif $\mathcal{C}_1 \to \mathcal{C}_2$ et $\eta$ est un isomorphisme $G \circ *_1 \rightiso *_2 \circ G$, tels que
  $$\xymatrix{
    G(M) \ar[d]^{\varpi_2} \ar[r]^{G(\varpi_1)} & G(M^{**}) \ar[d]^{\eta_{M^*}} \\
    G(M)^{**} \ar[r]^{(\eta_M)^*} & G(M^*)^*
  }$$
  est commutatif pour tout $M$.
\end{definition}

On définit aisément la somme orthogonale de formes $\epsilon$-hermitiennes. Une isométrie entre deux formes $\epsilon$-hermitiennes $(M_1,\phi_1)$, $(M_2,\phi_2)$ est un isomorphisme $h: M_1 \rightiso M_2$ tel que $h^* \phi_2 h = \phi_1$.

Pour une catégorie $F$-additive $(\mathcal{C},*,\varpi)$ et un objet $M$ dans $\mathcal{C}$, on associe la forme $\epsilon$-hermitienne hyperbolique comme la forme $\mathbb{H}(M) := (M \oplus M^*, \phi)$ où $\phi: M \oplus M^* \to M^* \oplus M^{**}$ a l'expression matricielle suivante
$$\phi:
  \begin{pmatrix}
    0 & \identity_{M^*} \\
    \epsilon\varpi_M & 0 \\
  \end{pmatrix}
$$

Les formes $\epsilon$-hermitiennes pour les anneaux discutées précédemment sont des exemples de ce formalisme.

Si $(\mathcal{C},*,\varpi)$ est un produit fini $\prod_{i=1}^r (\mathcal{C}_i, *_i, \varpi_i)$, alors la catégorie des formes $\epsilon$-hermitiennes dans $(\mathcal{C},*,\varpi)$ est canoniquement isomorphe au produit des catégories des formes $\epsilon$-hermitiennes dans les $(\mathcal{C}_i, *_i, \varpi_i)$.

\subsection{Kit de classification}\label{sec:kit-classification}
Fixons maintenant $\epsilon = \pm 1$ et une $F$-algèbre à involution $(A,\tau)$ de la forme suivante:
\begin{itemize}
  \item soit $A=F$, $\tau=\identity$;
  \item soit $A=E$ une extension quadratique de $F$ et $\tau$ la $F$-involution non triviale associée.
\end{itemize}
Si $(M,h)$ est une $(A,\tau)$-forme $\epsilon$-hermitienne, on note $\U(M,h)$ son groupe d'isométries.

Considérons la catégorie $F$-additive avec dualité $(\mathfrak{H},*,\varpi)$ (abrégée comme $\mathfrak{H}$ dans ce qui suit) suivante:
\begin{itemize}
  \item les objets sont les paires $(M,x)$, où $M$ est un $A$-module de type fini et $x \in \GL_A(M)$ est semi-simple;
  \item un morphisme $(M_1, x_1) \to (M_2, x_2)$ est un homomorphisme $A$-linéaire $f: M_1 \to M_2$ tel que $f \circ x_1 = x_2 \circ f$;
  \item pour tout objet $(M,x)$, on définit $(M,x)^* = (M^*, (x^{-1})^*)$;
  \item l'isomorphisme $\varpi: \identity \rightiso **$ est l'isomorphisme canonique $M \rightiso M^{**}$, pour tout $M$.
\end{itemize}

L'ensemble des classes d'isomorphismes de formes $\epsilon$-hermitiennes dans $\mathfrak{H}$ s'identifie à celui des triplets $(M,h,x)$ où $(M,h)$ est $\epsilon$-hermitienne et $x \in \U(M,h)$ est semi-simple, à isométrie près.

Soit $(\mathfrak{H}_0, *, \varpi)$ la catégorie $F$-additive avec dualité de $(A,\tau)$-formes $\epsilon$-hermitiennes. On a un foncteur d'oubli $\mathfrak{H} \to \mathfrak{H}_0$. Pour classifier les classes de conjugaison semi-simples dans $\U(M,h)$, il suffit de classifier les formes $\epsilon$-hermitiennes dans $\mathfrak{H}$ ayant une image dans $\mathfrak{H}_0$ isomorphe à $(M,h)$.

Pour ce faire, nous suivons la recette dans \cite{Knus91} II (6.6) qui est essentiellement une variante de l'équivalence de Morita. Tout d'abord, on vérifie les propriétés suivantes pour $\mathfrak{H}$ et $\mathfrak{H}_0$ (cf. \cite{Knus91} II. (5.2), (6.3)):

\begin{description}
  \item[C1]: tout idempotent se scinde;
  \item[C2]: tout objet $M$ admet une décomposition $M = \bigoplus_{i=1}^m N_i$ où $N_i$ est un objet indécomposable et $\End(N_i)$ est un corps; si $N_i=(V_i,x_i)$ est un objet dans $\mathfrak{H}$ alors $\End(N_i)$ est engendré par $x_i$ sur $A$;
  \item[C3]: pour tout objet $M$, le radical de Jacobson de $\End(M)$ est nul.
\end{description}

En effet, ces propriétés ne font pas intervenir la dualité $*$, elles résultent de l'algèbre linéaire. Elles impliquent aussi la propriété de Krull-Schmidt pour $\mathfrak{H}$ et $\mathfrak{H}_0$, au sens suivant.

\begin{definition}
  On dit qu'une catégorie additive vérifie la propriété de Krull-Schmidt si tout objet $M$ admet une décomposition $M = \bigoplus_{i=1}^m N_i$ en objets indécomposables, unique à permutation et isomorphisme près.
  
  Le type d'un objet $M$ dans une telle catégorie est l'ensemble des classes d'isomorphisme de ses composantes indécomposables $N_i$.
\end{definition}

La classification marchera en trois étapes. Soit $(M,x,h)$ une forme $\epsilon$-hermitienne dans $\mathfrak{H}$.

\paragraph{Étape 1} Puisque $\mathfrak{H}$ vérifie la propriété de Krull-Schmidt, on peut décomposer $(M,x)$ selon son type et puis regrouper par dualité de sorte que
$$ (M,x,h) = \bigoplus_{i=1}^m (M_i,x_i,h_i),$$
où $(M_i,x_i)$ est de type $N_i$ (avec $N_i \simeq N_i^*$) ou $(N_i, N_i^*)$ (avec $N_i \not\simeq N_i^*$) pour tout $i$.

\paragraph{Étape 2a} Fixons $1 \leq i \leq m$. Supposons d'abord $(M_i,x_i)$ de type $N_i$. D'après \cite{Knus91} II (6.5.1) l'objet $N_i$ admet une structure d'une forme hermitienne ou anti-hermitienne, disons $k_i: N_i \to N_i^*$. Posons $K_i := \End(N_i) = A(x_i)$, c'est un corps et il admet l'involution $A$-linéaire canonique d'adjonction (\cite{Knus91} II (3.2)):
$$ \tau_i: f \mapsto k_i^{-1}f^* k_i. $$
On a $\tau_i(x_i) = x_i^{-1}$ et $\tau_i$ prolonge $\tau$ sur $A$. Posons $K_i^\#$ le sous-corps fixé par $\tau_i$.

\begin{lemma}
  Si $x_i \neq \pm 1$, alors $\tau_i \neq \identity$.
\end{lemma}
\begin{proof}
  On a $\tau_i = \identity$ si et seulement si $\tau_i(x_i)=x_i$, c'est-à-dire $x_i=x_i^{-1}$, ce qui contredit l'hypothèse car $K_i$ est un corps.
\end{proof}

Si $x_i=\pm 1$, alors la classification de tels $(M_i,x_i,h_i)$ équivaut à la classification de $(A,\tau)-$formes $\epsilon$-hermitiennes. Conservons l'hypothèse que $x_i \neq \pm 1$ dans ce qui suit.

\paragraph{Étape 2b} D'autre part, si $(M_i, x_i)$ est de type $(N_i, N_i^*)$ alors il existe $r$ tel que $(M_i, x_i) = (N_i \oplus N_i^*)^{\oplus r}$ car $(M_i, x_i) \simeq (M_i, x_i)^*$ (\cite{Knus91} II (6.4)). On pose
$$ K_i := \End(N_i \oplus N_i^*) = \End(N_i) \times \End(N_i^*).$$

Il est muni de l'involution $\tau_i: (a,b) \mapsto (\varpi(b^*),a^*)$. La sous-algèbre $K_i^\#$ fixée par $\tau_i$ est isomorphe à $\End(N_i)$, donc $K_i^\#$ est un corps. On a
$$(K_i,\tau_i) \simeq (K_i^\# \times K_i^\#, (a,b) \mapsto (b,a)).$$

De plus, $x_i$ s'identifie à l'élément $(x_i|_{N_i},(x_i|_{N_i}^{-1})^*) \in K_i^\times$. Montrons que $\tau_i(x_i) = x_i^{-1} \neq x_i$. En effet, si $\tau_i(x_i)=x_i$ alors $x_i|_{N_i}=x_i|_{N_i}^{-1}$, d'où $x_i|_{N_i}=\pm 1$ car $\End(N_i)$ est un corps, mais cela implique que $N_i \simeq N_i^*$, qui est contradictoire.

\paragraph{Étape 3} Supposons pour l'instant que $x_i \neq \pm 1$ pour tout $1 \leq i \leq m$. Pour tout $i$, posons
$$ Q_i := \begin{cases}
  N_i, & \text{ si } (M_i,x_i) \text{ est de type } (N_i), \\
  N_i \oplus N_i^* & \text{ si } (M_i,x_i) \text{ est de type } (N_i, N_i^*).
\end{cases}$$

Les objets $K_i, K_i^\#, \tau_i$ sont définis comme précédemment. On peut identifier (non canoniquement) l'espace sous-jacent de $Q_i$ à $K_i$. Définissons $k_i: Q_i \to Q_i^*$ qui correspond à la forme trace sur le $A$-espace vectoriel $K_i$
$$ (q,q') \mapsto \Tr_{K_i/A}(\tau_i(q)q'). $$
Alors $(Q_i,k_i)$ est une forme hermitienne dans $\mathfrak{H}$ car $\tau_i(x_i)=x_i^{-1}$.

Posons maintenant $Q := \bigoplus_{i=1}^m Q_i$, muni du morphisme $k = \bigoplus k_i: Q \rightiso Q^*$ de sorte que $(Q,k)$ est une forme hermitienne. On construit les objets suivants:

\begin{align*}
  K & := \bigoplus_i K_i = \End(Q), \\
  x & := (x_i)_i, \; K=A[x],\\
  K^\# & := \bigoplus_i K_i^\#, \\
  \tau_K & := \bigoplus \tau_i.
\end{align*}

Alors $K$ est une $A$-algèbre étale dont $\tau_K$ est une involution prolongeant $\tau: A \to A$, $\tau_K(x)=x^{-1}$ et $K^\#$ est la sous-algèbre fixée par $\tau$.

Posons $\mathfrak{H}|_Q$ la sous-catégorie pleine de $\mathfrak{H}$ dont les objets sont facteurs directs des $Q^{\oplus r}$ ($r \geq 1$). Elle contient $(M, x)$. Considérons le foncteur $G := \Hom_{\mathfrak{H}}(Q,-)$ de $\mathfrak{H}|_Q$ dans la catégorie des $K$-modules projectifs de type fini. Il préserve les dualités et induit une équivalence de catégories $F$-additives avec dualités (\cite{Knus91} II. \S 3). Cela induit aussi une équivalence
$$ \mathfrak{Herm}^\epsilon(\mathfrak{H}|_Q) \rightiso \mathfrak{Herm}^\epsilon(K,\tau_K). $$

Donnons une forme plus utile de cette correspondance.

\begin{proposition}
  Il existe une correspondance biunivoque
  $$ \mathfrak{Herm}^\epsilon(K,\tau_K) \rightiso \mathfrak{Herm}^\epsilon(\mathfrak{H}|_Q) $$
  donnée par
  $$ (M',h') \mapsto (\Tr_{K/A})_* (M,h), $$
  où $M$ est muni de l'automorphisme qui agit par multiplication par $x \in Q^\times$. 
\end{proposition}
\begin{proof}
  Puisque la formation de cette application est compatible aux réunions de types $\{N_i\}, \{N_i, N_i^*\}$, il suffit de considérer le cas $Q = Q_i$ pour un $1 \leq i \leq m$. Soit $M = Q^{\oplus r}$, l'isomorphisme
$$ \Hom_\mathfrak{H}(M, M^*) \rightiso \Hom_K(G(M), G(M)^*) $$
respecte l'action de $\End_{\mathfrak{H}}(M) = \text{Mat}_{r \times r}(K)$ des deux côtés. Prenons une $(K,\tau)$-forme $\epsilon$-hermitienne $(M,h)$ munie de la  multiplication par $x$, alors ${\Tr_{K/A}}_* (M,h)$ munie de l'action de $x$ est une forme $\epsilon$-hermitienne dans $\mathfrak{H}|_Q$. En faisant agir les endomorphismes ``symétriques'' dans $\End_{\mathfrak{H}}(M)$, on voit que toute forme $\epsilon$-hermitienne sur $\mathfrak{H}|_Q$ est de la forme ${\Tr_{K/A}}_* (M,h)$ avec la multiplication par $x$. Cette correspondance est biunivoque.
\end{proof}

\begin{remark}\label{rem:hyperbolicite}
  Si $Q_i=N_i \oplus N_i^*$ où $N_i \not\simeq N_i^*$, alors toute forme $\epsilon$-hermitienne dans $\mathfrak{H}|_{Q_i}$ est hyperbolique(\cite{Knus91} II (6.4)).
\end{remark}

\paragraph{Conclusion} Compte tenu des raisonnements ci-dessus, on a obtenu:
\begin{theorem}\label{thm:classes-classification}
  Soit $(M,h)$ une $(A,\tau)$-forme $\epsilon$-hermitienne. Les classes de conjugaison semi-simples dans $\U(M,h)$ sont paramétrées par les données suivantes\index{paramètres pour les classes de conjugaison semi-simples}.
  \begin{itemize}
    \item Une $A$-algèbre étale de type fini $K$ et une involution $\tau_K: K \to K$. La sous-algèbre fixée par $\tau_K$ est notée par $K^\#$.
    \item Un élément semi-simple $x \in K^\times$ tel que $\tau(x)=x^{-1}$, $x \neq \pm 1$ et $K=A[x]$.
    \item Une $(K,\tau_K)$-forme $\epsilon$-hermitienne $(M_K, h_K)$. Quitte à rétrécir $(K,\tau_K)$, on peut supposer que $(M_K,h_K)$ est fidèle.
    \item Deux $(A,\tau)$-formes $\epsilon$-hermitiennes $(M_+, h_+)$ et $(M_-, h_-)$.
  \end{itemize}
  Ces paramètres sont soumis à la condition
  $$ (M,h) \simeq {\Tr_{K/A}}_* (M_K, h_K) \oplus (M_+,h_+) \oplus (M_-,h_-). $$

  Il y a une notion évidente d'équivalence pour les paramètres. Les paramètres pour une classe de conjugaison semi-simple sont uniquement déterminés à équivalence près.
\end{theorem}

Précisons cette correspondance. Étant donné un tel paramètre
$$(K/K^\#,x,(M_K,h_K),(M_\pm, h_\pm)),$$
on fixe un isomorphisme
$$ \phi: (M,h) \rightiso (M_0,h_0) := {\Tr_{K/A}}_* (M_K, h_K) \oplus (M_+,h_+) \oplus (M_-,h_-). $$

Il induit $\phi_*: \U(M_0,h_0) \rightiso \U(M,h)$. Soit $x_0 \in \U(M_0,h_0)$ qui agit comme multiplication par $x$ sur $M_K$ et comme $\pm\identity$ sur $M_\pm$. Alors $\mathcal{O}(\phi_*(x_0))$ est la classe cherchée; elle ne dépend pas du choix de $\phi$. La classe ainsi obtenue est notée $\mathcal{O}(K/K^\#,x, (M_K,h_K),(M_\pm,h_\pm))$.

\begin{remark}
  Si $A=E$ est une extension quadratique de $F$, alors $K=K^\# \otimes_F E$.
\end{remark}

Discutons deux opérations sur les paramètres.

\paragraph{Somme directe}\index{paramètres pour les classes de conjugaison semi-simples!somme directe} Soient $(M',h'), (M'',h'')$ deux $(A,\tau)$-formes $\epsilon$-hermitiennes et posons $(M,h) := (M',h') \oplus (M'',h'')$, alors on dispose d'un homomorphisme $\iota: U(M',h') \times U(M'',h'') \to U(M,h)$.

Soient $(K'/K'^\#,x',(M_{K'},h_{K'}),(M'_\pm, h'_\pm))$ un paramètre pour $\mathcal{O}(g') \in U(M',h')$ et $(K''/K''^\#,x'',(M_{K''},h_{K''}),(M''_\pm, h''_\pm))$ un paramètre pour $\mathcal{O}(g'') \in U(M'',h'')$. On définit leur somme directe
$$(K'/K'^\#,x',(M_{K'},h_{K'}),(M'_\pm, h'_\pm)) \oplus (K''/K''^\#,x'',(M_{K''},h_{K''}),(M''_\pm, h''_\pm))$$
comme un paramètre $(K/K^\#, x, (M_K,h_K), (M_\pm, h_\pm))$ où
$$ (K,\tau_K) = (K',\tau_{K'}) \times (K'', \tau_{K''}) $$
comme $F$-algèbres à involutions, $(M_K,h_K) = (M_{K'},h_{K'}) \oplus (M_{K''},h_{K''})$ la $(K,\tau_K)$-forme $\epsilon$-hermitienne correspondante, et on définit $(M_\pm,h_\pm) := (M'_\pm,h'_\pm) \oplus (M''_\pm,h''_\pm)$. Alors $(K/K^\#, x, (M_K,h_K), (M_\pm, h_\pm))$ paramètre la classe $\mathcal{O}(\iota(g',g'')) \in U(M,h)$.

\paragraph{Pousser-en-avant}\index{paramètres pour les classes de conjugaison semi-simples!pousser-en-avant} Soient $L$ une extension finie de $F$ et $B := A \otimes_F L$. Alors $(B,\identity \otimes \tau)$ est une extension finie de $(A,\tau)$. Soit $(M,h)$ une $(B,\identity \otimes \tau)$-forme $\epsilon$-hermitienne, alors $(M,{\Tr_{B/A}}_* h)$ est une $(A,\tau)$-forme $\epsilon$-hermitienne. On a une inclusion $U(M,h) \subset U(M, {\Tr_{B/A}}_* h)$.

Soit $(K/K^\#, x, (M_K,h_K),(M_\pm,h_\pm))$ un paramètre pour une classe $\mathcal{O}(g)$ dans $U(M,h)$, toutes les données étant définies par rapport à $L$, alors $(K/K^\#, x, (M_K, h_K), (M_\pm, {\Tr_{B/A}}_* h_\pm))$ paramètre la classe $\mathcal{O}(g)$ dans $U(M, {\Tr_{B/A}}_* h)$. Ici on regarde $K/K^\#$ comme une $F$-algèbre étale.

\paragraph{Décomposition}
\begin{remark}\label{rem:parametres-decomposition}\index{paramètres pour les classes de conjugaison semi-simples!décomposition}
  Si l'on décompose
  $$ (K/K^\#, x) = \prod_{i \in I} (K_i/K_i^\#, x_i) $$
  de sorte que $K_i^\#$ est un corps pour tout $i$, alors il y a un isomorphisme canonique
  \begin{align*}
    \mathfrak{Herm}^\epsilon(K,\tau_K) & \simeq \prod_{i \in I} \mathfrak{Herm}^\epsilon(K_i, \tau_{K_i}), \\
    (M_K,h_K) & \simeq \prod_{i \in I} (M_{K_i}, h_{K_i}).
  \end{align*}
  L'étude des $(K,\tau_K)$-formes $\epsilon$-hermitiennes se ramène à celle des $(K_i, \tau_{K_i})$-formes $\epsilon$-hermitiennes. Posons
  $$ I^* := \{i \in I : K_i \text{ est un corps} \}. $$
  D'après \ref{rem:hyperbolicite}, il suffit de considérer les indices $i \in I^*$.
\end{remark}

\begin{remark}\index{paramètres pour les classes de conjugaison semi-simples!restriction des scalaires}
  Indiquons deux extensions de ce formalisme. D'abord on peut étendre cette classification aux restrictions des scalaires de schémas en groupes classiques: cela équivaut à remplacer $F$ par un sous-corps $F_0 \subset F$. Ensuite, on peut aussi considérer des produits de la forme $\prod_i \Res_{F_i/F}(U_i)$, où $F_i$ est une extension finie de $F$ et $U_i$ est un groupe classique sur $F_i$, pour tout $i$.
\end{remark}

\subsection{Paramétrage explicite}\label{sec:parametrage}
Maintenant on peut décrire explicitement les classes de conjugaison dans les groupes classiques et leurs commutants. Par conséquent, on dispose aussi d'une description de la régularité.

\paragraph{Les groupes symplectiques}
Prenons $A=F$, $\epsilon=-1$ dans la classification ci-dessus. Soit $(W,\angles{\cdot|\cdot})$ un $F$-espace symplectique, alors $\U(W,\angles{\cdot|\cdot}) = \Sp(W)$. Les classes de conjugaison semi-simples dans $\Sp(W)$ sont donc paramétrées par les données suivantes.

  \begin{itemize}
    \item Données $(K/K^\#,x)$ comme dans \ref{thm:classes-classification}, $K=F[x]$.
    \item Une $(K,\tau_K)$-forme anti-hermitienne $(W_K, h_K)$ où $W_K$ est un $K$-module fidèle.
    \item Deux $F$-espaces symplectiques $(W_+, \angles{\cdot|\cdot}_+)$ et $(W_-, \angles{\cdot|\cdot}_-)$.
  \end{itemize}
  Les espaces symplectiques sont classifiés par leur dimension, donc les paramètres sont soumis à une seule condition
  $$ \dim_F W_K + \dim_F W_+ + \dim_F W_- = \dim_F W. $$

  \textit{Commutants}. Soit $g \in \mathcal{O}(K/K^\#,x,(W_K,h_K),(W_\pm,\angles{\cdot|\cdot}_\pm))$, alors
  $$ \Sp(W)^g = \Sp(W)_g = U(W_K,h_K) \times \Sp(W_+) \times \Sp(W_-). $$

  Décomposons $K=\prod_{i \in I} K_i$, $\tau_K = (\tau_i)$, $x=(x_i)$, et $(W_K,h_K) = ((W_i,h_i))_{i \in I}$ comme dans \ref{rem:parametres-decomposition}. Si $i \notin I^*$ alors $U(W_i,h_i) \simeq \GL_{K_i^\#}(n_i)$ avec $n_i := \frac{1}{2}\dim_{K_i^\#} W_i$. D'où:

  $$ U(W_K, h_K) = \prod_{i \in I^*} U_{K_i,\tau_i}(W_i, h_i) \times \prod_{i \notin I^*} \GL_{K_i^\#}(n_i). $$

  \textit{Régularité}. La classe ainsi paramétrée est régulière si et seulement si $W_+=W_-=\{0\}$ et $W_K \simeq K$. Une classe régulière est forcément fortement régulière.

\paragraph{Les groupes orthogonaux impairs}
Prenons $A=F$, $\epsilon=1$. Soit $(V,q)$ un $F$-espace quadratique de dimension impaire, alors $\U(V,q) = O(V,q)$. Les classes de conjugaison semi-simples dans $O(V,q)$ sont paramétrées par les données suivantes

  \begin{itemize}
    \item Données $(K/K^\#,x)$ comme dans \ref{thm:classes-classification}, $K=F[x]$.
    \item Une $(K,\tau_K)$-forme hermitienne $(V_K, h_K)$ où $V_K$ est un $K$-module fidèle.
    \item Deux $F$-espaces quadratiques $(V_+, q_+)$ et $(V_-, q_-)$.
  \end{itemize}
  Les paramètres sont soumis à la condition
  $$ (V,q) \simeq (\Tr_{K/F})_* (V_K, h_K) \oplus (V_+,q_+) \oplus (V_-,q_-). $$

  Pour que $\mathcal{O}(K/K^\#,x,(V_K,h_K),(V_\pm, q_\pm)))$ appartienne à $\SO(V,q)$, il faut et il suffit que
  \begin{gather*}
    \dim_F V_+ \equiv 1 \mod 2, \\
    \dim_F V_- \equiv 0 \mod 2.
  \end{gather*}

  \textit{Commutants}. Soit $g \in \mathcal{O}(K/K^\#,x,(V_K,h_K),(V_\pm,q_\pm))$, alors
  \begin{align*}
    \SO(V,q)^g & = U(V_K, h_K) \times \mathcal{R} \\
    \SO(V,q)_g & = U(V_K, h_K) \times \SO(V_+,q_+) \times \SO(V_-,q_-),
  \end{align*}
  où le groupe $U(V_K, h_K)$ admet une description analogue au cas symplectique, et
  $$ \mathcal{R} = \{(a,b) \in O(V_+,q_+) \times O(V_-,q_-) : \det a \cdot \det b = 1\}. $$
  Donc les composantes connexes de $\SO(V,q)^g$ sont définies sur $F$ et
  $$ (SO(V,q)^g : SO(V,q)_g) =
    \begin{cases}
      2, & \text{ si } V_- \neq \{0\},\\
      1, & \text{ sinon}.
    \end{cases}$$

  \textit{Régularité}. La classe ainsi paramétrée est régulière si et seulement si
  \begin{gather*}
    W_K \simeq K, \\
    \dim_F V_+ = 1, \\
    \dim_F V_- = 0 \text{ ou } 2.
  \end{gather*}
  Une classe régulière est fortement régulière si et seulement si $\dim V_- = 0$.

\paragraph{Les groupes orthogonaux pairs}
Soit $(V,q)$ un $F$-espace quadratique de dimension paire, alors $\U(V,q) = O(V,q)$. Les classes de conjugaison semi-simples dans $O(V,q)$ sont paramétrées par les mêmes données pour le cas impair, mais la condition sur $\dim_F V_\pm$ devient
  \begin{gather*}
    \dim_F V_+ \equiv 0 \mod 2, \\
    \dim_F V_- \equiv 0 \mod 2.
  \end{gather*}

Une classe de conjugaison $\mathcal{O}(K/K^\#,x,(V_K,h_K),(V_\pm,q_\pm))$ dans $O(V,q)$ se décompose en deux classes $\mathcal{O}^\pm$ par $\SO(V,q)$. On n'aura pas besoin de les distinguer dans ce texte.

\textit{Commutants}. La description des commutants est pareille que dans le cas impair.
 
\textit{Régularité}. Une classe $\mathcal{O}(K/K^\#,x,(V_K,h_K),(V_\pm,q_\pm))$ est régulière si et seulement si
\begin{gather*}
  W_K \simeq K, \\
  \dim_F V_\pm \leq 2.
\end{gather*}
Une classe régulière est fortement régulière si et seulement si $V_+=\{0\}$ ou $V_-=\{0\}$.

\paragraph{Les groupes unitaires}
Prenons $A=E$ une extension quadratique de $F$ et $\tau$ l'involution de $E$ associée.

Soit $(V,h)$ une $(E,\tau)$-forme $\epsilon$-hermitienne. Les classes de conjugaison semi-simples dans $\U(V,h)$ sont paramétrées par les données suivantes

  \begin{itemize}
    \item Données $(K/K^\#,x)$ comme dans \ref{thm:classes-classification}, $K=E[x]$. On a $K=K^\# \otimes_F E$ et $\tau_K = \identity \otimes \tau$.
    \item Une $(K,\tau_K)$-forme $\epsilon$-hermitienne $(V_K, h_K)$ où $V_K$ est un $K$-module fidèle.
    \item Deux $(E,\tau)$-formes $\epsilon$-hermitiennes $(V_+, h_+)$ et $(V_-, h_-)$.
  \end{itemize}
  Les paramètres sont soumis à la condition
  $$ (V,h) \simeq {\Tr_{K/E}}_* (V_K, h_K) \oplus (V_+,h_+) \oplus (V_-,h_-). $$

  \textit{Commutants}. Soit $g \in \mathcal{O}(K/K^\#,x,(V_K,h_K),(V_\pm,q_\pm))$, alors
  $$ U(V,h)^g = U(V,h)_g = U(V_K, h_K) \times U(V_+,h_+) \times U(V_-, h_-). $$
  Le groupe $U(V_K, h_K)$ se décrit comme dans le cas symplectique.
  
  \textit{Régularité}. La classe ainsi paramétrée est régulière si et seulement si
  \begin{gather*}
    W_K \simeq K, \\
    \dim_E V_\pm \leq 1.
  \end{gather*}
  Une classe régulière est automatiquement fortement régulière.

\begin{remark}
  La description des commutants admet une paraphrase schématique. Plus rigoureusement, il faudra remplacer le groupe $U(W_K, h_K)$ (resp. $U(V_K, h_K)$) par la restriction des scalaires $\Res_{K^\#/F} U(W_K,h_K)$ (resp. $\Res_{K^\#/F} U(V_K,h_K)$).
\end{remark}

\begin{remark}
  Supposons que $F$ est infini. Si l'on supprime $x$ dans les paramètres des classes régulières, on arrive à une classification des $F$-tores maximaux à conjugaison près dans les groupes classiques.%Cela résulte de la densité de points réguliers et l'unirationalité de tores.
\end{remark}

\paragraph{Sous-groupes de Lévi}
D'après la classification des classes de conjugaison semi-simples et la description de commutants, on arrive aussitôt à une classification de sous-groupes de Lévi.

\begin{proposition}\label{prop:Levi}
  Les sous-groupes de Lévi des groupes classiques sont classifiés comme suit.
  \begin{itemize}
    \item Si $(W,\angles{\cdot|\cdot})$ est une forme symplectique, alors les sous-groupes de Lévi de $\Sp(W)$ sont de la forme
      $$ \prod_{i=1}^r \GL_F(n_i) \times \Sp(W_+), $$
      sousmise à la condition
      $$\sum_{i=1}^r 2n_i + \dim W_+ = \dim W. $$
    \item Si $(V,q)$ est une forme orthogonale, alors les sous-groupes de Lévi de $\SO(V,q)$ sont de la forme
      $$ \prod_{i=1}^r \GL_F(n_i) \times \SO(V_+, q_+)$$
      soumise à la condition
      \begin{align*}
        (\sum_{i=1}^r n_i) \mathbb{H} \oplus (V_+,q_+) \simeq (V,q). \\
      \end{align*}
    \item Si $(V,h)$ est une $E/F$-forme hermitienne ou anti-hermitienne, alors les sous-groupes de Lévi de $U(V,h)$ sont de la forme
      $$ \prod_{i=1}^r \GL_F(n_i) \times U(V_+, h_+), $$
      soumise à la condition
      \begin{align*}
        (\sum_{i=1}^r n_i) \mathbb{H} \oplus (V_+,h_+) \simeq (V,q).\\
      \end{align*}
      où $\mathbb{H}$ signifie la $E/F$-forme hermitienne ou anti-hermitienne hyperbolique de dimension $2$.
  \end{itemize}
\end{proposition}
\begin{proof}
  Il suffit de considérer les commutants connexes des éléments semi-simples engendrant un $F$-tore déployé.
\end{proof}

\begin{remark}\label{rem:Levi-lagrangien}
  Nous utiliserons le fait suivant. En tout cas, les composantes $\GL_F(n_i)$ sont associés à des espaces hyperboliques (en une catégorie convenable), dans lesquelles les éléments laissent invariant un lagrangien. Cela découle de \S\ref{sec:kit-classification}.
\end{remark}

\paragraph{Classes assez régulières}\index{assez régulier}
\begin{definition}
  On dit qu'une classe de conjugaison semi-simple dans un groupe classique (symplectique, orthogonal ou unitaire) est assez régulière si:
  \begin{itemize}
    \item[(cas symplectique)] elle est régulière;
    \item[(cas orthogonal impair)] elle est régulière et paramétrée par $(K/K^\#,x,(V_K,h_K),(V_\pm,q_\pm))$ avec $V_- = \{0\}$, $\dim_F V_+ = 1$;
    \item[(cas orthogonal pair)] elle est régulière et paramétrée par $(K/K^\#,x,(V_K,h_K),(V_\pm,q_\pm))$ avec $V_+=V_-=\{0\}$;
    \item[(cas unitaire)] elle est régulière et paramétrée par $(K/K^\#,x,(V_K,h_K),(V_\pm,h_\pm))$ avec $V_+=V_-=\{0\}$.
  \end{itemize}
\end{definition}
Une classe assez régulière est automatiquement fortement régulière. Pour une classe assez régulière, la forme $h_K$ sur $W_K \simeq K$ dans son paramètre est décrite par
$$ \forall a,b \in K,\; h_K(a|b) = \Tr_{K/A}(c\tau(a)b) $$
où $c \in K^\times$ est tel que $\tau_K(c)=\epsilon c$.

Pour une classe assez régulière dans un groupe orthogonal impair $SO(V,q)$, la donnée $(V_+,q_+)$ dans son paramètre est déterminée par $(V,q)$ et l'autre donnée $(K/K^\#,x,c)$ d'après le théorème de Witt. En tout cas, on peut simplifier le paramètre d'une classe assez régulière en la donnée
$$(K/K^\#, x, c)$$
satisfaisant à $\tau_K(c)=\epsilon c, \tau_K(x)=x^{-1}$. Deux données $(K/K^\#,x,c)$ et $(K'/K'^\#,x',c')$ sont équivalentes si et seulement s'il existe un isomorphisme de $F$-algèbres à involutions $\sigma: (K,\tau_K) \rightiso (K',\tau_{K'})$ tel que
\begin{gather*}
  \sigma(x)=x' \\
  \sigma(c)c'^{-1} \in N_{K'/K'^\#}(K'^\times)
\end{gather*}

\begin{remark}\label{rem:c-normalisation}
  Écrivons les paramètres comme $K=\prod_{i \in I} K_i$, $K^\# = \prod_{i \in I} K_i^\#$ où $K_i^\#$ est un corps, $\tau=(\tau_i)$, et $c=(c_i)$ pour tout $i$. Pour le paramétrage dans \cite{Wa01} I.7, la forme est décrite par
  $$ h_K((a_i)_{i \in I}|(b_i)_{i \in I}) = \sum_{i \in I} [K_i:A]^{-1} \Tr_{K_i/K_i^\#}(\tau_i(a_i)b_i c_i).$$
  Autrement dit, nos paramètres $c_i$ correspondent à $c_i [K_i:A]^{-1}$ selon la convention de \cite{Wa01}.
\end{remark}

\subsection{Le cas de l'algèbre de Lie}
On peut décrire les classes de conjugaison semi-simples régulières dans les algèbres de Lie des groupes classiques au moyen des données $(K/K^\#,x,c)$ où $\tau_K(x)=-x$, $\tau_K(c)=\epsilon c$.
\begin{itemize}
  \item Pour un groupe symplectique $\Sp(W)$, on paramètre tous les éléments réguliers de cette manière. Les paramètres sont soumis à la condition
    $$ \dim_F K = \dim_F W. $$
  \item Pour un groupe orthogonal impair $\SO(V,q)$, on paramètre tous les éléments réguliers de cette manière. Les paramètres sont soumis à la condition qu'il existe un $F$-espace quadratique $(V_0,q_0)$ de dimension $1$ tel que
    $$ (V,q) \simeq (\Tr_{K/F})_* ((x,y) \mapsto \tau_K(x)yc) \oplus (V_0,q_0),$$
    l'espace $(V,q)$ est uniquement déterminé par le théorème de Witt.
  \item Pour un groupe orthogonal pair $\SO(V,q)$, on obtient seulement les éléments semi-simples réguliers qui n'ont pas de valeur propre nulle. Le paramètre correspond à deux orbites $\mathcal{O}^\pm$, mais on n'aura pas besoin de les distinguer. Les paramètres sont soumis à la condition
    $$ (V,q) \simeq (\Tr_{K/F})_* ((x,y) \mapsto \tau_K(x)yc). $$
  \item Pour un groupe unitaire $U(V,h)$, on paramètre tous les éléments réguliers de cette manière. Les paramètres sont soumis à la condition
    $$ (V,h) \simeq (\Tr_{K/E})_* ((x,y) \mapsto \tau_K(x)yc).$$
\end{itemize}

\subsection{Conjugaison géométrique, le cas des corps locaux}
Fixons un groupe classique $U$ sur $F$. %Pour simplifier la vie, on exclut le cas orthogonal pair.

\begin{proposition}
  Soient $\mathcal{O}(K/K^\#, x, c)$ et $\mathcal{O}(K'/K'^\#, x', c')$ deux classes de conjugaison semi-simple assez régulières dans $U(F)$. Elles appartiennent à la même classe de conjugaison géométrique si et seulement s'il existe un isomorphisme de $F$-algèbres à involutions
  $$\sigma: (K,\tau_K) \rightiso (K',\tau_{K'})$$
  tel que $\sigma(x)=x'$.

  La même assertion reste valide pour les classes de conjugaison semi-simples régulières dans l'algèbre de Lie.
\end{proposition}
Autrement dit, le passage à conjugaison géométrique équivaut à oublier la donnée $c$.

Supposons que $F$ est un corps local. Décomposons $(K/K^\#,x,c)= \bigoplus_{i \in I} (K_i/K_i^\#, x_i, c_i)$ comme dans \ref{rem:parametres-decomposition}. Posons\index{$\sgn_{K_i/K_i^\#}(\cdot)$}
$\sgn_{K_i/K_i^\#}: K_i^{\#\times} \to \bmu_2$ le caractère associé à l'extension quadratique $K_i/K_i^\#$ si $i \in I^*$, et $\sgn_{K_i/K_i^\#} = 1$ si $i \notin I^*$. En tout cas, on a
$$ \sgn_{K_i/K_i^\#}(a) =
  \begin{cases}
    1, & \text{ si } a \in N_{K_i/K_i^\#}(K_i^\#), \\
    -1, & \text{ sinon}.
  \end{cases}
$$

On a un isomorphisme
$$(\sgn_{K_i/K_i^\#})_{i \in I^*}: K^{\#\times}/N_{K/K^\#}(K^\times) \rightiso \bmu_2^{I^*}. $$

Posons d'ailleurs\index{$\sgn_{K/K^\#}(\cdot)$}
$$ \displaystyle \sgn_{K/K^\#} := \prod_{i \in I}\sgn_{K_i/K_i^\#}: K^{\#\times} \to \bmu_2. $$

Supposons que $U$ est un groupe classique et $(K/K^\#,x,c)$ paramètre une classe de conjugaison assez régulière $\mathcal{O}(K/K^\#,x,c)$ dans $U$. Soit $c' \in K^\times$ tel que $\tau(c')=\epsilon c'$, $\epsilon = -1$ dans le cas symplectique ou anti-hermitien et sinon $\epsilon=1$. On dit que $\mathcal{O}(K/K^\#,x,c')$ existe dans $U$ si $(K/K^\#,x,c')$ paramètre une classe de conjugaison dans $U$.

\begin{proposition}[\cite{Wa01}, I.7]
  Si $U$ est symplectique, alors $\mathcal{O}(K/K^\#,x,c')$ existe toujours dans $U$. Par conséquent les classes de conjugaison dans la classe de conjugaison géométrique contenant $\mathcal{O}(K/K^\#,x,c)$ sont en bijection avec $\bmu_2^{I^*}$.

  Supposons $F$ non archimédien. Si $U$ est orthogonal ou unitaire, alors $\mathcal{O}(K/K^\#,x,c')$ existe dans $U$ si et seulement si $$\sgn_{K/K^\#}(c^{-1}c')=1.$$ Les classes de conjugaison dans la classe de conjugaison géométrique contenant $\mathcal{O}(K/K^\#,x,c)$ sont en bijection avec
  $$(\bmu_2)_0^{I^*} := \{ (t_i) \in \bmu_2^{I^*} : \prod_i t_i = 1 \}.$$
  Les mêmes assertions restent valides pour l'algèbre de Lie.
\end{proposition}

\section{Le caractère de la représentation de Weil}
Sauf mention expresse du contraire, $F$ est toujours un corps local de caractéristique nulle dans cette section.

\subsection{Formules du caractère}
Nous adoptons l'approche de \cite{Th08} et nous utiliserons systématiquement la construction de Lion-Perrin dans cette section.

Désignons par $\overline{W}$ l'espace vectoriel symplectique $(W, -\angles{\cdot|\cdot})$. Si $x \in \Sp(W)$, alors le graphe
$$\Gamma_x := \{(w,xw): w \in W\}$$
est un lagrangien dans $\overline{W} \oplus W$. On a un plongement $f: \Sp(W) \to  \Sp(\overline{W} \oplus W)$ donné par $f(x) = (1,x)$. 

\begin{proposition}[\cite{Th08} 1.3]
  Il existe un homomorphisme continu injectif
  $$\tilde{f}: \MMp{2}(W) \to \MMp{2}(\overline{W} \oplus W)$$
  donné par $\tilde{f}(x,t) = ((1,x), f_x(t))$ dans la construction de Lion-Perrin, où $f_x(t)$ est déterminé par
  $$ f_x(t)(\ell \oplus \ell) = t(\ell). $$
  De plus, il rend le diagramme ci-dessous commutatif:
  $$\xymatrix{
    \MMp{2}(W) \ar[rr]^{\tilde{f}} \ar[d] &  & \MMp{2}(\overline{W} \oplus W) \ar[d] \\
    \Sp(W) \ar[rr]_f & & \Sp(\overline{W} \oplus W)
  }.$$
\end{proposition}

Définissons des ouverts de Zariski denses dans $\Sp(W)$\index{$\Sp(W)^\dagger,\Sp(W)^\ddagger$}:
\begin{align*}
  \Sp(W)^\dagger & := \{x \in \Sp(W) : \det(x-1) \neq 0 \}, \\
  \Sp(W)^\ddagger & := \{x \in \Sp(W) : \det(x^2-1) \neq 0, \}.
\end{align*}
On vérifie que $\Sp(W)^\dagger \supset \Sp(W)^\ddagger \supset \Sp(W)_\text{reg}$.
Posons $\MMp{\mathbf{f}}(W)^\dagger$, $\MMp{\mathbf{f}}(W)^\ddagger$, $\Spt_\psi(W)^\dagger$, $\Spt_\psi(W)^\ddagger$ leurs images réciproques dans $\MMp{\mathbf{f}}(W)$, $\Spt_\psi(W)$, respectivement, où $\mathbf{f} \in 2\Z_{\geq 1}$.

Fixons une mesure de Haar sur $\MMp{2}(W)$. L'admissibilité de $\omega_\psi^\pm$ permet de définir le caractère $\Theta_\psi^\pm = \Tr(\omega_\psi^\pm)$ comme une distribution sur $\MMp{2}(W)$\index{$\Theta_\psi^\pm$}, d'où la distribution $\Theta_\psi = \Tr(\omega_\psi)$. Énonçons les formules du caractère de Maktouf.

\begin{theorem}[K. Maktouf \cite{Mak99}] \label{prop:formule-caractere-Thomas}
  La distribution $\Theta_\psi$ est lisse sur $\MMp{2}(W)^\dagger$. Si $\tilde{x}=(x,t) \in \MMp{2}(W)^\dagger$, alors
  \begin{enumerate}
    \item en rappelant la définition de $\mathrm{ev}_{\Gamma_1}(\cdot)$ \eqref{eqn:ev-def}, on a
    $$\Theta_\psi(\tilde{x}) = \dfrac{\mathrm{ev}_{\Gamma_1}(\tilde{f}(\tilde{x}))}{|\det(x-1)|^{\frac{1}{2}}}; $$

    \item en fixant $\ell \in \Lag(W)$, on a aussi
  $$ \Theta_\psi(x,t)= \dfrac{t(\ell) \gamma_\psi(\tau(\Gamma_x, \Gamma_1, \ell \oplus \ell))}{|\det(x-1)|^{\frac{1}{2}}}; $$

    \item il y a une autre formule avec ambiguïté de signe:
  $$ \Theta_\psi(\tilde{x}) = \pm \dfrac{\gamma_\psi(1)^{\dim W -1} \gamma_\psi(\det(x-1))}{|\det(x-1)|^{\frac{1}{2}}}. $$
  \end{enumerate}
  
  De plus, $|\det(x+1)|^{\frac{1}{2}}\Theta_\psi(\tilde{x})$ est localement constante sur $\MMp{2}(W)^\dagger$.
\end{theorem}
\begin{remark}\label{rem:caractere-C}
  Si $F=\C$, on identifie $\MMp{2}(W)$ à $\bmu_2 \times \Sp(W)$. On a
  $$ \Theta_\psi(\noyau,x) = \frac{\noyau}{|\det(x-1)|_\C^{\frac{1}{2}}}; $$
  où la valeur absolue $|\cdot|_\C$ est définie par $|x+yi|_\C = x^2+y^2$ pour tout $x,y \in \R$.
\end{remark}

Notons par le même symbole $\Theta_\psi$ le caractère de $\omega_\psi$ sur la grosse extension $\Spt_\psi$, qui est bien défini comme une distribution. On en déduit une formule pour $\Theta_\psi$ au moyen du modèle de Schrödinger.

\begin{corollary}\label{prop:formule-caractere-Spt}
  Fixons $\ell \in \Lag(W)$. Soit $\bar{x} = (x, zM_\ell[x]) \in \Spt_\psi(W)$ où $z \in \C^\times$ et $x \in \Sp(W)^\dagger$, alors
  $$ \Theta_\psi(\bar{x})= \dfrac{z \gamma_\psi(\tau(\Gamma_x, \Gamma_1, \ell \oplus \ell))}{|\det(x-1)|^{\frac{1}{2}}}. $$
\end{corollary}
\begin{proof}
  Vu la spécificité de $\omega_\psi$, il suffit de vérifier l'énoncé pour le cas $\bar{x} \in \MMp{2}(W)$. L'immersion $\MMp{2}(W) \hookrightarrow \Spt_\psi(W)$ est donnée par $(x,t) \mapsto (x,t(\ell)M_\ell[x])$. Donc il suffit de vérifier le cas $z=t(\ell)$ et cela résulte immédiatement dudit théorème.
\end{proof}

\begin{corollary}\label{prop:caractere-stable}
  Supposons que $\mathbf{f} \in 2\Z_{\geq 1}$. Si $\tilde{x},\tilde{y} \in \MMp{\mathbf{f}}(W)^\dagger$ sont tels que $x,y \in \Sp(W)$ sont géométriquement conjugués, alors il existe $\noyau \in \bmu_\mathbf{f}$ tel que
  $$ \Theta_\psi(\tilde{x}) = \noyau \cdot \Theta_\psi(\tilde{y}).$$
\end{corollary}
\begin{proof}
  Pour $\mathbf{f}=2$, cela résulte de la troisième formule du théorème.  Le cas général en découle par la spécificité du caractère.
\end{proof}

\begin{corollary}\label{prop:-1-caracterisation}
  Soit $-1 \in \MMp{8}(W)$ l'élément défini dans \ref{def:-1}. Alors
  $$ \Theta_\psi(-1) = |2|^n. $$
\end{corollary}
\begin{proof}
  Vu \ref{prop:formule-caractere-Spt}, il suffit de démontrer que
  $$ \tau(\Gamma_{-1}, \Gamma_1, \ell \oplus \ell)=0$$
  pour tout $\ell \in \Lag(W)$. D'après \ref{prop:Maslov-dimension}, on a
  \begin{align*}
    \dim \tau(\Gamma_{-1}, \Gamma_1, \ell \oplus \ell)= & \frac{\dim W \oplus \overline{W}}{2} \\
    & - \dim \Gamma_{-1} \cap \Gamma_1 - \dim \Gamma_1 \cap (\ell \oplus \ell) - \dim (\ell\oplus\ell) \cap \Gamma_{-1} \\
    & + 2 \dim \Gamma_1 \cap \Gamma_{-1} \cap (\ell\oplus\ell) \\
    = & \dim W - 0 - \dim\ell - \dim\ell + 2 \cdot 0 \\
    = & 0. 
  \end{align*}
  D'où l'assertion.
\end{proof}

\begin{corollary}\label{prop:caractere-Levi}
  Fixons $\ell \in \Lag(W)$, soit $P_\ell$ le stabilisateur de $\ell$ et $\sigma_\ell: P_\ell(F) \to \MMp{8}(W)$ le scindage défini dans \ref{prop:scindage-Schrodinger}. Supposons que $x \in \Sp(W)^\dagger$ et qu'il existe $\ell' \in \Lag(W)$ tel que $W = \ell \oplus \ell'$ et $x \in M(F) := (P_\ell \cap P_{\ell'})(F)$, alors
  $$ \Theta_\psi(\sigma_\ell(x)) \in \R_{>0}. $$
\end{corollary}
\begin{proof}
  D'après \ref{prop:formule-caractere-Spt} et la définition de $\sigma_\ell$, il suffit de montrer que $\tau(\Gamma_x, \Gamma_1, \ell\oplus\ell)=0$. En identifiant $\ell'$ à $\ell^\vee$, le dual de $\ell$, on a $x = (y,{}^t \!y^{-1})$ où $y = x|_\ell$. Donc
  \begin{align*}
    \Gamma_x \cap \Gamma_1 & = \{ (w,w) : w \in W, x(w)=w \} \\
    & \simeq \{v \in \ell, y(v)=v \}^{\oplus 2}, \\
    \Gamma_1 \cap (\ell\oplus\ell) & = \{(w,w) : w \in \ell \} \simeq \ell,\\
    (\ell \oplus \ell) \cap \Gamma_x &= \{(w,x(w)) : w \in \ell \} \simeq \ell, \\
    \Gamma_x \cap \Gamma_1 \cap (\ell\oplus\ell) &= \{(w,w) : w \in \ell, x(w)=w \} \\
    & \simeq \{v \in \ell, y(v)=v \}.
  \end{align*}
  On en déduit que $\dim\tau(\Gamma_x, \Gamma_1, \ell\oplus\ell) = 0$ à l'aide de \ref{prop:Maslov-dimension}.
\end{proof}
\begin{remark}\label{rem:caractere-Levi}
  L'hypothèse est équivalente à: $x \in \Sp(W)^\dagger$ et $x$ appartient à un facteur de Lévi $M$ de $P_\ell$. Ce corollaire caractérise le scindage de $\rev: \MMp{8}(W) \to \Sp(W)$ au-dessus de $M(F)$ car $\Sp(W)^\dagger \cap M(F)$ est dense dans $M(F)$.
\end{remark}

\subsection{Formules pour $\Theta_\psi^+ - \Theta_\psi^-$, la forme de Cayley}
\begin{proposition}\label{prop:-1-echangement}
  La distribution $\Theta_\psi^+ - \Theta_\psi^-$ est lisse sur $\{\tilde{x} \in \MMp{8}(W) : \det(x+1) \neq 0 \}$. Pour tout $\tilde{x} \in \MMp{8}(W)$ tel que $\det(x+1) \neq 0$, on a
  $$ (\Theta_\psi^+ - \Theta_\psi^-)(\tilde{x}) = (\Theta_\psi^+ + \Theta_\psi^-)((-1) \cdot \tilde{x}).$$
\end{proposition}

Ici $-1$ est l'élément défini dans \ref{def:-1}.

\begin{proof}
  Soit $S_\psi = S_\psi^+ \oplus S_\psi^-$ la décomposition de $S_\psi$ selon $\omega_\psi = \omega_\psi^+ \oplus \omega_\psi^-$. On conclut par \ref{prop:-1-action}.
\end{proof}

Donnons un autre lien entre $\Theta_\psi^+ + \Theta_\psi^-$ et $\Theta_\psi^+ - \Theta_\psi^-$. Introduisons d'abord des formes quadratiques auxiliaires\index{$q[X]$}.

\begin{definition}
  Pour tout $X \in \syp(W)$ inversible, définissons une $F$-forme quadratique $q[X]$ sur $W$ par
  $$ q[X](w_1|w_2) = \angles{X w_1|w_2}. $$
\end{definition}

\begin{definition}
  Pour tout $x \in \Sp(W)^\ddagger$, définissons un élément $C_x$ dans $\End(W)$ par
  $$ C_x := 2 \cdot \frac{x-1}{x+1}. $$
  On vérifie que $C_x \in \syp(W)$. Si $x$ est semi-simple régulier alors $C_x$ l'est aussi.
\end{definition}

On appelle $q[C_x]$ la forme de Cayley\index{forme de Cayley}.

\begin{theorem}[T. Thomas \cite{ThNote}]\label{prop:formule-caractere-Thomas2}
  Si $\tilde{x} \in \Spt_\psi(W)^\ddagger$, alors
  $$ \frac{\Theta_\psi^+ + \Theta_\psi^-}{\Theta_\psi^+ - \Theta_\psi^-}(\tilde{x}) = \gamma_\psi(q[C_x]) \cdot  \left|\frac{\det(x+1)}{\det(x-1)}\right|^{\frac{1}{2}}. $$
\end{theorem}
\begin{proof}
  Par la spécificité de $\Theta_\psi^\pm$, il suffit de considérer le cas $m=2$. Fixons $\ell \in \Lag(W)$. Soit $\tilde{x} := (x,t) \in \MMp{2}(W)^\ddagger$ une paire dans la construction de Perrin-Lion. Rappelons que l'on a défini une constante $m_{-1} := m_{-1}(\ell)$ dans la construction. On a $(x,t)(-1,m_{-1})=(-x,m_{-1} t)$ car $\tau(\ell,\ell,x\ell)=0$ par \ref{prop:Maslov-dimension}. D'après \ref{prop:formule-caractere-Thomas}, on a
  \begin{align*}
    \Theta_\psi(x,t) & = \frac{\mathrm{ev}_{\Gamma_1}(\tilde{f}(x,t))}{|\det(x-1)|^{1/2}}, \\
    \Theta_\psi(-x,m_{-1}t) & = \frac{\mathrm{ev}_{\Gamma_1}(\tilde{f}(-x,m_{-1} t))}{|\det(x+1)|^{1/2}}.
  \end{align*}

  Comme $\tilde{f}$ est un homomorphisme, on a
  $$ \mathrm{ev}_{\Gamma_1}(\tilde{f}(-x,m_{-1} t)) = \mathrm{ev}_{\Gamma_1}(\tilde{f}(x,t)) \cdot \mathrm{ev}_{\Gamma_1}(\tilde{f}(-1, m_{-1})) \cdot \gamma_\psi(\tau(\Gamma_1, \Gamma_x, \Gamma_{-x})). $$

  On sait que
  \begin{align*}
    \mathrm{ev}_{\Gamma_1}(\tilde{f}(-1, m_{-1})) & = \mathrm{ev}_{\ell\oplus\ell}(\tilde{f}(-1, m_{-1})) \gamma_\psi(\tau(\ell\oplus\ell, \ell\oplus\ell, \Gamma_{-1}, \Gamma_1)), \\
    & = m_{-1}(\ell)
  \end{align*}
  car $\tau(\ell\oplus\ell, \ell\oplus\ell, \Gamma_{-1}, \Gamma_1)=0$ par \ref{prop:Maslov-dimension}.

  Soit $-1 \in \Spt_\psi(W)$ l'élément défini dans \ref{def:-1}, alors $(-1) \cdot (x,t) = m_{-1}(\ell)^{-1} \cdot (-x, m_{-1}t)$ comme un élément dans $\Spt_\psi(W)$, d'où $\Theta_\psi(-x, m_{-1}t) = m_{-1}(\ell)(\Theta_\psi^+ - \Theta_\psi^-)(x,t)$ par \ref{prop:-1-echangement}. En utilisant les formules ci-dessus, on arrive à
  $$ \frac{\Theta_\psi^+ + \Theta_\psi^-}{\Theta_\psi^+ - \Theta_\psi^-}(x,t) = \left|\frac{\det(x+1)}{\det(x-1)} \right|^{\frac{1}{2}} \gamma_\psi(\tau(\Gamma_1, \Gamma_x, \Gamma_{-x})). $$

  Calculons $\gamma_\psi(\tau(\Gamma_1, \Gamma_x, \Gamma_{-x}))$. En vertu de l'invariance symplectique, on a $\tau(\Gamma_1, \Gamma_x, \Gamma_{-x}) \simeq \tau(\Gamma_{x^{-1}}, \Gamma_1, \Gamma_{-1})$. Posons $y := x^{-1}$, alors $C_y = -C_x$ et donc $q[C_y] \simeq -q[C_x]$. Il reste à prouver que $q[C_y]$ est isomorphe à $\tau(\Gamma_{-1}, \Gamma_y, \Gamma_1)$.

  Puisque $\Gamma_1 \cap \Gamma_{-1} = \Gamma_1 \cap \Gamma_y = \{0\}$, on sait d'après \cite{Per81} 1.4.2 que  $\tau(\Gamma_{-1}, \Gamma_y, \Gamma_1)$ est Witt équivalent à la forme quadratique $r$ sur $\Gamma_y$ définie par
  $$ r(v) = \angles{\pi_{\Gamma_{-1}}(v)|v}_{\overline{W} \oplus W}, $$
  où $\pi_{\Gamma_{-1}}$ est la projection $\overline{W} \oplus W = \Gamma_1 \oplus \Gamma_{-1} \to \Gamma_{-1}$. Soit $v = (w,y(w))$ où $w \in W$. Posons
  \begin{align*}
    w_1 & := \frac{1+y}{2}(w), \\
    w_2 & := \frac{1-y}{2}(w),
  \end{align*}
  de sorte que $(w,y(w)) = (w_1, w_1) + (w_2, -w_2)$. Alors $\pi_{\Gamma_{-1}}(v) = (w_2, -w_2)$, et
  \begin{align*}
    r(v) & = \angles{(w_2,-w_2)|(w, y(w))}_{\overline{W} \oplus W} \\
    & = -\angles{w_2|w} - \angles{w_2|y(w)} \\
    & = \frac{-1}{2} \angles{(1-y)w|(1+y)w}.
  \end{align*}
  Après le changement de variable $w' = \frac{1}{2}(1+y)w$, la forme $r$ est isomorphe à la forme $r_1$ sur $W$ définie par:
  $$ r_1(w') = \angles{2(y-1)(y+1)^{-1}w'|w'} = q[C_y](w'), $$
  ce qui fallait démontrer.
\end{proof}

\subsection{Paramètres et la forme de Cayley}

\begin{lemma}\label{prop:changement-de-signes-0}
  Soit $K/K^\#$ une $F$-algèbre étale à involution. Pour tout $r \in K^{\#\times}$, soit $(K,q(r))$ le $F$-espace quadratique défini par
  $$ q(r) := (\Tr_{K^\#/K})_*(r N_{K/K^\#}(\cdot)) $$
  où on regarde $(K, N_{K/K^\#}(\cdot))$ comme un $K^\#$-espace quadratique. Alors pour tout $r,r' \in K^{\#\times}$,
  $$ \gamma_\psi(q(r'))=\gamma_\psi(q(r)) \; \sgn_{K/K^\#}\left(\frac{r}{r'}\right). $$
\end{lemma}
\begin{proof}
  Écrivons
  $$K = \prod_{i \in I} K_i $$
  tels que les $K_i^\#$ sont des corps, comme dans \ref{rem:parametres-decomposition}. Écrivons aussi $r=(r_i)_i$. Alors
  $$\sgn_{K/K^\#} = \prod_{i \in I^*} \sgn_{K_i/K_i^\#}. $$

  Posons $\psi_i := \psi \circ \Tr_{K_i^\#/F}$ pour $i \in I^*$. Alors
  $$ \gamma_\psi(q(r)) = \prod_{i \in I} \gamma_{\psi_i}(r_i N_{K_i/K_i^\#}(\cdot)),$$

  Effectuons la même décomposition pour $q(r')$. Il suffira de démontrer que pour tout $i \in I$,
  \begin{equation}\label{eqn:r_i}
    \gamma_{\psi_i}(r_i N_{K_i/K_i^\#}(\cdot)) = \gamma_{\psi_i}(r'_i N_{K_i/K_i^\#}(\cdot)) \cdot \sgn_{K_i/K_i^\#}\left(\frac{r_i}{r'_i}\right).
  \end{equation}

  Prenons $d_i \in K_i^\times$ tel que $K_i = K_i^\#(\sqrt{d_i})$, alors $\det(r_i N_{K_i/K_i^\#}(\cdot)) = \det N_{K_i/K_i^\#}(\cdot) = -d_i$. D'autre part, l'invariant de Hasse $s(\cdot)$ satisfait à
  \begin{align*}
    s(r_i N_{K_i/K_i^\#}(\cdot)) & = (r_i, -r_i d_i)_{K_i^\#} \\
    & = (r_i,-r_i)_{K_i^\#} (r_i, d_i)_{K_i^\#}\\
    & = (r_i,d_i)_{K_i^\#}.
  \end{align*}
  Les mêmes formules restent valides si $r_i$ est remplacé par $r'_i$.

  On sait que (\cite{LP81} 1.3.4) pour tout $F$-espace quadratique $(E,Q)$, on a
  $$\gamma_\psi(Q) = \gamma_\psi(1)^{\dim_F E - 1} \gamma_\psi(\det Q) s(Q). $$
  Cela entraîne que
  $$ \gamma_\psi(r_i N_{K_i/K_i^\#}(\cdot)) = \gamma_\psi(r'_i N_{K_i/K_i^\#}(\cdot)) \cdot \left(\frac{r_i}{r'_i}, d_i \right)_{K_i^\#}.$$
  Or $(\cdot ,d_i)_{K_i^\#} = \sgn_{K_i/K_i^\#}$, cela démontre \eqref{eqn:r_i}.
\end{proof}

Soit $X \in \syp(W)$ semi-simple régulier tel que
$X \in \mathcal{O}(K/K^\#,a,c)$. Désignons l'involution de $K$ par $\tau$, on a $\tau(a)=-a$, $\tau(c)=-c$. Alors on peut identifier $q[X]$ à la $F$-forme quadratique $q[a]$ sur $K$:
$$ q[a]: (x,y) \mapsto -\Tr_{K/F}(ac\tau(x)y). $$

\begin{lemma}
  Soit $c' \in K^\times$ tel que $\tau(c')=-c'$. Soit $X' \in \syp(W)$ semi-simple tel que
  $$ X' \in \mathcal{O}(K/K^\#,a,c'). $$
  Alors
  $$ \gamma_\psi(q[X']) = \gamma_\psi(q[X]) \cdot \sgn_{K/K^\#}(c'c^{-1}).$$
\end{lemma}
\begin{proof}
  On conclut en appliquant \ref{prop:changement-de-signes-0} à $r := ac$ et $r' := ac'$.
\end{proof}

\begin{corollary}\label{prop:changement-de-signes}
  Soient $x,y \in \Sp(W)$ semi-simples réguliers. Si
  \begin{gather*}
    x \in \mathcal{O}(K/K^\#,a,c) \\
    y \in \mathcal{O}(K/K^\#,a,c')
  \end{gather*}
  alors
  $$ \gamma_\psi(q[C_x]) = \gamma_\psi(q[C_y]) \cdot \sgn_{K/K^\#}(c'c^{-1}). $$
\end{corollary}
\begin{proof}
  On a
  \begin{gather*}
    C_x \in \mathcal{O}\left(K/K^\#, 2\cdot\frac{a-1}{a+1},c\right) \\
    C_y \in \mathcal{O}\left(K/K^\#, 2\cdot\frac{a-1}{a+1},c'\right).
  \end{gather*}
  On applique le lemme précédent pour conclure.
\end{proof}

Notre but principal est l'assertion suivante.

\begin{theorem}\label{prop:calcul-qX}
  Soient $P_X$ le polynôme caractéristique de $X \in \End_F(W)$ et $\dot{P}_X$ sa dérivée, alors
  $$ \gamma_\psi(q[X]) = \gamma_\psi((-1)^{n-1}) \gamma_\psi(\det X) \cdot \sgn_{K/K^\#}(c^{-1} \dot{P}_X(a)) .$$
\end{theorem}

La démonstration se fait en plusieurs étapes. Traduisons d'abord tous les objets en termes de paramètres. On a $K \simeq F[T]/(P_X(T))$, et $P_X$ est égal au polynôme caractéristique $P_a$ de $a \in K$. Le déterminant $\det X$ est égal à $N_{K/F}(a)$. Observons aussi $\dot{P}_a(a) \in K^\times$. Posons $b := a^2$ et $P_b$ son polynôme caractéristique, alors $K^\# = F[b]$ et

\begin{gather*}
  P_a(T) = P_b(T^2), \\
  \dot{P}_a(a) = 2\dot{P}_b(b)a.
\end{gather*}

On a $\dot{P}_X(X) \in \syp(W)$ et sa classe de conjugaison est paramétrée par $(K/K^\#,\dot{P}_a(a),c)$. Il suffira donc de prouver

\begin{align}\label{eqn:comparaison-facteur-transfert}
  \gamma_\psi(q[a]) & = \gamma_\psi((-1)^{n-1}) \gamma_\psi(N_{K/F}(a)) \sgn_{K/K^\#}(2c^{-1}\dot{P}_b(b)a).
\end{align}

D'après \ref{prop:changement-de-signes}, les deux côtés de \ref{prop:calcul-qX} varient de la même façon par rapport à $c$, donc il suffit d'établir \eqref{eqn:comparaison-facteur-transfert} dans le cas où
\begin{equation}\label{eqn:choix-c}
  c =  \dot{P}_a(a) = 2\dot{P}_b(b)a.
\end{equation}

La forme $q[a]$ est une $F$-forme quadratique sur le $F$-espace vectoriel $K$. On écrit un élément $w \in K$ comme $w=x+ya$ ($x,y \in K^\#$), alors $q[a]$ s'écrit comme

\begin{align*}
  w=x+ya & \mapsto -\Tr_{K/F}(2b \dot{P}_b(b) N_{K/K^\#}(x+ya) ) \\
  & = -4 \Tr_{K^\#/F}(b\dot{P}_b(b)(x^2 - by^2)). 
\end{align*}

En posant $x' = 2\dot{P}_b(b)by$, $y' = 2\dot{P}_b(b)x$, on obtient
$$q[a] \simeq (\Tr_{K^\#/F})_* \Qform{\dot{P}_b(b)^{-1}, -\dot{P}_b(b)^{-1}b}. $$

Les lemmes suivants sur la forme trace concluront la démonstration de \eqref{eqn:comparaison-facteur-transfert}.

\begin{lemma}\label{prop:trace0}
  Définissons
  $$ g_0 + g_1 T + \cdots + g_{n-1}T^{n-1} := P_b(T)/(T-b) \in K^\#[T]. $$
  Alors la base duale de $\{1,\ldots, b^{n-1}\}$ par rapport à la forme trace $(\Tr_{K^\#/F})_* \Qform{1}: x \mapsto \Tr_{K^\#/F}(x^2)$ est
    $$ \{\dot{P}_b(b)^{-1} g_0, \ldots, \dot{P}_b(b)^{-1} g_{n-1} \}. $$
    
  Si l'on écrit $P_b(T) = h_0 + h_1 T + \cdots + T^n$, alors les $g_k$ sont donnés par la formule

  $$ g_k = \sum_{j=k+1}^n h_j b^{j-k-1}. $$
\end{lemma}
\begin{proof}
  L'énoncé est bien connu dans le cas où $K^\#$ est un corps. Notre démonstration est aussi une variante de la démonstration traditionnelle. On a

  $$\forall s=0,\ldots,n-1, \; \sum_{\beta \in \bar{F}: P_b(\beta)=0} \frac{P_b(T)}{T-\beta} \frac{\beta^s}{\dot{P}_b(\beta)} = T^s $$
  car la différence des deux côtés est un polynôme sur $\bar{F}$ de degré $n-1$ qui s'annule en toute racine de $P_b$. Cela équivaut

  $$\forall s=0,\ldots,n-1, \; \Tr_{K^\#/F} \left( \frac{P_b(T)}{T-b} \frac{b^s}{\dot{P}_b(b)} \right) = T^s ,$$
  où $\Tr_{K^\#/F}(\cdots)$ signifie la trace appliquée à chaque coefficient d'un polynôme.

  En comparant les coefficients, on obtient
  $$ \Tr_{K^\#/F}(\dot{P}_b(b)^{-1} g_k \cdot b^s) = \delta_{k,s} .$$
    
  On vérifie la formule des $g_k$ par $g_{n-1}=1$ et la récurrence $g_{k-1}-bg_k = h_k$.
\end{proof}

Posons $m := \lfloor \frac{n}{2} \rfloor$.

\begin{lemma}\label{prop:trace1}
  La $F$-forme quadratique
  $$ q_1 := (\Tr_{K^\#/F})_* \Qform{\dot{P}_b(b)^{-1}}$$
  est isomorphe à
  \begin{enumerate}
    \item $m \mathbb{H}$, si $n=2m$.
    \item $m \mathbb{H} \oplus \Qform{1}$, si $n=2m+1$.
  \end{enumerate}
\end{lemma}
\begin{proof}
  Prenons la base $\{g_0, \ldots, g_{n-1}\}$. Elle est duale à $\{1,b, \ldots, b^{n-1}\}$ par rapport à $q_1$ d'après \ref{prop:trace0}. La matrice $Q := (q_1(g_i, g_j))_{i,j}$ est:
  $$\begin{pmatrix}
    h_1 &  h_2 & \cdots & h_{n-1} & 1 \\
    h_2 & h_3 & \cdots & 1 & 0 \\
    h_3 & h_4 & \cdots & 0 & 0\\
    \vdots & \vdots & \vdots & \vdots & \vdots\\
    1 & 0 & 0 & \cdots & 0
  \end{pmatrix}.$$

  Notons que $\det Q = (-1)^m$ et $g_0, \ldots, g_{m-1}$ forment un sous-espace totalement isotrope de dimension $m$. Si $n=2m$, alors $q_1 \simeq m\mathbb{H}$. Si $n=2m+1$, alors il existe $a \in F^\times$ tel que $q_1 \simeq m\mathbb{H} \oplus \Qform{a}$. En comparant les déterminants, on peut prendre $a =1$.
\end{proof}

\begin{lemma}\label{prop:trace2}
  La $F$-forme quadratique
  $$ q_2 := (\Tr_{K^\#/F})_* \Qform{b^{-1} \dot{P}_b(b)^{-1}}$$
  est isomorphe à
  \begin{enumerate}
    \item $(m-1) \mathbb{H} \oplus \Qform{1, -N_{K^\#/F}(b)}$, si $n=2m$;
    \item $m \mathbb{H} \oplus \Qform{N_{K^\#/F}(b)}$, si $n=2m+1$.
  \end{enumerate}
\end{lemma}
\begin{proof}
  Prenons la base $\{bg_0, \ldots, bg_{n-1}\}$. Elle est duale à $\{1,b, \ldots, b^{n-1}\}$ par rapport à $q_2$ d'après \ref{prop:trace0}. La matrice $Q_2 := (q_2(bg_i, bg_j))_{i,j}$ est:
  $$\begin{pmatrix}
    -h_0 & 0 & 0 & \cdots & 0 & 0 \\
    0 & h_2 &  h_3 & \cdots & h_{n-1} & 1 \\
    0 & h_3 & h_4 & \cdots & 1 & 0 \\
    0 & h_4 & h_5 & \cdots & 0 & 0\\
    \vdots & \vdots & \vdots & \vdots & \vdots & \vdots\\
    0 & 1 & 0 & 0 & \cdots & 0
  \end{pmatrix}.$$

  Donc $q_2$ se décompose en $q_2 = \Qform{-h_0} \oplus q_3$ selon les blocs. 

  La forme $q_3$ est déterminée de la même façon qu'en \ref{prop:trace1}: $q_3 \simeq m\mathbb{H}$ si $n=2m+1$, $q_3 \simeq (m-1)\mathbb{H} \oplus \Qform{1}$ si $n=2m$.
    
  D'autre part $-h_0 = (-1)^{n+1}N_{K^\#/F}(b)$ d'après la description de \ref{prop:trace0}. Ceci permet de conclure.
\end{proof}

Observons que $q_2 \simeq (\Tr_{E^\#/F})_* \Qform{b\dot{P}_b(b)^{-1}}$.

\begin{proof}[Démonstration du théorème \ref{prop:calcul-qX}]
  Avec le choix de $c$ fait en \eqref{eqn:choix-c}, on a vu que
  $$ q[a] \simeq q_1 \oplus -q_2, $$
  où $q_1, q_2$ sont définies dans \ref{prop:trace1} et \ref{prop:trace2}. En utilisant \ref{prop:trace1}, \ref{prop:trace2} et le fait $N_{K/F}(a) = N_{K^\#/F}(-b) = (-1)^n N_{K^\#/F}(b)$, on arrive à
   $$ q[a] \simeq (n-1)\mathbb{H} \oplus \Qform{(-1)^{n-1}, N_{K/F}(a)}. $$

  D'où
  $$ \gamma_\psi(q[a]) = \gamma_\psi((-1)^{n-1}) \gamma_\psi(N_{K/F}(a)).$$
  Puisque $\sgn_{K/K^\#}(2c^{-1}\dot{P}_b(b)a)=1$ avec notre choix de $c$, cela établit \eqref{eqn:comparaison-facteur-transfert}, donc finit la démonstration de \ref{prop:calcul-qX}.
\end{proof}

\subsection{La formule via le modèle latticiel}
Supposons désormais $F$ non archimédien de caractéristique résiduelle $p>2$ et fixons un $\mathfrak{o}_F$-réseau $L \subset W$ tel que $L=L^\perp$. Posons $K := \text{Stab}_{\Sp(W)}(L)$. Le modèle latticiel $(\rho_L, S_L)$ permet d'identifier $K$ à un sous-groupe compact ouvert de $\MMp{2}(W)$. On va étudier le comportement de $\Theta_\psi$ sur $K$.

Remarquons d'abord que le relèvement de $-1 \in K$ à $\Spt_\psi(W)$ est égal à l'élément $-1$ défini dans \ref{def:-1} d'après \ref{prop:Levi-latticiel-compatible}. Par conséquent, l'équation
$$\Theta_\psi(-x) = \Theta_\psi^+(x)-\Theta_\psi^-(x),\; x \in \Sp(W)^\ddagger \cap K $$
est vraie en deux sens: on peut considérer $-x$ comme un élément de $K$, relevé à $\MMp{2}(W)$ au moyen du modèle latticiel, ou au sens de \ref{prop:-1-echangement}. Ceci justifie notre abus de notations.

\begin{lemma}\label{prop:formule-caractere-K-1}
  Pour $x \in K \cap \Sp(W)^\dagger$, on a
  $$ \Theta_\psi(x) = \sum_{\substack{\dot{w} \in W/L \\ (x-1)w \in L}} \psi\left(\frac{\angles{x(w)|w}}{2} \right) ,$$
  qui est une somme finie si l'on fixe $x$.
\end{lemma}
\begin{proof}
  Prenons un système de représentants $R \subset W$ de $W/L$ et la base orthonormée $\{f_r\}_{r \in R}$ de $S_L$ dans \S\ref{sec:modele-latticiel}. Soit $\phi \in C_c^\infty(K \cap \Sp(W)^\dagger)$ quelconque. Pour $r \in R$, on a
  \begin{equation*}\begin{aligned}
    (\omega_\psi(\phi)f_r)(r,0) & = \int_{K \cap \Sp(W)^\dagger} \phi(x) (\omega_\psi(x)f_r)(r,0) \dd x \\
    & = \int_{K \cap \Sp(W)^\dagger} \phi(x) \mathbbm{1}_{E_r}(x) \psi\left(\frac{\angles{r|x^{-1}(r)}}{2} \right) \dd x \\
    & (E_r := \{x \in K: x^{-1}(r)-r \in L \}, \text{qui est ouvert}.)
  \end{aligned}\end{equation*}
  Si $x$ est fixé, la somme sur tout $r \in R$ du terme intérieur est finie et bornée par 
  $$|\phi(x)| \cdot \#((x^{-1}-1)^{-1}L/L) = |\phi(x)| \cdot |\det(x^{-1}-1)|^{-1},$$
  qui est uniformément borné dans $K \cap \Sp(W)^\dagger$. En utilisant le théorème de convergence dominée, on obtient
  \begin{equation*}\begin{aligned}
    \Theta_\psi(\phi) & = \sum_{r \in R} \;\int_{K \cap \Sp(W)^\dagger} \phi(x) \mathbbm{1}_{E_r}(x) \psi\left(\frac{\angles{r|x^{-1}(r)}}{2} \right) \dd x \\
    & = \int_{K \cap \Sp(W)^\dagger} \phi(x) \sum_{r \in R} \mathbbm{1}_{E_r}(x) \psi\left( \frac{\angles{r|x^{-1}(r)}}{2} \right) \dd x \\
    & = \int_{K \cap \Sp(W)^\dagger} \phi(x) \sum_{\substack{r \in R \\ (x^{-1}-1)r \in L}} \psi\left(\frac{\angles{r|x^{-1}(r)}}{2} \right) \dd x
  \end{aligned}\end{equation*}
  Le terme dans la somme ne dépend que de la classe de $r$ mod $L$. On a $\angles{r| x^{-1}(r)} = \angles{x(r)|r}$, et $(x^{-1}-1)r \in L$ si et seulement si $(x-1)r \in L$ car $x \in K$. D'où le résultat cherché.
\end{proof}

\begin{proposition}\label{prop:formule-caractere-K-2}\label{prop:-1-latticiel}
  Soit $x \in K \cap \Sp(W)^\dagger$. On a:
  \begin{enumerate}
    \item si $(x-1)(L)=L$, alors
      $$ \Theta_\psi(x)=1;$$
      la condition est satisfaite lorsque $-x$ est topologiquement unipotent;
    \item si $x$ topologiquement unipotent, alors $x \in \Sp(W)^\ddagger$ et
      $$ \Theta_\psi(x) = |\det(x-1)|^{-1/2} \cdot \gamma_\psi(q[C_x]).$$
  \end{enumerate}
\end{proposition}
\begin{proof}
  Montrons la première assertion. Si $(x-1)L=L$, la somme dans \ref{prop:formule-caractere-K-1} porte seulement sur l'élément $0$, donc $\Theta_\psi(x)=1$.
  
  Supposons que $-x$ est topologiquement unipotent. Soit $\bar{x}$ l'image de $x$ dans $\End_{\mathbb{F}_q}(L/\mathfrak{p}_F L)$. On a $(-\bar{x})^{q^m} = 1$ pourvu que $m$ soit assez grand. Puisque $q$ est impair, il en résulte que $-1$ n'est pas une valeur propre de $-\bar{x}$. Autrement dit: $(x-1)L=L$. Cela démontre la première assertion.

  Montrons la dernière assertion. Soit $x \in \Sp(W)^\dagger \cap K$ topologiquement unipotent. L'argument pour la première assertion montre que $(x+1)L=L$; en particulier $|\det(x+1)|=1$ et donc $x \in \Sp(W)^\ddagger$. Par \ref{prop:formule-caractere-Thomas2} et \ref{prop:-1-echangement}, on en déduit
  \begin{align*}
    \Theta_\psi(x) & = \gamma_\psi(q[C_x])|\det(x-1)|^{-1/2} (\Theta_\psi^+ - \Theta_\psi^-)(-x) \\
    & = \gamma_\psi(q[C_x]) |\det(x-1)|^{-1/2} \Theta_\psi(-x).
  \end{align*}

  Or $\Theta_\psi(-x)=1$ par la première assertion. D'où $\Theta_\psi(x) = |\det(x-1)|^{-1/2} \gamma_\psi(q[C_x])$.
\end{proof}

\begin{corollary}\label{prop:Theta-reduction-reguliere}
  Soit $x \in K \cap \Sp(W)^\dagger$ tel que sa réduction $\bar{x} \in \Sp(L/\mathfrak{p}L)$ est régulière, alors
  $$ \Theta_\psi(x) = \Theta_\psi(-x) = 1. $$
\end{corollary}
\begin{proof}
  Si $\bar{x}$ est régulière, alors $(x \pm 1)L=L$.
\end{proof}
Dans ce cas, on dit que $x$ est de réduction régulière\index{réduction régulière} par rapport à $L$.

\subsection{Formules sur l'algèbre de Lie}
On peut obtenir une formule pour $\Theta_\psi$ sur l'algèbre de Lie, ce qui aura un intérêt indépendant. Soit $X \in \syp(W)$ inversible. On fait l'une des hypothèses suivantes:
\begin{enumerate}
  \item l'élément $X$ est suffisamment proche de $0$,
  \item $F$ est non archimédien de caractéristique résiduelle $p$ assez grande par rapport à $n$ (\cite{Wa08} 4.3), $p>2$ et $X$ est topologiquement nilpotent.
\end{enumerate}
Alors il est loisible de définir $\exp(X)$ et $\exp(X) \in \Sp(W)^\ddagger$.

\begin{lemma}\label{prop:q_X-est-q_x}
  Supposons que $X$ vérifie l'une des hypothèses ci-dessus. Soit $x=\exp(X) \in \Sp(W)^\ddagger$, alors
  $$ q[C_x] \simeq q[X]. $$
\end{lemma}
\begin{proof}
  Il suffit de considérer le cas $X$ semi-simple; le cas général en résultera par continuité. Soit $A$ la sous-algèbre commutative de $\End(W)$ engendrée par $X$. Pour $a \in A$, considérons son rayon spectral
  $$ \|a\| := \sup \{|\lambda|_F : \lambda \text{ est une valeur propre de } a \}. $$
  C'est une norme sur $A$ car $X$ est semi-simple; $A$ est complète par rapport à $\|\cdot\|$.

  On cherche une série formelle $P(T) \in F[[T]]$ dont le rayon de convergence est $> \|X\|$ telle que $P(X)$ est inversible et $q[2\cdot\frac{x-1}{x+1}](v|w) = \angles{XP(X)v|P(X)w}$ pour tout $u,v \in W$. Cela équivaut à
  $$ P(X)P(-X)X = 2 \cdot \dfrac{\exp(X)-1}{\exp(X)+1},$$
  ou
  $$ P(X)P(-X) = \frac{2}{X} \cdot \dfrac{\exp(X)-1}{\exp(X)+1}.$$
  Le terme à droite est de la forme $1+Q(X)$, où $Q(T) \in F[[T^2]]$ a un rayon de convergence $> \|X\|$ et $T|Q(T)$. Cela nous permet de définir
  $$ P(T) := \exp\left(\frac{1}{2}\log(1+Q(T))\right) \in F[[T^2]]$$
  Cette série formelle a aussi un rayon de convergence $> \|X\|$. Donc $P(X)$ est inversible car il appartient à l'image de l'exponentielle. On a
  $$P(T)P(-T) = P(T)^2 = 1+ Q(T) = \frac{2}{T} \cdot \dfrac{\exp(T)-1}{\exp(T)+1}.$$
  On remplace $T$ par $X$ et cela permet de conclure.
\end{proof}

\begin{corollary}\label{prop:formule-caractere-Lie}
  Avec les mêmes hypothèses, il existe une constante $c$ indépendante de $X$ telle que
  $$ \Theta_\psi(\tilde{x}) = \frac{c \cdot \gamma_\psi(q[X])}{|\det(x-1)|^{1/2}}. $$
  où $\tilde{x} := \exp(X)$. Si l'hypothèse (2) est satisfaite, on a $c=1$.
\end{corollary}
\begin{proof}
  D'après \ref{prop:formule-caractere-Thomas2}, \ref{prop:-1-echangement} et le lemme,
  \begin{align*}
    \Theta_\psi(\tilde{x}) &= \left|\frac{\det(x+1)}{\det(x-1)}\right|^{1/2} \gamma_\psi(q[C_x]) (\Theta_\psi^+ - \Theta_\psi^-)(\tilde{x}) \\
    &= \left|\frac{\det(x+1)}{\det(x-1)}\right|^{1/2} \gamma_\psi(q[X]) \Theta_\psi(-\tilde{x}).
  \end{align*}
  Posons $c := |2|^n \Theta_\psi(-1)$. Lorsque $X$ est assez petit, $\Theta_\psi(-\tilde{x}) = |\det(x+1)|^{-1/2} c$. Dans le cas où l'hypothèse (2) est satisfaite, on a $c = 1$. D'où le corollaire.
\end{proof}

\subsection{Décompositions}
Fixons $\mathbf{f} \in 2\Z_{\geq 1}$. Supposons qu'il y a une décomposition orthogonale d'espaces symplectiques
$$W = W_1 \oplus \cdots \oplus W_r,$$
Alors il existe un homomorphisme canonique
$$ \iota: \prod_{k=k}^r \Sp(W_k) \to \Sp(W). $$

On en déduit un homomorphisme recouvrant $\iota$,
$$ j: \prod_{k=1}^r \MMp{\mathbf{f}}(W_k) \to \MMp{\mathbf{f}}(W) $$
telle que $j^{-1}(\MMp{\mathbf{f}}(W)^\dagger) = \prod_{k=1}^r \MMp{\mathbf{f}}(W_k)^\dagger$. Pour $1 \leq k \leq r$, notons $\Theta_\psi^{[k]}$ le caractère de la représentation de Weil de $\MMp{\mathbf{f}}(W_k)$ attachée à $\psi$; il est lisse sur $\MMp{\mathbf{f}}(W_k)^\dagger$. Le résultat découle de l'additivité symplectique de l'indice de Maslov.

\begin{proposition}\label{prop:caractere-decomposition-0}
  Soient $\tilde{x}_k \in \MMp{\mathbf{f}}(W_k)^\dagger$ pour $1 \leq k \leq r$, alors
  $$ \Theta_\psi(j(\tilde{x}_1,\ldots,\tilde{x}_r)) = \prod_{k=1}^r \Theta_\psi^{[k]}(\tilde{x}_k).$$
\end{proposition}

Supposons maintenant $F$ non archimédien de caractéristique résiduelle $p>2$. Fixons un réseau $L \subset W$ tel que $L=L^\perp$. Soient $L_k := W_k \cap L$. Supposons que $L= \bigoplus_{k=1}^r L_k$, alors $L_k^\perp = L_k$ dans $W_k$. Posons $K_k := \text{Stab}_{\Sp(W_k)}(L_k)$.

Soit $x \in \Sp(W)^\dagger \cap K$ tel que $x(L_k) \subset L_k$ pour tout $k$. Il existe alors $x_k \in \Sp(W_k)^\dagger \cap K_k$ pour tout $k$ tel que $\iota(x_1,\ldots,x_r)=x$. On regarde $x$ (resp. les $x_k$) comme un élément de $\MMp{2}(W)$ (resp. $\MMp{2}(W_k)$) à l'aide du modèle latticiel associé à $L$ (resp. $L_k$). La proposition suivant résulte immédiatement de \ref{prop:formule-caractere-K-1}.

\begin{proposition}
  Avec les hypothèses ci-dessus, on a
  $$ \Theta_\psi(x) = \prod_{k=1}^r \Theta_\psi^{[k]}(x_k).$$
\end{proposition}

\subsection{La formule du produit}
Dans cette section on se place dans le cas global. Fixons $\mathbf{f} \in 2\Z_{\geq 1}$. Rappelons que l'on a une immersion canonique $i: \Sp(W,F) \to \MMp{2}(W,\A) \subset \MMp{\mathbf{f}}(W,\A)$. On écrit $i(x)=[\tilde{x}_v
]_v$ où $\tilde{x}_v \in \MMp{\mathbf{f}}(W_v)$.

\begin{definition}\label{def:localement-geometriquement-conj}
  On dit que deux éléments $x = (x_v)_v$, $y=(y_v)_v$ dans $\Sp(W,\A)$ sont localement géométriquement conjugués si pour toute place $v$, $x_v$ et $y_v$ dans $\Sp(W,F_v)$ sont géométriquement conjugués.
\end{definition}

\begin{theorem}\label{prop:caractere-formule-produit}
  Soient $x \in \Sp(W,F)$ et $\tilde{y}=[\tilde{y}_v] \in \MMp{\mathbf{f}}(W,\A)$ tels que $y$ est localement géométriquement conjugué à $x$. Supposons que $\det(x-1) \neq 0$. Alors $\Theta_{\psi_v}(\tilde{y}_v)=1$ pour presque tout $v$, et le produit $\prod_v \Theta_{\psi_v}(\tilde{y}_v)$ est bien défini et à valeurs dans $\bmu_\mathbf{f}$. On a aussi
  $$ \prod_v \Theta_{\psi_v}(\tilde{x}_v) = 1.$$

  De même, si $\det(x+1) \neq 0$ alors les assertions ci-dessus restent valables avec $(\Theta^+_{\psi_v}-\Theta^-_{\psi_v})$ au lieu de $\Theta_{\psi_v}$.
\end{theorem}
\begin{proof}
  Traitons d'abord le cas $\det(x-1) \neq 0$. Pour presque toute place finie $v$ on a $L_v=L_v^\perp$, $\tilde{y}_v \in K_v$ et $(y_v - 1)(L_v)=L_v$, alors \ref{prop:formule-caractere-K-2} entraîne que $ \Theta_{\psi_v}(\tilde{y}_v) = 1$. Puisque les distributions locales $\Theta_{\psi_v}$ sont spécifiques, le produit $\prod_v \Theta_{\psi_v}(\tilde{y}_v)$ ne dépend pas du choix des $\tilde{y}_v$.

  Fixons un lagrangien $\ell \in \Lag(W)$. On plonge $\MMp{\mathbf{f}}(W,\A)$ dans $\Spt_\psi(W,\A)$, alors
  $$ i(x) = (x, \bigotimes_v M_{\ell_v}[x]) $$
  d'après \ref{prop:relevement-rationnel-explicite}. La formule \ref{prop:formule-caractere-Spt} entraîne que
  $$ \prod_v \Theta_{\psi_v}(\tilde{x}_v) = \prod_v \frac{\gamma_{\psi_v}(\tau(\Gamma_x, \Gamma_1, \ell\oplus\ell))}{|\det(x-1)|_v^{1/2}}. $$
  Or $\prod_v |\det(x-1)|_v = 1$, et la réciprocité de Weil entraîne que
  $$ \prod_v \gamma_{\psi_v}(\tau(\Gamma_x, \Gamma_1, \ell\oplus\ell)) = 1 $$
  car l'espace quadratique $\tau(\Gamma_x, \Gamma_1, \ell\oplus\ell)$ est défini sur $F$. D'où la formule du produit de la première assertion. Soit $\tilde{y}=[\tilde{y}_v] \in \Sp(W,\A)$ tel que $y$ est localement géométriquement conjugué à $x$, on a $\Theta_{\psi_v}(\tilde{x}_v) \Theta_{\psi_v}(\tilde{y}_v)^{-1} \in \bmu_\mathbf{f}$ pour toute place $v$ d'après \ref{prop:caractere-stable}, cela démontre le reste.

  Le cas $\det(x+1) \neq 0$ résulte du cas précédent, de \ref{prop:-1-echangement} et de \ref{prop:-1-relevement}.
\end{proof}

\section{Endoscopie}
Désormais, fixons $\mathbf{f} \in 2\Z_{\geq 1}$ et posons $\Mp(W) = \MMp{\mathbf{f}}(W)$\index{$\Mp(W)$} pour $W$ un $F$-espace symplectique où $F$ est un corps local. De même, posons $\Mp(W,\A) = \MMp{\mathbf{f}}(W,\A)$ si $F$ est un corps global.

\subsection{Données endoscopiques elliptiques}\index{données endoscopique elliptique}\label{sec:donnees-endoscopiques}
Fixons $F$ un corps local ou global de caractéristique $0$ et fixons $(W,\angles{\cdot|\cdot})$ un $F$-espace symplectique de dimension $2n$, $n \in \Z_{\geq 0}$. Soient $n', n'' \in \Z_{\geq 0}$, $n'+n''=n$. Nous étudierons les groupes

\begin{align*}
  G & :=\Sp(W), \\
  H & = H_{n',n''} := H' \times H'', \quad\text{où} \\
  H' &:= \SO(2n'+1), \\
  H'' & := \SO(2n''+1);\\
  \tilde{G} &:= \begin{cases} \Mp(W), & \text{lorsque } F \text{ est local}, \\
                  \Mp(W,\A), & \quad \text{lorsque } F \text{ est global}.
                \end{cases}
\end{align*}
Ici $\SO(2m+1)$ ($m=n'$ ou $n''$) signifie le groupe orthogonal impair déployé associé à la forme quadratique sur $F^{2m+1}$:
$$ (x_{-m}, \ldots, x_0, \ldots, x_m) \mapsto 2\sum_{i=1}^m x_i x_{-i} + x_0^2 .$$

Notons $\rev: \tilde{G} \to G(F)$ (resp. $\rev: \tilde{G} \to G(\A)$) le revêtement métaplectique si $F$ est local (resp. global).

Une donnée endoscopique elliptique de $\tilde{G}$ est une paire $(n',n'')$ comme ci-dessus. Désormais, fixons une telle donnée $(n',n'')$. Le groupe $H := H_{n',n''}$ est le groupe endoscopique elliptique associé à $(n',n'')$. Spécifions maintenant la correspondance de classes de conjugaison géométriques semi-simples.

Fixons des $F$-tores maximaux déployés $S' \subset H'$, $S'' \subset H''$ et posons $S = S' \times S''$, c'est un $F$-tore maximal déployé dans $H$. Fixons aussi des $F$-tores maximaux déployés $T' \subset \Sp(2n')$, $T'' \subset \Sp(2n'')$ et $T \subset \Sp(W)$. Ces objets sont uniques à $F$-conjugaison près.

Notons par $W^{H'}(S')$, $W^{H''}(S'')$, $W^{\Sp(2n')}(T')$, $W^{\Sp(2n'')}(T'')$, $W^G(T)$ les groupes de Weyl associés à ces tores maximaux, alors $W^H(S) = W^{H'}(S') \times W^{H''}(S'')$. Il y a des homomorphismes canoniques
\begin{gather*}
  W^{H'}(S') \rightiso W^{\Sp(2n')}(T'), \\
  W^{H''}(S'') \rightiso W^{\Sp(2n'')}(T''), \\
  W^{\Sp(2n')}(T') \times W^{\Sp(2n'')}(T'') \hookrightarrow W^{\Sp(2n)}(T),
\end{gather*}
et des $F$-isomorphismes qui respectent les homomorphismes ci-dessus
\begin{align*}
  \mu': & S' \rightiso T' ,\\
  \mu'': & S'' \rightiso T'', \\
  \nu: & T' \times T'' \rightiso T .
\end{align*}

On obtient ainsi des $F$-morphismes, notés par les mêmes lettres:
\begin{align*}
  \mu': & S'/W^{H'}(S') \rightiso T'/W^{\Sp(2n')}(T'), \\
  \mu'': & S''/W^{H''}(S'') \rightiso T''/W^{\Sp(2n'')}(T''), \\
  \nu: & T'/W^{\Sp(2n')}(T') \times T''/W^{\Sp(2n'')}(T'') \to T/W^G(T).
\end{align*}

Posons
\begin{equation}\label{eqn:mu-0}
  \mu = \mu_{n',n''} := \nu \circ (\identity, -\identity) \circ (\mu',\mu''): S/W^{H}(S) \to T/W^G(T).
\end{equation}

Puisque $G$ est simple et simplement connexe, $\mu$ donne naissance à une application, notée encore par $\mu$:

\begin{equation}\label{eqn:mu}
  \mu: \Cssgeo(H(F)) \to \Cssgeo(G(F)).
\end{equation}
Cette application ne dépend pas du choix de $F$-tores maximaux. De plus, elle est à fibres finies car $\mu',\mu'', \nu$ le sont.

\begin{definition}\index{correspondance des classes de conjugaison semi-simples}
  Si $\delta \in G(F)_\text{ss}$, $\gamma \in H(F)_\text{ss}$ sont tels que $\mu(\mathcal{O}^\text{geo}(\gamma)) = \mathcal{O}^\text{geo}(\delta)$, on dit que $\delta$ et $\gamma$ se correspondent.
\end{definition}

\begin{remark}
  Lorsque $n'=0$ ou $n''=0$,  $\mu$ est bijectif. C'est démontré dans \cite{Ho07} pour $F$ un corps local, mais l'argument marche aussi sur un corps global.
\end{remark}

Explicitons l'application $\mu$ en termes de paramètres.

\begin{proposition}
  Soit $\gamma \in H(F)_\text{ss} = H'(F)_\text{ss} \times H''(F)_\text{ss}$ tel que
$$ \gamma \in \mathcal{O}((K'/K'^\#, a', (V_{K'},h_{K'}), (V'_\pm, q'_\pm)) \oplus (K''/K''^\#, a'', (V_{K''},h_{K''}), (V''_\pm, q''_\pm))), $$
  alors un élément $\delta \in G(F)_\text{ss}$ lui correspond si et seulement si $\mathcal{O}(\delta)$ est paramétré par
  $$ \mathcal{O}(K/K^\#, (a',-a''), (W_K, h_K), (W_\pm,\angles{\cdot|\cdot}_\pm)) $$
  où $K := K' \times K''$ comme $F$-algèbres étales à involution, et les autres données sont soumises aux conditions
  \begin{gather*}
    W_K \simeq V_{K'} \oplus V_{K''} \; \text{ comme } K-\text{modules}, \\
    \dim_F W_+ + 1 = \dim_F V'_+ + \dim_F V''_- , \\
    \dim_F W_- + 1 = \dim_F V'_- + \dim_F V''_+ .
  \end{gather*}
\end{proposition}
\begin{proof}
  L'application $\mu'$ induit une application $\Cssgeo(H'(F)) \to \Cssgeo(\Sp(2n',F))$ qui envoie une classe de valeurs propres
  $$ a'_1, \ldots, a'_{n'}, 1, (a'_{n'})^{-1}, \ldots, (a'_1)^{-1} $$
  sur une classe de valeurs propres
  $$ a'_1, \ldots, a'_{n'}, (a'_{n'})^{-1}, \ldots, (a'_1)^{-1}. $$
  Il en est de même pour $\mu''$ avec $n''$ au lieu de $n'$.

  D'autre part $\nu$ induit une application $\Cssgeo(\Sp(2n',F)) \times \Cssgeo(\Sp(2n'',F)) \to \Cssgeo(G(F))$ préservant les valeurs propres. L'assertion en résulte par la construction de $\mu$ \eqref{eqn:mu-0}, \eqref{eqn:mu}.
\end{proof}

\begin{definition}\index{$G$-régulier}
  On dit qu'un élément $\gamma \in H(F)_\text{ss}$ est $G$-régulier si $\mu(\mathcal{O}^\text{st}(\gamma))$ est régulier. On note l'ouvert de Zariski des éléments $G$-réguliers par $H_{G-\text{reg}}$.
\end{definition}

Si $\gamma=(\gamma',\gamma'')$ est $G$-régulier, alors $\gamma'$ et $\gamma''$ sont assez réguliers.

\begin{corollary}
  Deux éléments $\gamma=(\gamma',\gamma'') \in H_{G-\text{reg}}(F)$ et $\delta \in G_\text{reg}(F)$ se correspondent si et seulement s'ils admettent des paramètres de la forme suivante
  \begin{align*}
     \gamma' & \in \mathcal{O}(K'/K'^\#, a', c') \\
     \gamma'' & \in \mathcal{O}(K''/K''^\#, a'', c'') \\
     \delta & \in \mathcal{O}(K/K^\#, (a',-a''), c)
  \end{align*}
  où $K=K' \times K''$ comme $F$-algèbres étales à involution. Il n'y a pas de restrictions sur les données $c',c'', c$.
\end{corollary}

Enfin, la formation de l'application $\mu$ est compatible aux complétions.

\subsection{Une notion de stabilité}\label{sec:stabilite}
Supposons que $F$ est un corps local. Soient $\delta \in G_\text{reg}(F)$ et $T$ l'unique $F$-tore maximal contenant $\delta$. Alors $\mathcal{D}(\delta) := \mathcal{O}^\text{st}(\delta)/\text{conj}$ est un torseur sous $H^1(F,T)$. Explicitons cette action. Si $\mathcal{O}(\delta)$ est paramétré par $(K/K^\#,a,c)$, $K=\prod_{i \in I} K_i$, alors on a des isomorphismes canoniques
$$ H^1(F,T) \simeq K^{\#\times}/N_{K/K^\#}(K^\times) \xrightarrow[(\sgn_i)_{i \in I}]{\sim} \bmu_2^{I_*} $$
où $\sgn_i := \sgn_{K_i/K_i^\#}$. D'autre part $\mathcal{O}^\text{st}(\delta)/\text{conj}$ est isomorphe au même ensemble, sur lequel $H^1(F,T)$ agit par multiplication.

\begin{definition}\index{conjugaison stable pour le groupe métaplectique}
  Soient $\tilde{\delta}_1, \tilde{\delta}_2 \in \tilde{G}_\text{reg}$, on dit qu'ils sont stablement conjugués si leurs images dans $G(F)$ sont stablement conjugués et si
  $$(\Theta_\psi^+ - \Theta_\psi^-)(\tilde{\delta}_1) = (\Theta_\psi^+ - \Theta_\psi^-)(\tilde{\delta}_2).$$

  Soit $\tilde{\delta} \in \tilde{G}_\text{reg}$, notons $\mathcal{O}^\text{st}(\tilde{\delta})$ l'ensemble d'éléments dans $\tilde{G}_\text{reg}$ stablement conjugués à $\tilde{\delta}$. On définit $\mathcal{D}(\tilde{\delta}) := \mathcal{O}^\text{st}(\tilde{\delta})/\text{conj}$.
\end{definition}

Pour $F=\R$, cette notion \textit{ad hoc} coïncide avec celle d'Adams (\cite{Ad98} 8.10) si l'on se restreint à $\MMp{2}(W)$. Pour $F=\C$, on identifie $\tilde{G}$ à $\Ker(\rev) \times G(\C)$. Deux éléments réguliers $(\noyau_1, \delta_1), (\noyau_2, \delta_2)$ sont stablement conjugués si et seulement si $\delta_1, \delta_2$ le sont et $\noyau_1 = \noyau_2$, d'après \ref{rem:caractere-C}.

\begin{lemma}
  Soit $\tilde{\delta} \in \tilde{G}_\text{reg}$, alors $\rev: \tilde{G} \to G(F)$ induit une bijection $\mathcal{D}(\tilde{\delta}) \to \mathcal{D}(\delta)$.
\end{lemma}
\begin{proof}
  Soit $\delta_1 \in \mathcal{O}^\text{st}(\delta)$ et $\tilde{\delta}_1 \in \rev^{-1}(\delta_1)$ quelconque. D'après \ref{prop:caractere-stable}, il existe $\noyau \in \Ker(\rev)$ tel que
  $$ (\Theta_\psi^+ - \Theta_\psi^-)(\tilde{\delta}_1) = \noyau \cdot (\Theta_\psi^+ - \Theta_\psi^-)(\tilde{\delta}). $$

  Or $\Theta_\psi^+ - \Theta_\psi^-$ est spécifique, donc il existe un unique $\tilde{\delta}_1 \in \rev^{-1}(\delta_1)$ stablement conjugué à $\tilde{\delta}$. Cela permet de conclure.
\end{proof}

Soient $\tilde{\delta}, \tilde{\delta}_1$ stablement conjugués. Notons $T$ le $F$-tore maximal contenant $\delta$, on pose $\inv(\tilde{\delta},\tilde{\delta}_1) := \inv(\delta,\delta_1)$; c'est l'unique élément $c \in H^1(F,T)$ tel que $c \cdot \delta = \delta_1$.

\begin{remark}\label{rem:stabilite-compatibilite}
  La notion de stabilité est compatible avec la restriction: si $\mathbf{f},\mathbf{f}' \in 2\Z_{\geq 1}$, $\mathbf{f}|\mathbf{f}'$ et $\tilde{\delta}, \tilde{\delta}_1 \in \MMp{\mathbf{f}}(W)_\text{reg}$, alors $\tilde{\delta}$ et $\tilde{\delta}_1$ sont stablement conjugués dans $\MMp{\mathbf{f}}(W)$ si et seulement s'ils sont stablement conjugués dans $\MMp{\mathbf{f}'}(W)$.
\end{remark}

\subsection{Facteur de transfert}\label{sec:facteur-de-transfert}
Sauf mention expresse du contraire, dans cette section $F$ est toujours un corps local.

Soient $\tilde{\delta} \in \tilde{G}_\text{reg}$, $\gamma=(\gamma',\gamma'') \in H_{G-\text{reg}}(F)$ tels que $\delta$ et $\gamma$ se correspondent. Supposons que
\begin{gather*}
  \gamma' \in \mathcal{O}(K'/K'^\#, a', c'), \\
  \gamma'' \in \mathcal{O}(K''/K''^\#, a'', c'');
\end{gather*}

alors on a une unique décomposition orthogonale $W=W' \oplus W''$ stable par $\delta$ telle que si l'on pose $\delta' := \delta|_{W'}$, $\delta'' := \delta|_{W''}$, alors
\begin{gather*}
  \delta' \in \mathcal{O}(K'/K'^\#, a', c') \\
  \delta'' \in \mathcal{O}(K''/K''^\#, -a'', c'').
\end{gather*}

On a un diagramme commutatif associé à $W=W'\oplus W''$:
$$\xymatrix{
  \Mp(W') \times \Mp(W'') \ar[rr]^{j} \ar[d]_{(\rev',\rev'')} & & \Mp(W) \ar[d]^{p} \\
  \Sp(W') \times \Sp(W'') \ar@{^{(}->}[rr]_{\iota} & & \Sp(W) \\
}$$

et il existe $\tilde{\delta}' \in \Mp(W')$, $\tilde{\delta}'' \in \Mp(W'')$ tels que $j(\tilde{\delta}',\tilde{\delta}'')=\tilde{\delta}$, $\iota(\delta',\delta'')=\delta$. La paire $(\delta', \delta'')$ est unique et appartient à $\Sp(W')_\text{reg} \times \Sp(W'')_\text{reg}$. Par contre la paire $(\tilde{\delta}', \tilde{\delta}'')$ est unique à multiplication par $\{(\noyau,\noyau^{-1}) : \noyau \in \Ker(\rev) \}$ près.

\begin{definition}\index{$\Delta, \Delta', \Delta'', \Delta_0$}
  Avec les hypothèses ci-dessus, définissons
  \begin{align*}
    \Delta'(\tilde{\delta}') & := \frac{{\Theta'}_\psi^+ - {\Theta'}_\psi^-}{|{\Theta'}_\psi^+ - {\Theta'}_\psi^-|}(\tilde{\delta}'), \\
    \Delta''(\tilde{\delta}'') & := \frac{{\Theta''}_\psi^+ +{\Theta''}_\psi^-}{|{\Theta''}_\psi^+ + {\Theta''}_\psi^-|}(\tilde{\delta}''), \\
    \Delta_0(\delta', \delta'') & := \sgn_{K''/K''^\#}(P_{a'}(a'')(-a'')^{-n'} \det(\delta'+1)),
  \end{align*}
  où ${\Theta'}_\psi^\pm$ (resp. ${\Theta''}_\psi^\pm$) sont les caractères définis sur $\Mp(W')$ (resp. $\Mp(W'')$), et $P_{a'}(T) \in F[T]$ est le polynôme caractéristique de $a' \in K'^\times$ (qui est aussi le celui de $\delta' \in \End_F(W')$). Remarquons que $P_{a'}(a'')(-a'')^{-n'} \in K''^{\#\times}$. En effet, $P_{a'}(T)(T)^{-n'} \in F[T+T^{-1}]$ car $\tau(a')=a'^{-1}$ où $\tau$ est l'involution de $K'$, et la régularité de $\delta$ entraîne que $P_{a'}(a'') \neq 0$.

  Le facteur de transfert\index{facteur de transfert!pour le groupe métaplectique} est défini par
  $$ \Delta(\gamma, \tilde{\delta}) := \Delta_0(\delta',\delta'') \Delta'(\tilde{\delta}') \Delta''(\tilde{\delta}''). $$

  Comme $\Delta', \Delta''$ sont des distributions spécifiques, $\Delta(\gamma, \tilde{\delta})$ est bien défini. Ce ne dépend que de $\mathcal{O}^\text{st}(\gamma)$ et $\mathcal{O}(\tilde{\delta})$, et le terme $\Delta_0$ ne dépend que de $\mathcal{O}^\text{st}(\gamma)$.

  Définissons $\Delta(\gamma, \tilde{\delta})=0$ si $\gamma$ et $\delta$ ne se correspondent pas. Cela définit une fonction $\Delta$ sur $H_{G-\text{reg}}(F) \times \tilde{G}_\text{reg}$.
\end{definition}

\begin{remark}
  Le terme $\Delta_0$ est trivial lorsque $n'=0$ ou $n''=0$. Lorsque $F=\R$ et $n''=0$, $\Delta$ coïncide avec le facteur de transfert défini par Adams \cite{Ad98}, quitte à se restreindre sur $\MMp{2}(W)$. Lorsque $F=\C$, $\Delta$ est trivial.
\end{remark}

Déduisons d'autres formules utiles pour $\Delta$.
\begin{proposition}\label{prop:Delta-variante}
  On a aussi
  \begin{align*}
    \Delta(\gamma, \tilde{\delta}) &= \Delta_0(\delta',\delta'') \cdot  \frac{\Theta_\psi^+ - \Theta_\psi^-}{|\Theta_\psi^+ - \Theta_\psi^-|}(\tilde{\delta}) \cdot \gamma_\psi(q[C_{\delta''}]) \\
    & = \Delta_0(\delta',\delta'') \cdot  \frac{\Theta_\psi^+ + \Theta_\psi^-}{|\Theta_\psi^+ + \Theta_\psi^-|}(\tilde{\delta}) \cdot \gamma_\psi(q[C_{\delta'}])^{-1}.
  \end{align*}
\end{proposition}
\begin{proof}
  Cela résulte de \ref{prop:formule-caractere-Thomas2} et \ref{prop:caractere-decomposition-0}.
\end{proof}

\begin{proposition}[Spécificité]\label{prop:facteur-specificite}
  $\Delta$ est spécifique au sens suivant
$$ \forall \noyau \in \Ker(\rev),\quad \Delta(\gamma, \noyau\tilde{\delta}) = \noyau\Delta(\gamma, \tilde{\delta}). $$
\end{proposition}
\begin{proof}
  On utilise la première formule de \ref{prop:Delta-variante}. Les termes $\Delta_0(\delta',\delta'')$ et $\gamma_\psi(q[C_{\delta''}])$ ne dépendent que de $\delta$, pourtant le terme $(\Theta_\psi^+ - \Theta_\psi^-)(\tilde{\delta})/|\cdots|$ est spécifique. Cela permet de conclure.
\end{proof}

Maintenant soit $T$ le $F$-tore maximal contenant $\delta$, et $T=T' \times T''$ la décomposition correspondant à $\delta=i(\delta',\delta'')$. Décomposons les $F$-algèbres étales dans les paramètres pour $\delta', \delta''$ comme
\begin{align*}
  K' &= \prod_{i \in I'} K'_i , \\
  K'' &= \prod_{i \in I''} K''_i .
\end{align*}

Cela nous permet d'écrire
\begin{align*}
  H^1(F,T') & =  \bmu_2^{I'^*}, \\
  H^1(F,T'') & = \bmu_2^{I''^*}.
\end{align*}

Comme l'endoscopie pour les groupes réductifs, on dispose de la notion du caractère endoscopique\index{caractère endoscopique}\index{$\kappa$} $\kappa = \kappa_T: H^1(F,T) \to \bmu_2$. Vu l'identification ci-dessus, c'est défini par
\begin{align*}
  \kappa: \bmu_2^{I'^*} \times \bmu_2^{I''^*} & \longrightarrow \bmu_2 \\
  ((t'_i)_{i \in I'^*}, (t''_i)_{i \in I''^*}) & \longmapsto  \prod_{i \in I''^*} t''_i .
\end{align*}

\begin{proposition}[Propriété de cocycle]\label{prop:propriete-cocycle}
  Conservons le formalisme ci-dessus. Si $\tilde{\delta}$ est stablement conjugué à $\tilde{\delta}_1$, alors
  $$ \Delta(\gamma, \tilde{\delta}_1) = \angles{\kappa, \inv(\tilde{\delta}, \tilde{\delta}_1)} \Delta(\gamma, \tilde{\delta}).$$
\end{proposition}
\begin{proof}
  On utilise la première formule de \ref{prop:Delta-variante}. Le terme $\Delta_0$ ne dépend que de $(K'/K'^\#,a')$ et $(K''/K''^\#,a'')$, donc la stabilité n'y intervient pas. D'autre part $\Theta_\psi^+ - \Theta_\psi^-$ est constante sur une classe de conjugaison stable par définition. Il suffit de traiter le terme $\gamma_\psi(q[C_{\delta''}])$. L'assertion résulte immédiatement de \ref{prop:changement-de-signes}.
\end{proof}

\begin{remark}
  Si $F$ est global, les caractères locaux définis ci-dessus s'assemblent, par la théorie du corps du classes, en un caractère
  $$ \kappa_T: H^1(\A/F, T) \to \bmu_2 $$
  avec la notation de \cite{Lab99} \S 1.4. Ceci sera utile pour la stabilisation de la formule des traces.
\end{remark}

Considérons la question de la normalisation (cf. \cite{Wa08} 4.6). Pour l'instant, supposons que $F$ est non archimédien de caractéristique résiduelle $p>2$ et $\psi$ est de conducteur $\mathfrak{o}_F$. Fixons un réseau autodual $L \subset W$ et le sous-groupe hyperspécial associé  $K \subset G(F)$. Fixons aussi un sous-groupe hyperspécial $K_H$ de $H(F)$. Supposons qu'il existe $\gamma \in K_H$, $\delta \in K$ qui se correspondent tels que $\gamma,\delta$ sont de réductions régulières. De tels choix existent.

\begin{proposition}[Normalisation à la Waldspurger]\label{prop:normalisation-waldspurger}
  Pour $(\gamma,\delta)$ comme ci-dessus, on a $\Delta(\gamma,\delta)=1$.
\end{proposition}
\begin{proof}
  Il existe une décomposition $L=L' \oplus L''$ où $L' = W' \cap L$, $L'' = W'' \cap L$ par nos hypothèses. Soient $K' \subset \Sp(W')$, $K'' \subset \Sp(W'')$ les sous-groupes hyperspéciaux associés, alors $\delta' \in K'$, $\delta'' \in K''$ et ils sont de réductions régulières.

  Montrons que $\Delta_0(\delta',\delta'')$, $\Delta'(\delta')$ et $\Delta''(\delta'')$ sont tous $1$. C'est clair pour $\Delta_0$; pour $\Delta'$ et $\Delta''$, on utilise \ref{prop:Theta-reduction-reguliere}.
\end{proof}

\begin{proposition}[Symétrie]\label{prop:symetrie}
  Notons $\Delta_{n',n''}$ le facteur de transfert associé à la paire $(n',n'')$ et $\Delta_{n'',n'}$ celui associé à $(n'',n')$. Pour tous $\gamma=(\gamma',\gamma'') \in H(F)$ et $\delta \in G(F)$ qui se correspondent, on a
  $$ \Delta_{n',n''}((\gamma', \gamma''), \tilde{\delta}) = \Delta_{n'',n'}((\gamma'',\gamma'), -\tilde{\delta}) \quad \text{si } 8|\mathbf{f} $$

  où $-1 \in \tilde{G}$ est celui défini dans \ref{def:-1}. Plus précisément, on a
  \begin{align*}
    \Delta'_{(n',n'')}(\tilde{\delta}')\Delta''_{(n',n'')}(\tilde{\delta}'') &= \Delta'_{(n'',n')}(-\tilde{\delta}'')\Delta''_{(n'',n')}(-\tilde{\delta}''), \quad \text{si } 8|\mathbf{f}; \\
    \Delta_{0,(n',n'')}(\delta',\delta'') &= \Delta_{0,(n'',n')}(-\delta'',-\delta'), \quad \mathbf{f} \text{ quelconque}.
  \end{align*}
\end{proposition}
\begin{proof}
  L'assertion pour $\Delta'\Delta''$ résulte de \ref{prop:-1-echangement}. L'assertion pour $\Delta_0$ sera démontrée en \ref{prop:reciprocite}.
\end{proof}

Enfin, on dispose aussi d'une formule du produit. Supposons $F$ global, $\psi: \A/F \to \mathbb{S}^1$ un caractère non-trivial avec décomposition $\psi = \prod_v \psi_v$. On a défini un facteur de transfert $\Delta_v$ pour toute place $v$.

\begin{proposition}[Formule du produit]\label{prop:Delta-forumule-produit}
  Soient $\gamma \in H_{G-\text{reg}}(F)$ et $\delta \in G(F)$ qui se correspondent. Soit $i(\delta) = [\tilde{\delta}_v]_v$, alors:
  \begin{itemize}
    \item pour presque toute place $v$, on a $\Delta_v(\gamma, \tilde{\delta}_v)=1$;
    \item $\prod_v \Delta_v(\gamma, \tilde{\delta}_v)= 1$.
  \end{itemize}
\end{proposition}
\begin{proof}
  Dans ce cas-là, on a une décomposition $W=W' \oplus W''$ et des éléments $\delta' \in \Sp(W')_\text{reg}$,$\delta'' \in \Sp(W'')_\text{reg}$ tels que $\iota(\delta',\delta'')=\delta$. Tous ces objets sont définis sur $F$. Vu la formule du produit pour $\Delta', \Delta''$ (\ref{prop:caractere-formule-produit}), il reste à prouver:
  \begin{itemize}
    \item pour presque toute place $v$, on a $\Delta_{0,v}(\delta',\delta'')=1$;
    \item $\prod_v \Delta_{0,v}(\delta',\delta'')= 1$.
  \end{itemize}
  Supposons que $\delta' \in \mathcal{O}(K'/K'^\#,a',c')$, $\delta'' \in \mathcal{O}(K''/K''^\#, a'', c'')$. Posons d'autre part
  $$\alpha'' := P_{a'}(a'')(-a'')^{-n'}\det(1+a') \in K''^{\#\times}. $$

  La formation de $\alpha''$ est compatible à complétion. Il s'agit de montrer que
  $$ \prod_v \sgn_{K''_v/{K''_v}^\#}(\alpha'')=1$$
  où $K''_v$ est la complétée de $K''$ comme une $F$-algèbre étale à involution. C'est une conséquence de la théorie des corps du classes.
\end{proof}

\subsection{Descente parabolique}\label{sec:descente-parabolique}
Supposons maintenant que $8|\mathbf{f}$. Rappelons que les sous-groupes de Lévi de $H$ sont de la forme
\begin{equation}
  M_H = \prod_{i \in I} (\GL(n'_i) \times \GL(n''_i)) \times H^\flat
\end{equation}
où $I$ est un ensemble fini, $H^\flat = \SO(2m'+1) \times \SO(2m''+1)$ avec $(m',m'') \in \Z_{\geq 0}^2$, les $(n'_i,n''_i) \in \Z_{\geq 0}^2$ sont tels que $\sum_{i \in I} n'_i  + m' = n'$, $\sum_{i \in I} n''_i  + m'' = n''$. Le plongement de $M_H$ dans $M$ est défini par $\prod_i \GL(n'_i) \times \SO(2m'+1) \hookrightarrow \SO(2n'+1)$ et $\prod_i \GL(n''_i) \times \SO(2m''+1) \hookrightarrow \SO(2n''+1)$. Posons $m=m'+m''$. On dit qu'un sous-groupe de Lévi $M$ de $G$ correspond à $(M_H,H)$ \index{correspondance de sous-groupes de Lévi}s'il est de la forme
\begin{equation}
  M = \prod_{i \in I} \GL(n_i) \times G^\flat,
\end{equation}
où $G^\flat := \Sp(W^\flat)$ avec $W^\flat$ un $F$-espace symplectique de dimension $2m$ tel que $n'_i + n''_i = n_i$ pour tout $i \in I$.

Fixons $M_H$ et supposons $M$ associé à $(M_H,H)$. À l'aide des décompositions ci-dessus, on écrit les éléments de $M(F)$ comme  $((\delta_i)_{i\in I}, \delta^\flat)$ et on écrit les éléments de $M_H(F)$ comme $((\gamma_i)_{i \in I}, \gamma^\flat)$. Notons $\tilde{G}^\flat$ la fibre de $\rev: \tilde{G} \to G(F)$ au-dessus de $G^\flat(F)$, alors $\rev: \tilde{G}^\flat \to G^\flat(F)$ est le revêtement métaplectique et $H^\flat$ est un groupe endoscopique elliptique de $\tilde{G}^\flat$ déterminé par la paire $(m',m'')$. Le revêtement $p$ se scinde canoniquement sur les composantes $\GL(n_i)$ par \ref{rem:Levi-lagrangien},\ref{prop:scindage-Schrodinger} et \ref{rem:caractere-Levi}, il y a donc un isomorphisme canonique
$$ j: \prod_i \GL(n_i) \times \tilde{G}^\flat \rightiso \rev^{-1}(M(F)). $$

Pour $\tilde{\delta} \in \rev^{-1}(\delta)$, notons $\tilde{\delta}^\flat \in \tilde{G}^\flat$ l'élément tel qu'il existe $(\delta_i)$ avec $j((\delta_i), \tilde{\delta}^\flat)=\tilde{\delta}$.

Un élément dans $M(F)_\text{ss}$ est dit $G$-régulier s'il est régulier dans $G$. On note $M_{G-\text{reg}}$ l'ouvert de Zariski des éléments $G$-réguliers. Un élément $\gamma \in {M_H(F)}_\text{ss}$ est dit $G$-régulier s'il existe $\delta \in M_{G-\text{reg}}(F)$ qui lui correspond.

\begin{proposition}[Descente parabolique du facteur de transfert]\label{prop:descente-parabolique}
  Soient $M_H$ un sous-groupe de Lévi de $H$ et $M$ un sous-groupe de Lévi de $G$ qui correspond à $(M_H,H)$. Soient $\gamma \in M_H(F)$ et $\delta \in M_{G-\text{reg}}(F)$ qui se correspondent comme éléments dans $H(F)$ et $G(F)$. Alors $\gamma^\flat$ et $\delta^\flat$ se correspondent. Notons $\Delta_{H,\tilde{G}}$ et $\Delta_{H^\flat,\tilde{G}^\flat}$ les facteurs de transfert associés à $(G,H)$ et $(G^\flat, H^\flat)$, alors
  $$ \Delta_{H,\tilde{G}}(\gamma, \tilde{\delta}) = \Delta_{H^\flat,\tilde{G}^\flat}(\gamma^\flat, \tilde{\delta}^\flat)$$
  pour tout $\tilde{\delta} \in \rev^{-1}(\delta)$. Plus précisément, cette propriété est satisfaite pour tous les deux facteurs $\Delta'\Delta''$ et $\Delta_0$.
\end{proposition}
\begin{proof}
  Pour le terme $\Delta'\Delta''$, observons que, d'après \ref{prop:caractere-decomposition-0}, \ref{rem:Levi-lagrangien} et \ref{prop:caractere-Levi}, on a
  $$ \Theta_\psi(j((\delta_i), \tilde{\delta}^\flat)) = \Theta_\psi(\tilde{\delta}^\flat) $$
  où le terme à gauche est défini par rapport à $\tilde{G}^\flat$. Par \ref{prop:Delta-variante}, on voit que
  $$ \Delta'(\tilde{\delta}')\Delta''(\tilde{\delta}'')= \frac{\Theta_\psi(\tilde{\delta}^\flat)}{|\Theta_\psi(\tilde{\delta}^\flat)|} \gamma_\psi(q[C_{\delta''}]). $$
  Or les classes de conjugaison des composantes $\delta_i$ donnent des formes hyperboliques comme facteurs directs de $q[C_{\delta''}]$, qui n'affectent pas $\gamma_\psi(q[C_{\delta''}])$. On obtient la descente du facteur $\Delta'\Delta''$ en appliquant \ref{prop:Delta-variante} encore une fois.

  Pour le terme $\Delta_0$, observons que $\mathcal{O}(\delta)$ est paramétré par une donnée
  $$ (K/K^\#,a,c) = \bigoplus_i (K_i/K_i^\#, a_i, c_i) \oplus (E/E^\#, \alpha, \gamma) $$
  où $\mathcal{O}(\delta^\flat)=\mathcal{O}(E/E^\#, \alpha, \gamma)$ et $K_i \simeq K_i^\# \times K_i^\#$ pour tout $i$. Les données ont de plus les décompositions $K_i/K_i^\# = K'_i/{K'_i}^\# \oplus K''_i/{K''_i}^\#$, etc.

  Il s'agit de démontrer que l'on peut enlever $(K'_i/{K'_i}^\#, a'_i, c'_i)$ et $(K''_i/{K''_i}^\#, a''_i, c''_i)$ dans la définition de $\Delta_0$. C'est clair pour $(K''_i/{K''_i}^\#, a''_i, c''_i)$ car le caractère $\sgn_{K''/K''^\#}(\cdot)$ se factorise par $K^{\#\times} \to E^{\#\times}$. Pour $(K'_i/{K'_i}^\#, a'_i, c'_i)$, observons que la symétrie \ref{prop:reciprocite} pour $\Delta_0$ échange les rôles de $\delta'$ et $\delta''$ quitte à les multiplier par $-1$. Un argument de va-et-vient permet de se ramener au cas précédent.
\end{proof}

\subsection{Énoncés du transfert et du lemme fondamental}\label{sec:conj-transfert}
Dans cette section, $F$ est un corps local.

\paragraph{Intégrales orbitales}
Soit $M$ un $F$-groupe réductif connexe et soit $m \in M_\text{reg}(F)$. Posons\index{$D_M(\cdot)$}
$$ D_M(m) := \det(1 - \Ad(m)|\mathfrak{m}/\mathfrak{m}_m). $$

Soit $f \in C_c^\infty(M(F))$. Fixons des mesures de Haar sur $M(F)$ et $M_m(F)$. L'intégrale orbitale de $f$ en $m$ est définie comme suit\index{$J_M(\cdot,\cdot)$}
$$ J_M(m, f) := |D_M(m)|^{1/2} \cdot \int_{M_m(F) \backslash M(F)} f(x^{-1}mx) \dd \dot{x}. $$

Si $m \in M(F)$ est fortement régulier, l'intégrale orbitale stable est définie comme\index{$J_M^\text{st}(\cdot,\cdot)$}
$$ J_M^\text{st}(m,f) := \sum_{m_1} J_M(m_1, f)$$
où $m_1$ parcourt des représentants de $\mathcal{O}^\text{st}(m)/\text{conj}$.

Les mêmes définitions s'adaptent à l'algèbre de Lie. Soit $X \in \mathfrak{m}_\text{reg}(F)$. Posons
$$ D_M(X) := \det(\ad(X)|\mathfrak{m}/\mathfrak{m}_X). $$

Soit $f \in C_c^\infty(\mathfrak{m}(F))$. Fixons les mesures comme précédemment et définissons
\begin{align}
  J_M(X,f) & := |D_M(X)|^{1/2} \cdot \int_{M_m(F) \backslash M(F)} f(x^{-1}Xx) \dd \dot{x},\\
  J_M^\text{st}(X,f) & := \sum_{X_1} J_M(X_1, f)
\end{align}
où $X_1$ parcourt des représentants de $\mathcal{O}^\text{st}(X)/\text{conj}$.

On considère aussi les intégrales orbitales sur un revêtement. Soit $\rev: \tilde{M} \to M(F)$ un revêtement d'un $F$-groupe réductif connexe satisfaisant à la condition simplificatrice suivante.

\begin{hypothesis}\label{hyp:revetement-commutant}
  On suppose que $\tilde{x},\tilde{y} \in \tilde{M}$ commutent si $x,y \in M(F)$ commutent.
\end{hypothesis}

Définissons les sous-ensembles $\tilde{M}_\text{reg}$ et $\tilde{M}_\text{ss}$ comme les images réciproques de leurs avatars sur $M(F)$. Fixons les mesures comme précédemment et définissons\index{$J_{\tilde{M}}(\cdot,\cdot)$}
\begin{equation}\label{eqn:integrale-endoscopique}
  J_{\tilde{M}}(\tilde{m}, f) := |D_M(m)|^{1/2} \int_{M_m(F) \backslash M(F)} f(\tilde{x}^{-1}\tilde{m}\tilde{x}) \dd \dot{x}
\end{equation}
où $\tilde{x}$ est une image réciproque quelconque de $x$. Si $F$ est archimédien, alors on peut définir les intégrales orbitales d'une fonction $f$ dans l'espace de Schwartz $\mathcal{S}(M(F))$ ou $\mathcal{S}(\tilde{M})$.

Revenons au cas du groupe métaplectique pour lequel l'hypothèse \ref{hyp:revetement-commutant} est satisfaite. Soient $\gamma \in H_{G-\text{reg}}(F)$, $f \in C_{c,\asp}^\infty(\tilde{G})$. Définissons l'intégrale orbitale endoscopique par\index{intégrale orbitale endoscopique}\index{$J_{H,\tilde{G}}(\cdot,\cdot)$}
$$ J_{H,\tilde{G}}(\gamma, f) := \sum_{\delta} \Delta(\gamma, \tilde{\delta}) J_{\tilde{G}}(\tilde{\delta}, f) $$
où:
\begin{itemize}
  \item les $\delta$ parcourent les représentants des classes de conjugaison dans $G(F)$ qui se correspondent à $\gamma$;
  \item $\tilde{\delta} \in \tilde{G}$ est une image réciproque quelconque de $\delta$, le choix n'affecte pas la définition car $\Delta$ est spécifique (\ref{prop:facteur-specificite}) et $f$ est anti-spécifique.
\end{itemize}

Si $F$ est archimédien, on peut aussi définir $J_{H,\tilde{G}}(\cdot, f)$ pour $f \in \mathcal{S}_{\asp}(\tilde{G})$.

\paragraph{La conjecture de transfert}
Soient $\gamma \in H_{G-\text{reg}}(F)$ et $\delta \in G_\text{reg}(F)$ qui se correspondent. D'après la description des commutants via paramètres dans \S\ref{sec:parametrage}, on a un isomorphisme entre $F$-tores $H_\gamma \rightiso G_\delta$. On en déduit une correspondance entre les mesures de Haar sur $H_\gamma(F)$ et $G_\delta(F)$.

Le résultat que nous établirons est le suivant.
\begin{theorem}[Transfert d'intégrales orbitales]\label{prop:transfert}
  Fixons des mesures de Haar sur $G(F)$ et $H(F)$. Soit $f \in C_{c,\asp}^\infty(\tilde{G})$, alors il existe $f^H \in C_c^\infty(H(F))$ tel que
  $$ J_{H,\tilde{G}}(\gamma, f) = J^\text{st}(\gamma, f^H) $$
  pour tout $\gamma \in H_{G-\text{reg}}(F)$, où on utilise les mesures de Haar qui se correspondent sur les commutants. On dit que $f^H$ est un transfert de $f$.

  Si $F$ est archimédien et $f \in \mathcal{S}_{\asp}(\tilde{G})$, alors on peut prendre $f^H \in \mathcal{S}(H(F))$ satisfaisant à l'égalité ci-dessus.
\end{theorem}

La fonction $f^H$ n'est pas unique, mais ses intégrales orbitales stables sont caractérisées par $f$.

\paragraph{Le cas non ramifié: le lemme fondamental pour les unités de l'algèbre de Hecke}

\begin{hypothesis}\label{hyp:non-ramifie}\index{le cas non ramifié}
  On fait les hypothèses ci-dessous.
  \begin{itemize}
    \item $F$ est non archimédien de caractéristique résiduelle $p>2$ et $|\mathfrak{o}_F/\mathfrak{p}_F| > 3$.
    \item $p$ est assez grand par rapport à $n$. Une minoration possible de $p$ est celle dans \cite{Wa08} 4.4, appliquée aux groupes $G$ et $H$.
    \item $\psi$ est de conducteur $\mathfrak{o}_F$.
  \end{itemize}
\end{hypothesis}

\begin{definition}
  Soit $M$ un $F$-groupe réductif non ramifié. On dit qu'une mesure de Haar sur $M(F)$ est non ramifiée si $\mes(K)=1$ pour tout sous-groupe hyperspécial $K$ de $M$.
\end{definition}

Remarquons que les sous-groupes hyperspéciaux sont conjugués par $M_\text{ad}(F)$, et la conjugaison préserve les mesures de Haar, il suffit donc de considérer un choix de $K$.

Prenons un réseau autodual $L \subset W$ et notons $K = \mathrm{Stab}_G(L)$ le sous-groupe hyperspécial de $G(F)$ associé. Fixons aussi un sous-groupe hyperspécial $K_H$ de $H(F)$. Le lemme fondamental pour les unités concerne le transfert de la fonction\index{$f_K$} anti-spécifique
$$ f_K(\tilde{x}) = \begin{cases}
    \noyau^{-1}, & \text{si } \tilde{x} \in \noyau K, \noyau \in \Ker(\rev), \\
    0, & \text{si } \tilde{x} \notin \rev^{-1}(K).
\end{cases}$$

\begin{theorem}[Le lemme fondamental pour les unités de l'algèbre de Hecke sphérique anti-spécifique]\label{prop:LF-unite}
  Soit $K_H$ un sous-groupe hyperspécial $H(F)$, alors $\mathbbm{1}_{K_H}$ est un transfert de $f_K$ si l'on utilise les mesures non ramifiées sur $G(F)$ et $H(F)$.
\end{theorem}

La démonstration occupera les sections suivantes. Nous obtiendrons aussi des résultats plus forts sur le transfert pour le cas archimédien.

\paragraph{Remarques sur le choix de revêtements}
  Les seuls résultats qui fait intervenir l'hypothèse $8|\mathbf{f}$ sont les suivants:
\begin{itemize}
  \item la symétrie pour $\Delta'\Delta''$ (\ref{prop:symetrie}), qui fait intervenir l'élément $-1 \in \tilde{G}$;
  \item la descente parabolique \S\ref{sec:descente-parabolique}.
\end{itemize}

En particulier, on peut formuler la conjecture de transfert et le lemme fondamental pour chaque choix de $\mathbf{f}$. Montrons qu'elles sont équivalentes. Soient $\mathbf{f},\mathbf{f}' \in 2\Z_{\geq 1}$ et supposons que $\mathbf{f}|\mathbf{f}'$. Posons $\tilde{G}' := \MMp{\mathbf{f}'}(W)$, $\tilde{G} := \MMp{\mathbf{f}}(W)$. Alors $\tilde{G} \hookrightarrow \tilde{G}'$. Le résultat suivant découle immédiatement.

\begin{lemma}
  Soient $f' \in C_{c,\asp}^\infty(\tilde{G}')$ et $f \in C_{c,\asp}^\infty(\tilde{G})$ sa restriction. Pour tout $\tilde{\delta} \in \tilde{G}$, on a
  $$ J_{\tilde{G}'}(\tilde{\delta}, f') = J_{\tilde{G}}(\tilde{\delta}, f), $$
  où on utilise les mêmes mesures de Haar sur $G(F)$ et $G_\delta(F)$.
\end{lemma}

\begin{proposition}\label{prop:rem-revetements}
  Les énoncés de transfert (cf. \ref{prop:transfert}) pour $\tilde{G}$ et $\tilde{G}'$ sont équivalents.

  Dans le cas non ramifié, les énoncés du lemme fondamental pour les unités (cf. \ref{prop:LF-unite}) pour $\tilde{G}$ et $\tilde{G}'$ sont équivalents.
\end{proposition}
\begin{proof}
  Soient $f', f$ comme ci-dessus. Dans la définition de l'intégrale orbitale endoscopique $J_{H,\tilde{G}'}(\gamma,f')$, on peut choisir les $\tilde{\delta}$ dans $\tilde{G}$. Vu la compatibilité de la notion de stabilité (par \ref{rem:stabilite-compatibilite}) et du facteur de transfert (par sa spécificité) par rapport à restriction, on déduit par ledit lemme que
  $$ J_{H,\tilde{G}'}(\gamma,f') =  J_{H,\tilde{G}}(\gamma,f).$$
\end{proof}
Le cas des fonctions de Schwartz dans le cas réel est analogue.

\section{Transfert: le cas archimédien}\label{sec:trans-arch}
Les résultats ici sont valables pour le transfert des fonctions à support compact ainsi que les fonctions de Schwartz. Nous traitons principalement le premier cas et nous signalerons à la fin les modifications nécessaires pour les fonctions de Schwartz (\ref{rem:Schwartz}).

Puisque la restriction des scalaires ne s'applique pas au groupe métaplectique, il faut traiter le cas complexe séparément. Pour la plupart, nous nous occupons du cas $F=\R$; le cas $F=\C$ est beaucoup plus simple. Nous rajouterons quelques remarques à la fin.

D'après \ref{prop:rem-revetements}, il suffit d'établir le transfert pour $\mathbf{f}=8$. Notons $\Mp(W) := \MMp{8}(W)$ pour tout $F$-espace symplectique $W$.

\subsection{L'endoscopie chez Renard}
Posons $F=\R$ et $n'+n''=n$. Soient $W$ un $F$-espace symplectique de dimension $2n$ et $W=W' \oplus W''$, $\dim W' = 2n'$, $\dim W'' = 2n''$. À la paire $(n',n'')$ et à un $F$-tore maximal $T$ dans $G$ est associé l'objet $\kappa$ dans \ref{prop:propriete-cocycle}.

Afin de réconcilier avec le formalisme de \cite{Re99}, posons
\begin{gather*}
  G := \Sp(W), \\
  \tilde{G} := \Mp(W), \\
  G^\diamond := \Sp(W') \times \Sp(W''),\\
  \tilde{G}^\diamond := \Mp(W') \times \Mp(W''),\\
  p^\diamond: \tilde{G}^\diamond \to G^\diamond, \\
  \iota: G^\diamond \hookrightarrow G,\\
  j: \tilde{G}^\diamond \to \tilde{G}.
\end{gather*}

Dans \cite{Re99}, Renard travaille avec les revêtements à deux feuillets et il considère le groupe $\Image(j)$ au lieu de $\tilde{G}^\diamond$. Mais peu importe.

\begin{definition}
  On dit qu'une fonction $f \in C_c^\infty(\tilde{G}^\diamond)$ est spécifique (resp. anti-spécifique) si $f$ est spécifique (resp. anti-spécifique) en les deux coordonnées. L'espace de telles fonctions est noté par $C_{c,-}^\infty(\tilde{G}^\diamond)$ (resp. $C_{c,\asp}^\infty(\tilde{G}^\diamond)$).

  On dit qu'un élément $(\delta',\delta'') \in G^\diamond$ est $G$-régulier si $\iota(\delta',\delta'')$ est semi-simple régulier. L'ensemble des éléments $G$-réguliers est noté $G^\diamond_{G-\text{reg}}$, et son image réciproque $\tilde{G}^\diamond_{G-\text{reg}}$.
\end{definition}

Nous utiliserons systématiquement l'application
\begin{align*}
  \tau: & \tilde{G}^\diamond \to \tilde{G}^\diamond \\
  & (\tilde{x}',\tilde{x}'') \mapsto (\tilde{x}', -\tilde{x}''),
\end{align*}
où $\tilde{x}'' \mapsto -\tilde{x}''$ est l'opération définie par \ref{def:-1}. On a $\tau^2 = \identity$.

Conservons les conventions de \cite{Re99} pour les mesures et les intégrales orbitales. Identifions donc une intégrale orbitale sur $\tilde{G}$ à une certaine fonction invariante $\phi \in C^\infty(\tilde{G}_\text{reg})$; la même convention s'applique à $\tilde{G}^\diamond$. Dans ce qui suit, nous ne considérons que les fonctions anti-spécifiques.

\begin{definition}
  On dit qu'une fonction $\phi \in C^\infty(\tilde{G}_\text{reg})$ est une intégrale orbitale anti-spécifique s'il existe $f \in C_{c,\asp}^\infty(\tilde{G})$ telle que $\phi = J_{\tilde{G}}(\cdot,f)$. L'espace des intégrales orbitales anti-spécifiques est notée $\mathcal{I}_{\asp}(\tilde{G})$. De même, définissons l'espace $\mathcal{I}_{\asp}(\tilde{G}^\diamond)$.

  En utilisant la notion de stabilité introduite en \S\ref{sec:stabilite}, qui coïncide avec celle de \cite{Re99}, nous définissons les espaces d'intégrales orbitales stables anti-spécifiques:
  $$ \mathcal{I}^\text{st}_{\asp}(\tilde{G}), \mathcal{I}^\text{st}_{\asp}(\tilde{G}^\diamond). $$

  On dit qu'une fonction $\phi \in C^\infty(\tilde{G}_\text{reg})$ est une intégrale $\kappa$-orbitale anti-spécifique si $\tau^* \phi$ est une intégrale orbitale stable anti-spécifique. L'espace des intégrales $\kappa$-orbitales anti-spécifiques est noté $\mathcal{I}_{\kappa,\asp}(\tilde{G}^\diamond)$.
\end{definition}

Notre définition des intégrales $\kappa$-orbitales coïncide avec celle de Renard d'après \cite{Re99} \S 3 p.1226. L'espace $\mathcal{I}_{\kappa,\asp}(\tilde{G}^\diamond)$ est une  limite inductive d'espaces de Fréchet.

La démonstration du transfert archimédien repose sur trois ingrédients: le transfert de $\tilde{G}$ vers $\tilde{G}^\diamond$, le transfert de $\tilde{G}^\diamond$ vers le groupe endoscopique $H$, et la comparaison des facteurs de transfert. La comparaison sera fait dans la sous-section suivante; les deux théorèmes de transfert se trouvent dans \cite{Re98} et \cite{Re99}, mais il faut les mettre sous une forme convenable.

\begin{theorem}[cf. \cite{Re99} 4.7] \label{prop:Renard}
  On sait définir un facteur de transfert
  $$ \Delta_R: G_{G-\text{reg}}^\diamond \to \bmu_2 $$
  de sorte qu'il existe une application linéaire continue
  \begin{align*}
    \mathrm{Trans}: & \mathcal{I}_{\asp}(\tilde{G}) \to \mathcal{I}_{\kappa,\asp}(\tilde{G}^\diamond) \\
    & \phi \mapsto \phi^\diamond
  \end{align*}
  définie de la façon suivante: $\phi^\diamond$ est déterminée sur le sous-ensemble dense $\tilde{G}_{G-\text{reg}}^\diamond$ par
  \begin{align*}
    \tilde{\delta}_0 &:= j(\tilde{\delta}',\tilde{\delta}''), \\
    \phi^\diamond(\tilde{\delta}',\tilde{\delta}'') &= \Delta_R(\delta',\delta'') \sum_{\tilde{\delta} \in \mathcal{O}^\text{st}(\tilde{\delta}_0)/\text{conj}} \angles{\kappa,\inv(\delta,\delta_0)} \phi(\tilde{\delta}).
  \end{align*}
\end{theorem}
Observons en passant que $j$ identifie $\mathcal{O}^\text{st}(\tilde{\delta}')/\text{conj} \times \mathcal{O}^\text{st}(\tilde{\delta}'')/\text{conj}$ à $\mathcal{O}^\text{st}(\tilde{\delta}_0)/\text{conj}$.

Nous donnerons une expression explicite de $\Delta_R$ dans la section suivante.
\begin{proof}
  Dans \cite{Re99}, ce théorème est énoncé pour les revêtements à deux feuillets, le groupe $\Image(j)$ au lieu de $\tilde{G}^\diamond$ et les intégrales orbitales spécifiques. Notre énoncé s'obtient en trois étapes.
  \begin{itemize}
    \item Le théorème \cite{Re99} 4.7 est valide si l'on remplace $\Image(j)$ par $\tilde{G}^\diamond$. Ceci est clair.
    \item Il reste valide pour les intégrales orbitales anti-spécifiques; ceci est trivial car les objets spécifiques et anti-spécifiques coïncident pour les revêtements à deux feuillets.
    \item Finalement, on étend le théorème \cite{Re99} 4.7 à tout revêtement de $G(F)$. Cela se fait comme dans \ref{prop:rem-revetements}.
  \end{itemize}
\end{proof}

\begin{theorem}[cf. \cite{Re98} 6.7]\label{prop:Renard-Adams}
  Posons $H := \SO(2n'+1) \times \SO(2n''+1)$, alors on peut définir une application linéaire
  $$ \mathcal{T}: \mathcal{I}_{\asp}^\text{st}(\tilde{G}^\diamond) \longrightarrow \mathcal{I}^\text{st}(H(F))$$
  caractérisée par
  $$ \mathcal{T}(\phi)(\gamma',\gamma'') = \Delta'(\tilde{\delta}') \Delta'(\tilde{\delta}'') \phi(\tilde{\delta}', \tilde{\delta}'')$$
  où $\gamma'$ et $\delta' \in \Mp(W')_\text{reg}$ se correspondent par rapport à la donnée endoscopique $(\Mp(W'), \SO(2n'+1) \times \{1\})$, $\gamma''$ et $\delta'' \in \Mp(W'')_\text{reg}$ se correspondent par rapport à $(\Mp(W''), \SO(2n''+1) \times \{1\})$, et
  $$ \Delta' := \frac{\Theta^+_\psi - \Theta^-_\psi}{|\Theta^+_\psi - \Theta^-_\psi|}. $$
\end{theorem}
\begin{proof}
  Cet énoncé est pour l'essentiel celui de Renard. L'énoncé dans \cite{Re98} ne concerne que le cas $n''=0$, mais on l'étend sans difficulté aux produits directs. On procède comme dans la démonstration de \ref{prop:Renard} et obtient la version énoncée ci-dessus.
\end{proof}

\subsection{Comparaison de facteurs de transfert}\label{sec:comparaison-Renard}
Les tores maximaux dans $G$ sont isomorphes à des tores ``standards'' de la forme
\begin{gather*}
  T^{m,r,s} := (\C^\times)^m \times (\mathbb{S}^1)^r \times (\R^\times)^s, \\
  2m+r+s=n.
\end{gather*}
Cela décrit tous les tores maximaux dans $G$ à conjugaison stable près.

On utilise les coordonnées $z=(z_i)_{i=1}^m$, $w=(w_j)_{j=1}^r$, $a=(a_k)_{k=1}^s$ pour $T^{m,r,s}$. Soient $\hat{z}_i, \hat{w}_j, \hat{a}_k$ les caractères correspondants.

\begin{proposition}\label{prop:Renard-formule}\index{$\Delta_R$}
  Soit $(\delta',\delta'') \in G^\diamond$ qui est $G$-régulier. Supposons que le tore maximal $T' \times T''$ contenant $(\delta',\delta'')$ est de la forme $T^{m',r',s'} \times T^{m'',r'',s''}$. On paramètre $(\delta',\delta'')$ par ses coordonnées $(z',w',a',z'',w'',a'')$. Alors
  $$ \Delta_R(\delta',\delta'') = \prod_{j',j''} \sgn(\mathrm{Re}(w'_{j'})-\mathrm{Re}(w''_{j''})). $$
\end{proposition}
\begin{proof}
  C'est essentiellement \cite{Re99} (4.8) et les discussions qui suivent. 
\end{proof}

\begin{corollary}
  Le facteur $\Delta_R$ est constant sur une classe de conjugaison stable.
\end{corollary}

\begin{proposition}\label{prop:comparaison-Renard}
  On a $\Delta_0=\Delta_R$.
\end{proposition}
\begin{proof}
  Tous les deux facteurs $\Delta_R$ et $\Delta_0$ satisfont à la descente parabolique \ref{prop:descente-parabolique}. Il suffit donc de les comparer sur tores de la forme $T^{m',r',0} \times T^{m'',r'',0}$. Soient $(K'/K'^\#, b', c')$ et $(K''/K''^\#, b'', c'')$ les paramètres de $\mathcal{O}(\delta')$ et $\mathcal{O}(\delta'')$, respectivement.

  Calculons d'abord le polynôme $P_{b'}(T)(-T)^{-n'}$. On a
  \begin{align*}
    P_{b'}(T)(-T)^{-n'} & = \prod_i \frac{(T-z'_i)(T-{z'_i}^{-1})(T-\overline{z'_i})(T-\overline{z'_i}^{-1})}{(-T)^2} \cdot \\
    & \cdot \prod_j \frac{T^2 - 2\mathrm{Re}(w'_j)+1}{-T} \\
    & = \prod_i ((T+T^{-1})-(z_i+z_i^{-1}))((T+T^{-1})-\overline{(z'_i+{z'_i}^{-1})}) \cdot \\
    & \cdot \prod_j (2\mathrm{Re}(w'_j)-(T+T^{-1})).
  \end{align*}

  Les termes dans le produit sur $i$ sont non négatifs si $T$ est remplacé par un élément dans $\mathbb{S}^{1}$. D'autre part,
  \begin{align*}
    \det(1+\delta') & = \prod_i (1+z'_i)(1+{z'_i}^{-1})(1+\overline{z'_i})(1+\overline{{z'_i}}^{-1}) \cdot \\
    & \cdot \prod_j (1+w'_j)(1+\overline{w'_j}),
  \end{align*}
  qui est un réel positif, d'où
  \begin{align*}
    \sgn(P_{b'}(w''_{j''})(-w''_{j''})^{-n'} \det(1+\delta')) = & \prod_{j'} \sgn(\mathrm{Re}(w'_{j'})-\mathrm{Re}(w''_{j''})).
  \end{align*}

  On conclut en le comparant avec \ref{prop:Renard-formule}.
\end{proof}

Conservons maintenant les notations de \S\ref{sec:donnees-endoscopiques}.

\begin{theorem}
  Soit $f \in C_{c,\asp}^\infty(\tilde{G})$, alors il existe $f^H \in C_c^\infty(H(F))$ telle que pour tout $\gamma=(\gamma',\gamma'') \in H_{G-\text{reg}}(F)$,
  $$ J_H^\text{st}(\gamma, f^H) = \sum_{\delta} \Delta(\epsilon, \tilde{\delta}) J_{\tilde{G}}(\tilde{\delta}, f) $$
  où la somme parcourt les classes de conjugaison des $\delta \in G(F)$ qui correspondent à $\gamma$ et $\tilde{\delta} \in \tilde{G}$ est une image réciproque quelconque de $\delta$.

  De plus, l'application linéaire $\phi \mapsto \phi_H$ de $\mathcal{I}_{\asp}(\tilde{G})$ dans $\mathcal{I}^\text{st}(H(F))$ induite par $f \mapsto f^H$ est continue.
\end{theorem}
\begin{proof}
  Fixons $\tilde{\delta}_0$ tel que $\delta_0$ et $\gamma$ se correspondent. Prenons $(\tilde{\delta}'_0, \tilde{\delta}''_0) \in \tilde{G}^\diamond$ comme dans \S\ref{sec:facteur-de-transfert}, alors $j(\tilde{\delta}'_0, \tilde{\delta}''_0) = \tilde{\delta}_0$. On peut supposer que les $\tilde{\delta}$ dans la somme parcourent $\mathcal{O}^\text{st}(\tilde{\delta}_0)/\text{conj}$. Posons $\phi := J_{\tilde{G}}(\cdot,f) \in \mathcal{I}_{\asp}(\tilde{G})$. D'après \ref{prop:propriete-cocycle}, le terme à droite vaut
  $$ \Delta'(\tilde{\delta}'_0)\Delta''(\tilde{\delta}''_0)\Delta_0(\delta'_0,\delta''_0) \sum_{\tilde{\delta}} \angles{\kappa, \inv(\delta_0,\delta)} \phi(\tilde{\delta}). $$

  Or \ref{prop:comparaison-Renard} et \ref{prop:Renard} entraînent qu'il est égal à
  $$ \Delta'(\tilde{\delta}'_0)\Delta''(\tilde{\delta}''_0) \phi^\diamond(\tilde{\delta}'_0, \tilde{\delta}''_0), $$
  ou
  $$ \Delta'(\tilde{\delta}'_0)\Delta'(-\tilde{\delta}''_0) \cdot (\tau^* \phi^\diamond) (\tilde{\delta}'_0, -\tilde{\delta}''_0). $$

  Maintenant, $\delta'_0$ correspond à $\gamma'$ pour la donnée endoscopique $(\Mp(W'), \SO(2n'+1) \times \{1\})$ et $-\delta''_0$ correspond à $\gamma''$ pour $(\Mp(W''), \SO(2n''+1) \times \{1\})$. D'autre part $\tau^* \phi^\diamond \in \mathcal{I}_{\asp}^\text{st}(\tilde{G}^\diamond)$. En appliquant \ref{prop:Renard-Adams}, on voit que la fonction $\phi_H \in C^\infty(H_{G-\text{reg}}(F))$ définie par
  $$ \gamma \longmapsto \Delta'(\tilde{\delta}'_0)\Delta'(-\tilde{\delta}''_0) \cdot (\tau^* \phi^\diamond) (\tilde{\delta}'_0, -\tilde{\delta}''_0) $$
  appartient à $\mathcal{I}^\text{st}(H(F))$. Ceci démontre l'existence de $f^H$. L'application $\text{Trans}$ de \ref{prop:Renard} et l'application $\mathcal{T}$ de \ref{prop:Renard-Adams} sont continues, d'où l'assertion sur la continuité.
\end{proof}

Ce théorème établit le transfert \ref{prop:transfert} pour $F=\R$ et pour les fonctions à support compacts.

\begin{remark}\label{rem:Schwartz}
  On peut démontrer une variante du théorème ci-dessus où $f \in \mathcal{S}_{\asp}(\tilde{G})$ et $f^H \in \mathcal{S}(H)$. Il faut adapter nos arguments, ainsi que ceux de \cite{Re98,Re99}, aux intégrales orbitales de fonctions de Schwartz. Des résultats de Harish-Chandra et Shelstad permettent de caractériser ces espaces. Par exemple, les conditions les plus subtiles ($I_2^\text{st}$), ($I_3^\kappa$) dans \cite{Re99} ne changent pas pour les fonctions de Schwartz, donc les mêmes arguments marchent toujours.
\end{remark}

\begin{remark}
  C'est aussi possible d'établir le transfert par la descente parabolique et la descente semi-simple, qui sera établie dans la section suivante. Il suffit de reprendre les arguments de \cite{Sh06} et utiliser la caractérisation des intégrales orbitales énoncée dans \cite{Re99}.
\end{remark}

\subsection{Le cas complexe}
Supposons maintenant que $F=\C$. On confond les $\C$-groupes et leurs points complexes. Dans ce cas:
\begin{itemize}
  \item $\tilde{G} = \bmu_\mathbf{f} \times G$, et $\Delta(\gamma, (t, \delta))=t$ pour tout $t \in \bmu_\mathbf{f}$ et tous $\gamma \in H$, $\delta \in G$ qui se correspondent;
  \item on peut identifier $C_{c,\asp}^\infty(\tilde{G})$ à $C_c^\infty(G)$ via $f \mapsto f(1, \cdot)$;
  \item la conjugaison se confond avec la conjugaison géométrique;
  \item les tores maximaux sont conjugués;
  \item les intégrales orbitales stables se confondent avec les intégrales orbitales.
\end{itemize}

On peut toujours définir les espaces vectoriels topologiques d'intégrales orbitales $\mathcal{I}(G)$ et $\mathcal{I}(H)$.

L'existence de transfert est équivalente à
\begin{theorem}
   Soit $f \in C_c^\infty(G)$, alors il existe $f^H \in C_c^\infty(H)$ telle que pour tout $\gamma=(\gamma',\gamma'') \in H_{G-\text{reg}}$,
  $$ J_H(\gamma, f^H) = J_{G}(\delta, f) $$
  pour tout $\delta$ qui correspond à $\gamma$.

  De plus, l'application linéaire $\phi \mapsto \phi^H$ de $\mathcal{I}(G)$ dans $\mathcal{I}(H)$ induite par $f \mapsto f^H$ est continue.
\end{theorem}
\begin{proof}[Esquisse d'une démonstration]
  La caractérisation des intégrales orbitales se simplifie énormément pour les groupes complexes: il n'y a plus de ``conditions de sauts'' (eg. les conditions $I_2^\text{st}$ et $I_3^\kappa$ pour le groupe métaplectique réel). En effet, les singularités éventuelles sont associées aux racines; localement elles forment des murs de codimension réelle $2$. D'après un principe de Harish-Chandra, on peut outrepasser ces murs.

  Fixons des tores maximaux $S \subset H$ et $T \subset G$. On sait qu'il existe un isomorphisme $S \rightiso T$, équivariant par rapport au homomorphisme des groupes de Weyl $W^H \to W^G$. Les classes de conjugaison semi-simples dans $G$ (resp. $H$) sont paramétrées par $T/W^G$ (resp. $S/W^H$). Le transfert en découle immédiatement en appliquant la caractérisation des intégrales orbitales mentionnée précédemment.
\end{proof}

\begin{remark}
  Comme le cas réel, il y a aussi une variante de ce théorème pour les fonctions de Schwartz.
\end{remark}

\section{Descente semi-simple du facteur de transfert}
Fixons $\mathbf{f}$ tel que $8|\mathbf{f}$. Les revêtements métaplectiques dans cette section désignent les revêtements à $\mathbf{f}$ feuillets.

\subsection{Le formalisme de descente} \label{sec:descente-formalisme}
Fixons $F$ un corps local, $(W,\angles{\cdot|\cdot})$ un $F$-espace symplectique de dimension $2n$, $(n',n'') \in \Z_{\geq 0}^2$ tel que $n'+n''=n$. On en déduit les objets suivants.
\begin{align*}
  \tilde{G} & := \Mp(W), \\
  G & := \Sp(W), \\
  H' & := \SO(2n'+1)\; \text{déployé},\\
  H'' & := \SO(2n''+1)\; \text{déployé},\\
  H & := H' \times H'' .
\end{align*}

Fixons $\epsilon = (\epsilon',\epsilon'') \in H(F)_\text{ss}$ et $\tilde{\eta} \in \tilde{G}$ tels que $\eta \in G(F)_\text{ss}$ et $\epsilon$ se correspondent. Pour simplifier la vie, supposons aussi que $\tilde{\eta}=\pm 1$ lorsque $\eta=\pm 1$ (où $-1 \in \tilde{G}$ est celui défini dans \ref{def:-1}). Écrivons les paramètres de leurs classes de conjugaison comme
\begin{align*}
  \eta & \in \mathcal{O}(K/K^\#, v, (W_K, h_K), (W_\pm, \angles{\cdot|\cdot}_\pm)), \\
  \epsilon' & \in \mathcal{O}(K'/K'^\#, v', (V'_K,h'_K), (V'_\pm,q'_\pm)), \\
  \epsilon'' & \in \mathcal{O}(K''/K''^\#, v'', (V''_K,h''_K), (V''_\pm,q''_\pm));
\end{align*}
où $h_K$ est anti-hermitienne et $h'_K$, $h''_K$ sont hermitiennes.

La correspondance de classes exige que
\begin{gather*}
  (K/K^\#,v) = (K'/K'^\#, v') \oplus (K''/K''^\#, -v''), \\
  W_K \simeq V'_{K'} \oplus V''_{K''} \;\text{ comme } K-\text{modules}, \\
  \dim_F W_+ + 1 = \dim_F V'_+ + \dim_F V''_-, \\ 
  \dim_F W_- + 1 = \dim_F V'_- + \dim_F V''_+.
\end{gather*}

On décompose $K$ en produit de $F$-algèbres à involution $L$ pour lesquelles $L^\#$ sont des corps. Si l'on pose $u := v|_L$ alors $L=F(u)$. Adoptons la convention d'indexer l'ensemble des paires $(L,u)$ par $u$. Idem pour $K',K''$. Alors
\begin{align*}
  (K/K^\#, \cdots) &= \bigoplus_u (L/L^\#, u, (W_u,h_u)) \oplus (W_+, \angles{\cdot|\cdot}_+) \oplus (W_-, \angles{\cdot|\cdot}_-), \\
  (K'/K'^\#, \cdots) &= \bigoplus_u) (L/L^\#, u, (V'_u,h'_u)) \oplus (V'_+, q'_+) \oplus (V'_-, q'_-),\\
  (K''/K''^\#, \cdots) &= \bigoplus_u (L/L^\#, -u, (V''_u,h''_u)) \oplus (V''_+, q''_+) \oplus (V''_-, q''_-);
\end{align*}
où
\begin{itemize}
  \item on permet que pour tout $u$, au plus l'un de $V'_u$ et $V''_u$ est trivial;
  \item les paires $(L,u)$ sont deux à deux inéquivalentes;
  \item pour tout $u$, $ W_u \simeq V'_u \oplus V''_u$ comme $L$-modules.
\end{itemize}

D'après \S\ref{sec:parametrage},
\begin{align*}
  G_\eta & = \prod_u U(W_u, h_u) \times \Sp(W_+) \times \Sp(W_-), \\
  H'_{\epsilon'} & = \prod_u U(V'_u, h'_u) \times \SO(V'_+, q'_+) \times \SO(V'_-,q'_-), \\
  H''_{\epsilon''} & = \prod_u U(V''_u, h''_u) \times \SO(V''_+, q''_+) \times \SO(V''_-,q''_-),\\
  H_\epsilon &= H'_{\epsilon'} \times H''_{\epsilon''}.
\end{align*}

On note abusivement la restriction de $\eta$ sur $U(W_u,h_u)$ par $u$, celle de $\epsilon'$ sur $U(V'_u,h'_u)$ par $u'$ et celle de $-\epsilon''$ sur $U(V''_u,h''_u)$ par $u''$.

Rappelons que $U(W_u, h_u) \simeq \GL_{L^\#}(\frac{1}{2}\dim_{L^\#} W_u)$ si $L \simeq L^\# \times L^\#$. On a des immersions canoniques
\begin{gather*}
  U(W_u, h_u) \hookrightarrow \Sp(W_u, \angles{\cdot|\cdot}_u) \\
  U(V'_u,h'_u) \hookrightarrow \SO(V'_u, (\Tr_{L/L^\#})_* h'_u) \\
  U(V''_u,h''_u) \hookrightarrow \SO(V''_u, (\Tr_{L/L^\#})_* h''_u)
\end{gather*}
où $\angles{\cdot|\cdot}_u := (\Tr_{L/L^\#})_* h_u)$.

\begin{lemma}
  Si $H_\epsilon$ est quasi-déployé, alors $\SO(V'_+,q'_+)$ et $\SO(V''_+, q''_+)$ sont déployés.
\end{lemma}
\begin{proof}
  Si $H_\epsilon$ est quasi-déployé, d'après la description des commutants dans \S\ref{sec:parametrage} on déduit que $\SO(V'_+,q'_+)$ et $\SO(V''_+, q''_+)$ sont aussi quasi-déployés. Or un groupe orthogonal impair quasi-déployé est forcément déployé.
\end{proof}

Soient
\begin{align*}
  X &= ((X_u)_u, X_+, X_-) \in \mathfrak{g}_\eta(F), \\
  Y' & = ((Y'_u)_u, Y'_+, Y'_-) \in \mathfrak{h}'_{\epsilon'}(F), \\
  Y'' & = ((Y''_u)_u, Y''_+, Y''_-) \in \mathfrak{h}''_{\epsilon''}(F) \\
  Y &= (Y',Y'') \in \mathfrak{h}_\epsilon(F).
\end{align*}

\begin{hypothesis}
  On suppose $X,Y$ semi-simples réguliers et assez petits.
\end{hypothesis}

Supposons de plus que $\delta := \exp(X)\eta$ et $\gamma := \exp(Y)\epsilon$ se correspondent.

\begin{definition}\label{def:CVP}\index{correspondance par valeurs propres}\index{$\CVP$}
  Soient $\mathfrak{m}_1, \mathfrak{m}_2$ des algèbres de Lie de produits de restrictions des scalaires de groupes classiques. On dit que deux éléments $Z_i \in \mathfrak{m}_i$ ($i=1,2$) semi-simples assez réguliers sont en correspondance par valeurs propres s'ils admettent des paramètres de la forme
  $$ Z_i \in \mathcal{O}(K/K^\#,a, c_i) $$
  pour $c_i \in K^\times$ convenables ($i=1,2$). On note cette relation par
  $$ Z_1 \CVP Z_2 . $$
\end{definition}

\begin{lemma}\label{prop:CVP}
  Avec ces hypothèses, on a
  \begin{gather*}
    \forall u, \; X_u \CVP (Y'_u, Y''_u),, \\
    X_+ \CVP (Y'_+, Y''_-), \\
    X_- \CVP (Y'_-, Y''_+).
  \end{gather*}
\end{lemma}
\begin{proof}
  Si $X,Y$ sont assez petits, on peut écarter les valeurs propres provenant de $u$ différents. Cela permet de conclure.
\end{proof}

La correspondance ci-dessus fournit aussi des décompositions orthogonales $W_u =W'_u \oplus W''_u$ pour tout $u$ et $W_\pm = W'_\pm \oplus W''_\pm$ selon les valeurs propres. Les éléments $X_u, X_\pm$ se décomposent ainsi en
\begin{gather*}
  \forall u, \; X_u =(X'_u, X''_u),\\
  X_+ = (X'_+, X''_+), \\
  X_- = (X'_-, X''_-), \\
\end{gather*}
tels que
\begin{gather*}
  \forall u, \; X'_u \CVP Y'_u, \; X''_u \CVP Y''_u, \\
  X'_+ \CVP Y'_+, \; X''_+ \CVP Y''_-, \\
  X'_- \CVP Y'_-, \; X''_- \CVP Y''_+.
\end{gather*}

Enfin, posons
$$\tilde{\delta} := \exp(X)\tilde{\eta}.$$

Pour tout $u$, prenons $\tilde{u} \in \Mp(W_u)$ au-dessus de $u$ tel que l'image de $((\tilde{u})_u, 1, -1)$ par l'homomorphisme
$$ \prod_u \Mp(W_u) \times \Mp(W_+) \times \Mp(W_-) \to \Mp(W)$$
est $\tilde{\eta}$. C'est possible grâce à l'hypothèse que $\tilde{\eta}=\pm 1$ lorsque $\eta=\pm 1$. Ensuite, prenons $(\tilde{u}', \tilde{u}'') \in \Mp(W'_u) \times \Mp(W''_u)$ qui s'envoie sur $\tilde{u} \in \Mp(W_u)$. Posons $\tilde{\eta}'$ (resp. $\tilde{\eta}''$) l'image de $((\tilde{u}')_u, 1, -1)$ (resp. $((\tilde{u}'')_u, 1, -1)$) par
  \begin{align*}
    \prod_u \Mp(W'_u) \times \Mp(W'_+) \times \Mp(W'_-) & \to \Mp(W') \\
    (\text{resp.}  \prod_u \Mp(W''_u) \times \Mp(W''_+) \times \Mp(W''_-) & \to \Mp(W'')).
  \end{align*}

Décrivons le comportement du facteur de transfert $\Delta = \Delta_0\Delta'\Delta''$.

\begin{proposition}
  Par rapport à ces décompositions, le facteur $\Delta'$ satisfait à
  $$ \Delta'(\exp(X')\tilde{\eta}') = \prod_u \Delta'(\exp(X'_u)\tilde{u}') \cdot \Delta'(\exp(X'_+)) \Delta'(-\exp(X'_-)). $$

  De même,
  $$ \Delta''(\exp(X'')\tilde{\eta}'') = \prod_u \Delta''(\exp(X''_u)\tilde{u}'') \cdot \Delta''(\exp(X''_+)) \Delta''(-\exp(X''_-)). $$

  Le produit $\Delta'\Delta''$ ne dépend pas de choix de $\tilde{\eta}', \tilde{\eta}''$ et $\tilde{u}', \tilde{u}''$
\end{proposition}
Il est sous-entendu que les termes $\Delta'$, $\Delta''$ à droite sont pris par rapport à des espaces symplectiques convenables (eg. $W'_u, W''_u$ etc.)

\begin{proof}
  Cela résulte de \ref{prop:caractere-decomposition-0}.
\end{proof}

Autrement dit, $\Delta'\Delta''$ est ``additif'' par rapport aux sommes directes de paramètres. Le comportement du terme $\Delta_0$ est plus pénible à écrire: il est ``bi-additif''.

\begin{proposition}\label{prop:Delta-biadditif}
  Le facteur $\Delta_0(\exp(X')\eta', \exp(X'')\eta'')$ est produit des termes suivants
  \begin{gather*}
    \Delta_0(\exp(X'_+),\exp(X''_+)),\\
    \Delta_0(-\exp(X'_-), -\exp(X''_-)), \\
    \prod_u \Delta_0(\exp(X'_u)u, \exp(X''_u)u),\\
    \prod_{u' \neq u''} \Delta_0(\exp(X'_{u'})u',\exp(X''_{u''})u''), \\
    \prod_u \Delta_0(\exp(X'_u)u, \exp(X''_+)),\\
    \prod_u \Delta_0(\exp(X'_u)u, -\exp(X''_-)),\\
    \prod_u \Delta_0(\exp(X'_+), \exp(X''_u)u),\\
    \prod_u \Delta_0(-\exp(X'_-), \exp(X''_u)u), \\
    \Delta_0(\exp(X'_+), -\exp(X''_-)), \\
    \Delta_0(-\exp(X'_-),\exp(X''_+)).
  \end{gather*}
\end{proposition}
\begin{proof}
  D'une part, le caractère $\sgn_{K''/K''^\#}(\cdot)$ est additif par rapport aux sommes directes des paramètres $K''/K''^\#$. D'autre part, si $K''/K''^\#$ est fixé, alors
  $$ P_{a'}(a'')(-a'')^{-n'} \quad \text{et} \quad \det(\delta'+1) $$
  sont tous additifs par rapport aux sommes directes de paramètres $(K'/K'^\#,a')$. D'où l'assertion. 
\end{proof}
On démontrera que seuls les trois premiers termes survivent après descente.

\subsection{Le cas non ramifié}\label{sec:descente-nr}
Afin d'établir le lemme fondamental, on aura besoin de considérer la version non ramifiée de la descente. Nous conservons la plupart du formalisme précédent et précisons les modifications ci-dessous.

Conservons l'hypothèse \ref{hyp:non-ramifie}; en particulier, $F$ est un corps local non archimédien de caractéristique résiduelle $p>2$. Fixons un réseau autodual $L \subset W$. Soit $K = \text{Stab}_G(L)$ le sous-groupe hyperspécial de $G(F)$ associé. Cela permet d'identifier $K$ comme un sous-groupe de $\tilde{G}$.

\begin{hypothesis}
  Supposons que
  \begin{itemize}
    \item $\eta$ et $\epsilon$ sont d'ordres finis premiers à $p$;
    \item $\delta,\gamma$ sont des éléments compacts avec décompositions de Jordan topologiques
      \begin{align*}
        \delta &= \exp(X)\eta, \\
        \gamma &= \exp(Y)\epsilon,\\
        X,Y:& \text{ topologiquement nilpotents };
      \end{align*}
    \item $\delta,\gamma$ se correspondent;
    \item $\eta \in K$, $\eta = \tilde{\eta}$;
    \item $H_\epsilon$ est non ramifié.
  \end{itemize}
\end{hypothesis}

Dans ce cas, c'est loisible de supposer que $(V'_\pm, q'_\pm)$, $(V''_\pm, q''_\pm)$, $(V'_u, h'_u)$, $(V''_u, h''_u)$, $(W_u, h_u)$ admettent des réseaux autoduaux (\cite{Wa08} 5.3). La nilpotence topologique de $X,Y$ dans le cas non ramifié remplace la condition précédente que $X,Y$ soient assez petits.

On décompose $G_\eta$ et $H_\epsilon$ comme dans la section précédente. Les éléments $u$ sont d'ordre fini premier à $p$. En particulier, $L$ est une extension non ramifiée; lorsque $L \simeq L^\# \times L^\#$, cela signifie que $L$ est une extension non ramifiée.

\begin{lemma}\label{prop:CVP-nr}
  Avec ces hypothèses, on a
  \begin{gather*}
    \forall u, \; X_u \CVP (Y'_u, Y''_u), \\
    X_+ \CVP (Y'_+, Y''_-), \\
    X_- \CVP (Y'_-, Y''_+).
  \end{gather*}
\end{lemma}
\begin{proof}
   Cela résulte de l'unicité de la décomposition de Jordan topologique.
\end{proof}

\subsection{Énoncé de résultats}
Sauf mention expresse du contraire, $F$ est un corps local de caractéristique nulle. Conservons aussi les formalismes précédents.

Pour simplifier la vie, introduisons des conventions.
\begin{notation}\index{bonne constante}
  On dit qu'une expression est une \textit{bonne constante} si
  \begin{itemize}
    \item elle ne dépend que de $\mathcal{O}^\text{st}(\epsilon)$ et $\mathcal{O}(\tilde{\eta})$;
    \item elle vaut $1$ dans le cas non ramifié.
  \end{itemize}

  Soient $a,b$ deux éléments inversibles dans une $F$-algèbre étale $K$. On dit que $a \approx b$ si
  $$\sup_{\sigma \in \Hom_{F-\text{alg}}(K,\bar{F})} \left| \sigma\left(\frac{a}{b}\right)-1 \right|_{\bar{F}} \; \text{ est assez petit},$$
  où $|\cdot|_{\bar{F}}$ est l'unique valeur absolue sur $\bar{F}$ qui prolonge $|\cdot|_F$. La borne exacte dépendra du contexte.

  Dans le cas non ramifié, on dit que $a \approx b$ si $\frac{a}{b}$ est topologiquement unipotent.
\end{notation}

\begin{theorem}\label{prop:descente-facteur-transfert}
  Posons
  $$ \Delta^\flat(Y,X) := \Delta(\exp(Y)\epsilon, \exp(X)\tilde{\eta}), $$
  alors il existe une bonne constante $c$ telle que
  \begin{align*}
    \Delta^\flat(Y,X) & = c \prod_u \Delta_u((Y'_u,Y''_u), X_u) \cdot \\
    & \cdot \Delta_+((Y'_+,Y''_-),X_+) \cdot \Delta_-((Y'_-,Y''_+),X_-).
  \end{align*}
  Les termes $\Delta_u, \Delta_\pm$ sont définis de la façon suivante. Avec la convention $\bullet \in \{+, -, u\}$, supposons que $X_\bullet \in \mathcal{O}(K/K^\#,a,c)$ et $(K'/K'^\#,a',c') \oplus (K''/K''^\#,a'',c'')$ est la décomposition correspondant à $X_\bullet = (X'_\bullet,X''_\bullet)$. Définissons
  \begin{align*}
    \Delta_u((Y'_u,Y''_u), X_u) & := \sgn_{K''/K''^\#}(\gamma_u c''^{-1} \dot{P}_{X_u|L}(a'')), \; \text{ pour tout } u \\
    \Delta_+((Y'_+,Y''_-),X_+) & := \sgn_{K''/K''^\#}(c''^{-1} \dot{P}_{X_+}(a'')), \\
    \Delta_-((Y'_-,Y''_+),X_-) & := \sgn_{K'/K'^\#}(c'^{-1} \dot{P}_{X_-}(a'));
  \end{align*}
  où $P_{X_\pm} \in F[T]$ est le polynôme caractéristique de $X_\pm \in \End_F(W_\pm)$ et $P_{X_u|L} \in L[T]$ est celui de $X_u \in \End_L(W_u)$. Pour tout $u$, la constante $\gamma_u \in L^\times$ satisfait à $\tau(\gamma_u) = (-1)^{\dim_L W_u} \gamma_u$ et $\gamma_u \in \mathfrak{o}_L^\times$ dans le cas non ramifié.
\end{theorem}
Observons que $P_{X_u|L}$ est bien défini même si $L \simeq L^\# \times L^\#$; de plus, dans ce cas-là $\Delta_u = 1$ car $K'' = K''^\# \otimes_{L^\#} L \simeq K''^\# \times K''^\#$.

\begin{proof}
  Ce théorème s'obtient en multipliant les formules dans \ref{prop:descente-X_+}, \ref{prop:descente-X_-}, \ref{prop:descente-X_u}, \ref{prop:descente-Delta0-principale} et \ref{prop:descente-Delta0-residuelle}.
\end{proof}

\begin{corollary}
  Pour tout $\lambda \in F^\times$, on a
  $$ \Delta^\flat(Y,X) = \Delta^\flat(\lambda^2 Y, \lambda^2 X).$$
\end{corollary}

  Cela nous permet de prolonger $\Delta^\flat$ en une fonction définie sur toute paire
  $$(Y,X) \in (\mathfrak{h}_\epsilon)_{G-\text{reg}}(F) \times (\mathfrak{g}_\eta)_\text{reg}(F).$$

\begin{proof}
  Vu le théorème, il suffit de constater que pour tout $\bullet \in \{+, -, u\}$, le facteur $\Delta_\bullet$ vérifie la même propriété, qui est immédiat.
\end{proof}

\subsection{Des lemmes techniques}
Les lemmes suivants seront les seuls ingrédients non-triviaux dans la démonstration de \ref{prop:descente-facteur-transfert}. On établit d'abord une formule de réciprocité pour $\sgn(\cdot)$. Commençons par une observation élémentaire. Soit $z$ un élément dans une $F$-algèbre étale, $P_{z}$ désigne toujours son polynôme minimal sur $F$.

\begin{lemma}\label{prop:Mobius}
  Fixons un corps $L$. Soient $K$ une $L$-algèbre étale et $z \in K^\times$ tels que $K=F[z]$. Soient
  \begin{gather*}
    \begin{bmatrix} a & b \\ c & d \end{bmatrix} \in \PGL(2,L), \; cz+d \neq 0, \\
    w := \frac{az+b}{cz+d}.
  \end{gather*}

  Soient $P_z \in L[T]$ le polynôme caractéristique de $z$ et $P_w \in L[T]$ celui de $w$. Posons $m := \deg P_z = \deg P_w = \dim_F K$, alors
  \begin{gather*}
    (cT+d)^m P_w\left( \frac{aT+b}{cT+d} \right) = c^m P_w\left( \frac{a}{c} \right) P_z(T), \\
    (ad-bc)(cz+d)^{m-2} \dot{P}_w(w) = c^m P_w\left( \frac{a}{c} \right) \dot{P}_z(z).
  \end{gather*}
\end{lemma}
\begin{proof}
  Il suffit de démontrer la première formule. Observons que $w \neq \frac{a}{c}$, sinon $ad-bc=0$. On se ramène au cas où $K$ est un corps. Le polynôme à gauche s'annule en $z$ et a degré $m$, donc il est égal à $P_z(T)$ multiplié par une constante non nulle. On calcule son coefficient de $T^m$ et on arrive aisément au terme $c^m P_w(\frac{a}{c})$.
\end{proof}

\begin{lemma}\label{prop:reciprocite-Lie}
  Supposons qu'il y a une décomposition orthogonale $W=W' \oplus W''$, $n' := \frac{1}{2} \dim W'$, $n'' := \frac{1}{2} \dim W''$. Soit $X \in \syp(W)$ semi-simple régulier tel que $X = (X',X'')$ où $X' \in \syp(W')$ et $X'' \in \syp(W'')$. Supposons que $X \in \mathcal{O}(K/K^\#,a,c)$ et
  $$(K/K^\#,a,c) = (K'/K'^\#,a',c') \oplus (K''/K''^\#,a'',c'')$$
  par rapport à la décomposition $X=(X',X'')$. Alors
  $$ \sgn_{K'/K'^\#}(P_{X''}(a')) \cdot \sgn_{K''/K''^\#}(P_{X'}(a'')) = (-1, -1)_F^{n'n''} (\det X', \det X'')_F. $$
\end{lemma}
\begin{proof}
  Appliquons \ref{prop:calcul-qX} à $X \in \syp(W)$, $X' \in \syp(W')$ et $X'' \in \syp(W'')$. On voit que
  \begin{align}
    \label{eqn:qX} \gamma_\psi(q[X]) &= \gamma_\psi((-1)^{n-1}) \gamma_\psi(\det X) \cdot  \sgn_{K/K^\#}(c^{-1} \dot{P}_X(a)),\\
    \label{eqn:qX'} \gamma_\psi(q[X']) &= \gamma_\psi((-1)^{n'-1}) \gamma_\psi(\det X') \cdot  \sgn_{K'/K'^\#}(c'^{-1} \dot{P}_{X'}(a')),\\
    \label{eqn:qX''} \gamma_\psi(q[X'']) &= \gamma_\psi((-1)^{n''-1}) \gamma_\psi(\det X'') \cdot  \sgn_{K''/K''^\#}(c''^{-1} \dot{P}_{X''}(a'')).
  \end{align}

  On a aussi $P_X = P_{X'} P_{X''}$ et $q[X]=q[X'] \oplus q[X'']$. En prenant  \eqref{eqn:qX} $\div$ (\eqref{eqn:qX'} $\times$ \eqref{eqn:qX''}), on obtient
  \begin{multline}
    \sgn_{K'/K'^\#}(P_{X''}(a')) \cdot \sgn_{K''/K''^\#}(P_{X'}(a'')) = \\ \frac{\gamma_\psi(\det X)\gamma_\psi(1)}{\gamma_\psi(\det X') \gamma_\psi(\det X'')} \cdot \frac{\gamma_\psi((-1)^{n-1})}{\gamma_\psi(1) \gamma_\psi(-1)^{n'-1})\gamma_\psi((-1)^{n''-1}))}.
  \end{multline}

  En discutant les parités de $n',n''$ et en rappelant que $\gamma_\psi(-1)=\gamma_\psi(1)^{-1}$, on déduit
  $$ \frac{\gamma_\psi((-1)^{n-1})}{\gamma_\psi(1) \gamma_\psi((-1)^{n'-1})\gamma_\psi((-1)^{n''-1})} = \begin{cases} \gamma_\psi(1)^{-4}, & \text{ si } n',n'' \text{ sont impairs}; \\ 1, & \text{ sinon}. \end{cases} $$

  On utilise le fait suivant (\cite{Weil64}, \S 25 prop. 3 et \S 28 prop. 4): pour tout $a,b \in F^\times$, on a
  $$\frac{\gamma_\psi(ab)\gamma_\psi(1)}{\gamma_\psi(a)\gamma_\psi(b)}=(a,b)_F. $$

  D'une part, en prenant $a=b=-1$, il en résulte que $\gamma_\psi(1)^4 = (-1,-1)_F = ((-1)^{n'n''},-1)_F$ si $n',n''$ sont impairs; en tout cas:
  $$ \frac{\gamma_\psi((-1)^{n-1})}{\gamma_\psi(1)\gamma_\psi((-1)^{n'-1})\gamma_\psi((-1)^{n''-1})} = (-1,-1)_F^{n'n''} . $$

  D'autre part, en prenant $a=\det X'$ et $b=\det X''$, il en résulte que
  $$ \frac{\gamma_\psi(\det X)\gamma_\psi(1)}{\gamma_\psi(\det X') \gamma_\psi(\det X'')} = (\det X', \det X'')_F . $$

  Cela achève la démonstration.
\end{proof}

On en déduit une version au niveau du groupe à l'aide de la transformation de Cayley.

\begin{lemma}\label{prop:reciprocite-groupe}
  Soit $\delta \in \Sp(W)_\text{reg}$ tel que $\delta = (\delta',\delta'')$ où $W=W'\oplus W''$, $\delta' \in \Sp(W')$ et $\delta'' \in \Sp(W'')$. Supposons que $\delta \in \mathcal{O}(K/K^\#,a,c)$ et
  $$(K/K^\#,a,c) = (K'/K'^\#,a',c') \oplus (K''/K''^\#,a'',c'')$$
  par rapport à la décomposition $\delta=(\delta',\delta'')$. Alors
  \begin{multline}
    \sgn_{K''/K''^\#}(P_{a'}(a'')(-a'')^{-n'} \det(\delta'-1)) \cdot \sgn_{K'/K'^\#}(P_{a''}(a')(-a')^{-n''} \det(\delta''-1)) \\
    = (-1, -1)_F^{n'n''} \left(\det\frac{\delta'+1}{\delta'-1}, \det\frac{\delta''+1}{\delta''-1} \right)_F.
  \end{multline}
\end{lemma}
\begin{proof}
  Posons $z'=a'+\frac{1}{a'} \in K'^\#$, $z''=a''+\frac{1}{a''} \in K''^\#$; posons d'autre part $b' = \frac{z'+2}{z'-2}$, et idem pour $b''$. Alors \ref{prop:Mobius} affirme que
  $$ P_{z'}(T)P_{b'}(1) = (T-2)^{n'} P_{b'}\left(\frac{T+2}{T-2}\right),$$
  et idem pour $z''$ et $b''$, ce qui entraîne
  \begin{align}
    \label{eqn:reciprocite-g1} P_{z'}(z'')P_{b'}(1) &= (z''-2)^{n'} P_{b'}(b''),\\
    \label{eqn:reciprocite-g2} P_{z''}(z')P_{b''}(1) &= (z'-2)^{n''} P_{b''}(b').
  \end{align}

  On a $\left(\frac{a'+1}{a'-1}\right)^2 = b'$ (idem pour $a''$); on déduit du fait $\tau(\frac{a''+1}{a''-1})=-\frac{a''+1}{a''-1}$ (idem pour $a'$) que $\displaystyle P_{\frac{a''+1}{a''-1}}(T)=P_{b''}(T^2)$, d'où
  $$ P_{b''}(b') = P_{\frac{a''+1}{a''-1}}\left(\frac{a'+1}{a'-1}\right), $$
  et idem si l'on échange $a'$ et $a''$. Observons aussi que $P_{z'}(z'')=P_{a'}(a'')(a'')^{-n'}$ et $P_{z''}(z')=P_{a''}(a')(a')^{-n''}$. Assemblons toutes ces égalités dans \eqref{eqn:reciprocite-g1}, \eqref{eqn:reciprocite-g2} et prenons $\sgn_{K''/K''^\#}(\cdot)$, $\sgn_{K'/K'^\#}(\cdot)$. Pour \eqref{eqn:reciprocite-g1}, cela donne
  $$ \sgn_{K''/K''^\#}(P_{a'}(a'')(a'')^{-n'} P_{\frac{a'+1}{a'-1}}(1)) = \sgn_{K''/K''^\#}\left((z''-2)^{n''} P_{\frac{a'+1}{a'-1}}\left(\frac{a''+1}{a''-1}\right)\right). $$

  Or $z''-2 = -(1-a'')(1-\frac{1}{a''}) \in (-1)N_{K''/K''^\#}(K''^\times)$ et $P_{\frac{a'+1}{a'-1}}(1) = (-2)^{2n'} \det(\delta'-1)^{-1}$, d'où
  $$ \sgn_{K''/K''^\#}(P_{a'}(a'')(-a'')^{-n'} \det(\delta'-1)) = \sgn_{K''/K''^\#}\left( P_{\frac{a'+1}{a'-1}}\left(\frac{a''+1}{a''-1}\right)\right). $$
  Appliquons le même argument à \eqref{eqn:reciprocite-g2} et multiplions les formules qui en résultent. Une application de \ref{prop:reciprocite-Lie} donne le résultat cherché.
\end{proof}

Considérons maintenant une situation différente. On aura besoin de deux lemmes concernant les classes de conjugaison semi-simples de groupes orthogonaux pairs.

\begin{lemma}\label{prop:signe-transfert}
  Soit $(V,q)$ un $F$-espace quadratique de dimension paire. Soit $K/K^\#$ une $F$-algèbre étale à involution. Soit $t \in F^\times$. S'il existe $a,c \in K^\#$ tels que $\tau(a)=-a$, $\tau(c)=c$ et $\mathcal{O}(K/K^\#,a,c)$ existe dans $\SO(V,q)$, alors
  $$ \sgn_{K/K^\#}(t) = (t, (-1)^{\frac{1}{2} \dim_F V} \det q)_F . $$
\end{lemma}
\begin{proof}
  Supposons que $\mathcal{O}(K/K^\#,a,c)$ existe dans $\SO(V,q)$, alors
  \begin{align*}
    (V,q) & \simeq  (K, (\Tr_{K/F})_* (c N_{K/K^\#}(\cdot))), \\
    (V,tq) & \simeq  (K, (\Tr_{K/F})_* (tc N_{K/K^\#}(\cdot))),
  \end{align*}
  où $N_{K/K^\#}$ est la $(K,\tau_K)$-forme hermitienne $w \mapsto N_{K/K^\#}(w)$.

  On a $\gamma_\psi(tq)/\gamma_\psi(q) = \sgn_{K/K^\#}(t)$ d'après \ref{prop:changement-de-signes-0}. D'autre part la formule de $\gamma_\psi$ en termes du déterminant et de l'invariant de Hasse $s(\cdot)$ (\cite{LP81} 1.3.4) entraîne
  $$ \frac{\gamma_\psi(tq)}{\gamma_\psi(q)} = s(tq) \cdot s(q) $$
  car $\dim_F V$ est paire. Posons $m=\frac{1}{2}\dim_F V$. Prenons une diagonalisation $q \simeq \Qform{d_1, \ldots, d_{2m}}$ quelconque. Alors
  \begin{align*}
    s(tq)s(q) & = \prod_{i<j} ((ta_i, ta_j)_F (a_i,a_j)_F) \\
    & = \prod_{i<j} ((t,t)_F (t,a_j)_F (a_i,t)_F) \\
    & = (t,t)_F^m (t, \det q)_F, \;\text{ car } {2m \choose 2} \equiv m \mod 2 \\ 
    & = (t, -1)_F^m (t, \det q)_F,  \;\text{ car } (t,t)_F = (t,-1)_F \\
    & = (t, (-1)^m \det q)_F,
  \end{align*}
  ce qu'il fallait démontrer.
\end{proof}

\begin{lemma}\label{prop:pfaffien}
  Soit $(V,q)$ un $F$-espace quadratique de dimension paire. Si $Y \in \so(V,q)$ est inversible, alors $\det Y \in \det q \cdot F^{\times 2}$.
\end{lemma}
\begin{proof}
  Fixons une base de $V$ de sorte que $q = \Qform{a_1, \ldots, a_{2m}}$ et regardons $Y$ comme une matrice. Soit $Q$ la matrice diagonale $\text{diag}(a_1,\ldots,a_{2m})$, alors $Y \in \so(V,q)$ équivaut à ce que $YQ$ soit une matrice anti-symétrique. Notons $\text{Pf}(YQ) \in F^\times$ son pfaffien, alors
  $$ \det Y = \det YQ  \cdot (\det Q)^{-1} = \text{Pf}(YQ)^2 \cdot (\det Q)^{-1}. $$

  Or $\det Q = a_1 \cdots a_{2m} = \det q$, d'où l'assertion.
\end{proof}

Établissons maintenant la réciprocité du facteur $\Delta_0$.
\begin{corollary}\label{prop:reciprocite}
  Sous les hypothèses de \ref{prop:reciprocite-groupe}, on a
  $$ \Delta_0(\delta',\delta'') = \Delta_0(-\delta'', -\delta'), $$
  où $\Delta_0(-\delta'', -\delta')$ est défini par rapport au groupe endoscopique transposé $H'' \times H'$.
\end{corollary}
\begin{proof}
  D'après \ref{prop:reciprocite-groupe}, on a
  \begin{multline*}
    \Delta_0(\delta', \delta'') \Delta_0(-\delta'', -\delta') = \sgn_{K''/K''^\#}\left( \det\frac{\delta'+1}{\delta'-1} \right) \sgn_{K'/K'^\#}((-1)^{n''}) \\
    \cdot (-1, -1)_F^{n'n''} \left(\det\frac{\delta'+1}{\delta'-1}, \det\frac{\delta''+1}{\delta''-1} \right)_F.
  \end{multline*}

  La classe de conjugaison contenant $\frac{\delta'+1}{\delta'-1} \in \syp(W')$ est paramétrée par $(K'/K'^\#,\frac{a'+1}{a'-1},c')$; prenons $q'$ la $F$-forme quadratique sur $K'$ définie par $(\Tr_{K'^\#/F})_*(c' N_{K'/K'^\#}(\cdot))$. Il existe $c'_0 \in K'^{\#\times}$ et $Y' \in \so(K',q')$ tels que $\mathcal{O}(Y')$ est paramétré par $(K'/K'^\#, \frac{a'+1}{a'-1}, c'_0)$. Appliquons la même construction à $a''$ pour obtenir une $F$-forme quadratique $q''$ sur $K''$. Par \ref{prop:signe-transfert}, on a alors
  \begin{align*}
    \sgn_{K''/K''^\#}\left( \det\frac{\delta'+1}{\delta'-1} \right) &= \left( \det\frac{\delta'+1}{\delta'-1}, (-1)^{n''} \det q'' \right)_F, \\ \sgn_{K'/K'^\#}((-1)^{n''}) & =  ((-1)^{n''}, (-1)^{n'} \det q')_F .
  \end{align*}

  Par \ref{prop:pfaffien}, appliqué à $Y' \in \so(K',q')$ et $Y'' \in \so(K'',q'')$, on a
  \begin{align*}
    \sgn_{K''/K''^\#}\left( \det\frac{\delta'+1}{\delta'-1} \right) &= \left( \det\frac{\delta'+1}{\delta'-1}, (-1)^{n''} \det\frac{\delta''-1}{\delta''+1} \right)_F ,\\
    \sgn_{K'/K'^\#}((-1)^{n''}) &= \left((-1)^{n''}, (-1)^{n'} \det\frac{\delta'+1}{\delta'-1} \right)_F .
  \end{align*}

  On en déduit que $\Delta_0(\delta',\delta'')\Delta_0(-\delta'', -\delta') = (-1,-1)_F^{n'n''} ((-1)^{n'}, (-1)^{n''})_F = 1$. Cela achève la démonstration.
\end{proof}

\subsection{Descente des termes $\Delta'$, $\Delta''$}

\begin{proposition}\label{prop:descente-X_+}
  Supposons que $X_+ \in \mathcal{O}(K/K^\#, a, c)$, soient $X_+=(X'_+,X''_+)$ et $$(K/K^\#,a,c)=(K'/K'^\#,a',c') \oplus (K''/K''^\#,a'',c'')$$ la décomposition de paramètres correspondante. Il existe alors une bonne constante $c_+$ telle que
  $$ \Delta'(\exp(X'_+))\Delta''(\exp(X''_+)) = c_+ \sgn_{K''/K''^\#}(c''^{-1} \dot{P}_{X''_+}(a'')).$$
\end{proposition}

\begin{proposition}\label{prop:descente-X_-}
  Supposons que $X_- \in \mathcal{O}(K/K^\#, a, c)$, $X_-=(X'_-,X''_-)$ et $(K/K^\#,a,c)=(K'/K'^\#,a',c') \oplus (K''/K''^\#,a'',c'')$ la décomposition de paramètres correspondante.

  Il existe alors une bonne constante $c_-$ telle que
  $$ \Delta'(-\exp(X'_-))\Delta''(-\exp(X''_-)) = c_- \sgn_{K'/K'^\#}(c'^{-1} \dot{P}_{X'_+}(a')).$$
\end{proposition}

Les assertions \ref{prop:descente-X_+}, \ref{prop:descente-X_-} sont équivalentes. En effet, d'après \ref{prop:-1-echangement},
  \begin{align*}
    \Delta'(\exp(X'_+))\Delta''(\exp(X''_+)) & = \Delta'(-\exp(X''_+)) \Delta''(-\exp(X'_+)) \\
    \Delta'(-\exp(X'_-))\Delta''(-\exp(X''_-)) & = \Delta'(\exp(X''_-)) \Delta''(\exp(X'_-)).
  \end{align*}
Par conséquent, il suffit d'établir \ref{prop:descente-X_+}.

\begin{proof}[Démonstration de \ref{prop:descente-X_+}]
  On a $\Delta'(\exp(X'_+))=1$ par \ref{prop:-1-echangement}, \ref{prop:-1-caracterisation} et la lissité de $\Theta_\psi$; dans le cas non ramifié, on utilise \ref{prop:formule-caractere-K-2}. D'après \ref{prop:formule-caractere-Lie}, il existe une bonne constante $c$ telle que
  $$ \Delta'(\exp(X'_+))\Delta''(\exp(X''_+)) = c \cdot \gamma_\psi(q[X''_+]). $$
  Vu \ref{prop:calcul-qX}, on a
  $$ \gamma_\psi(q[X''_+]) = \gamma_\psi((-1)^{n''-1})\gamma_\psi(\det X''_+) \cdot \sgn_{K''/K''^\#}(c''^{-1} \dot{P}_{X''_+}(a'')). $$
  Le terme $\gamma_\psi((-1)^{n''-1})$ étant une bonne constante, il suffit de montrer que $\gamma_\psi(\det X''_+)$ l'est aussi. Rappelons que $X''_+ \CVP Y''_-$ et $Y''_- \in \so(V''_-, q''_-)$, $V''_-$ est de dimension paire. D'où $\gamma_\psi(\det X''_+) = \gamma_\psi(\det q''_-)$ par \ref{prop:pfaffien}, qui est une bonne constante car $\det q''_- \in \mathfrak{o}_F^\times$ dans le cas non ramifié.
\end{proof}

\begin{proposition}\label{prop:descente-X_u}
  Fixons la donnée $u$ et posons $L=F[u]$. Supposons que $X_u \in \mathcal{O}(K/K^\#, a, c)$, soient $X_u=(X'_u,X''_u)$ et $(K/K^\#,a,c)=(K'/K'^\#,a',c') \oplus (K''/K''^\#,a'',c'')$ la décomposition de paramètres correspondante.

  Il existe alors une bonne constante $c_u$ et une constante $\alpha_u \in L^\times$, $\alpha_u \in \mathfrak{o}_L^\times$ dans le cas non ramifié, $\tau(\alpha_u)=(-1)^{\dim_L W''_u}\alpha_u$, telles que
  $$ \Delta'(\exp(X'_u)\tilde{u}')\Delta''(\exp(X''_u)\tilde{u}'') = c_u \sgn_{K''/K''^\#}(\alpha_u c''^{-1} \dot{P}_{X''_u|L}(a'')). $$
\end{proposition}
\begin{proof}
  D'après \ref{prop:Delta-variante},
  $$ \Delta'(\exp(X'_u)\tilde{u}')\Delta''(\exp(X''_u)\tilde{u}'') = \Delta'(\exp(X_u)\tilde{u}) \gamma_\psi(q[C_{\exp(X''_u)u''}]). $$

  Puisque $\det(u^2-1) \neq 0$, le terme $\Delta'(\exp(X_u)\tilde{u})=\Delta'(\tilde{u})$ est une bonne constante d'après \ref{prop:-1-echangement}; dans le cas non ramifié, il faut aussi \ref{prop:formule-caractere-K-2}. Il reste à traiter $\gamma_\psi(q[C_{\exp(X''_u)u''}])$.

  On a $K'' = K''^\# \otimes_L^\# L$. Si $L \simeq L^\# \times L^\#$, alors la forme $q[C_{\exp(X''_u)u''}]$ est hyperbolique d'après \ref{rem:hyperbolicite}, donc $\gamma_\psi(q[C_{\exp(X''_u)u''}])=1$. D'autre part $\sgn_{K''/K''^\#}(\cdot)=1$ dans ce cas-là et la proposition est prouvée. Supposons désormais que $L$ est un corps. Fixons un plongement $L \hookrightarrow \bar{F}$; la formule finale sera indépendante du plongement.

  Posons
  \begin{align*}
    m'' & := \dim_L W''_u, \\
    A'' & := \exp(a''), \\
    C'' & := \frac{uA''-1}{uA''+1}.
  \end{align*}
  On voit que $A'' \approx 1$ et $C'' \approx \frac{u-1}{u+1} \neq 0$. Grâce à \ref{prop:calcul-qX},
  \begin{equation}\label{eqn:X_u-1}\begin{split}
    \gamma_\psi(q[C_{\exp(X''_u)u''}]) & = \gamma_\psi((-1)^{[L:F]m'' - 1}) \gamma_\psi(N_{K''/F}(2C'')) \cdot \\
    & \cdot \sgn_{K''/K''^\#}(c''^{-1} \dot{P}_{2C''}(2C'')).
  \end{split}\end{equation}

  Le terme $\gamma_\psi((-1)^{[L:F]m'' - 1})$ est une bonne constante, le terme $\gamma_\psi(N_{K''/F}(2C''))$ est égal à
  $$\gamma_\psi\left(N_{K''/F}\left(2\cdot\frac{u-1}{u+1}\right)\right) = \gamma_\psi(\det_F (2(u+1)(u-1)^{-1}|W''_u)),$$
  ce qui est aussi une bonne constante. Il reste donc à traiter le terme $\sgn_{K''/K''^\#}(\cdots)$.

  Enfin, $\sgn_{K''/K''^\#}(c''^{-1} \dot{P}_{2C''}(2C'')) = \sgn_{K''/K''^\#}(2c''^{-1} \dot{P}_{C''}(C''))$. Posons
  $$ \Sigma_{L/F} := \Hom_{F-\text{alg}}(L,\bar{F}) $$
  et posons $\sigma_0 \in \Sigma_{L/F}$ l'unique élément fixant $u$.

  Observons que
  $$ P_{C''} = \prod_{\sigma \in \Sigma_{L/F}} P_{C''|L}^\sigma $$
  où la notation $P_{\bullet|L}$ signifie le polynôme caractéristique de $\bullet$ sur $L$, l'action de $\Sigma_{L/F}$ sur $L[T]$ provient de l'action sur les coefficients. On a aussi $\deg P_{C''|L} = m''$. Il en résulte que
  \begin{equation}\label{eqn:X_u-2}\begin{split}
    \dot{P}_{C''}(C'') &= \dot{P}_{C''|L}(C'') \cdot \prod_{\substack{\sigma \in \Sigma_{L/F}, \\ \sigma \neq \sigma_0}} P_{C''|L}^\sigma(C'') \\
    & = k_u \dot{P}_{C''|L}(C''),
  \end{split}\end{equation}
  où $k_u \in L^\times$ est une constante; $k_u \in \mathfrak{o}_L^\times$ dans le cas non ramifié.

  D'après \ref{prop:Mobius}, on a
  $$ 2u (uA''+1)^{m''-2} \dot{P}_{C''|L}(C'') = u^{m''} P_{C''|L}(1) \dot{P}_{A''|L}(A''), $$
  d'où
  \begin{equation}\label{eqn:X_u-3}\begin{split}
    \dot{P}_{C''|L}(C'') &= \frac{1}{2} u^{m''-1} (uA''+1)^{2-m''} P_{C''|L}(1) \dot{P}_{A''|L}(A'') \\
    & \approx \frac{1}{2} u^{m''-1} (u+1)^{2-m''} \det_L (1-(u-1)(u+1)^{-1}| W''_u) \dot{P}_{A''|L}(A'') \\
    & \approx h_u \dot{P}_{a''|L}(a''),
  \end{split}\end{equation}
  où $h_u \in L^\times$ est une constante; $h_u \in \mathfrak{o}_L^\times$ dans le cas non ramifié.

  En combinant \eqref{eqn:X_u-1},\eqref{eqn:X_u-2},\eqref{eqn:X_u-3}, on voit qu'il existe une bonne constante $c_u$ et une constante $\alpha_u \in L^\times$, $\alpha_u \in \mathfrak{o}_L^\times$ dans le cas non ramifié, telles que
  $$ \gamma_\psi(q[C_{\exp(X''_u)u''}]) = c_u \cdot \sgn_{K''/K''^\#}(\alpha_u c''^{-1} \dot{P}_{a''|L}(a'')). $$

  Pour que le terme dans $\sgn_{K''/K''^\#}$ appartienne à $K''^\#$, il faut que $\tau(\alpha_u) \in (-1)^{m''}\alpha_u$. Cela achève la démonstration.

\end{proof}

\subsection{Descente du terme $\Delta_0$}
Nous considérons d'abord les trois premiers termes des dix termes dans \ref{prop:Delta-biadditif}.

\begin{proposition}\label{prop:descente-Delta0-principale}
  Soit $\bullet \in \{+,-,u\}$. Supposons que
  \begin{align*}
    X'_\bullet &\in \mathcal{O}(K'/K'^\#, a', c'), \\
    X''_\bullet & \in \mathcal{O}(K''/K''^\#,a'',c'');
  \end{align*}
  alors il existe des bonnes constantes $d_\bullet$ telles que
  \begin{enumerate}
    \item $\Delta_0(\exp(X'_+),\exp(X''_+)) = d_+ \cdot \sgn_{K''/K''^\#}(P_{X'_+}(a'')),$
    \item $\Delta_0(-\exp(X'_-),-\exp(X''_-)) = d_- \cdot \sgn_{K'/K'^\#}(P_{X''_-}(a')),$
    \item $\Delta_0(\exp(X'_u)u,\exp(X''_u)u) = d_u \cdot \sgn(\beta_u P_{X'_u|L}(a''))$, où $P_{X'_u|L}$ est le polynôme caractéristique de $X'_u$ sur $L$, et $\beta_u \in L^\times$ est une constante, $\tau(\beta_u) = (-1)^{\dim_L W'_u}\beta_u$; $\beta_u \in \mathfrak{o}_L^\times$ dans le cas non ramifié.
  \end{enumerate}
\end{proposition}

Dans les démonstrations ci-dessous, nous donnerons des formes explicites pour $d_+,d_-,d_u$ et $\beta_u$. Vu \ref{prop:reciprocite}, il suffit d'établir le cas $\bullet=+$ et $\bullet=u$.

\begin{proof}[Démonstration pour $\bullet=+$]
  Posons
  \begin{align*}
    m' & := \frac{1}{2} \dim_F W'_+, \\
    m'' & := \frac{1}{2} \dim_F W''_+, \\
    A' & := \exp(a'),\\
    A'' & := \exp(a'').
  \end{align*}

  Vu la définition de $\Delta_0$, on a
  \begin{align*}
    \Delta_0(\exp(X'_+),\exp(X''_+)) & = \sgn_{K''/K''^\#}(P_{A'}(A'')(-A'')^{-m'} 2^{2m'}) \\
    & = \sgn_{K''/K''^\#}(-1)^{m'} \sgn_{K''/K''^\#}(P_{a'}(a'')),
  \end{align*}
  où on a utilisé le fait $P_{A'}(A'') \approx P_{a'}(a'')$.

  Montrons que $\sgn_{K''/K''^\#}(-1)$ est une bonne constante.  Il y a une correspondance $X''_+ \CVP Y''_- \in \so(V''_-,q''_-)$. D'après \ref{prop:signe-transfert},
  $$ \sgn_{K''/K''^\#}(-1) = (-1, (-1)^{m''} \det q''_-)_F. $$

  Prenons $d_+ = ((-1)^{m'}, (-1)^{m''} \det q''_-)_F$, c'est une bonne constante.
\end{proof}

\begin{proof}[Démonstration pour $\bullet=u$]
  La démonstration est analogue à celle de  $\Delta', \Delta''$. Si $L=F[u] \simeq L^\# \times L^\#$, alors les deux côtés valent $1$ et on peut prendre $d_u=1$, $\beta_u$ quelconque. Supposons désormais que $L$ est un corps et fixons un plongement $L \hookrightarrow \bar{F}$.

  Posons
  \begin{align*}
    m' & := \dim_L W'_u, \\
    m'' & := \dim_L W''_u, \\
    A' & := \exp(a'), \\
    A'' & := \exp(a''), \\
    \Sigma_{L/K} & := \Hom_{F-\text{alg}}(L,\bar{F}).
  \end{align*}
  Regardons $L$ comme un sous-corps de $\bar{F}$. Notons $\sigma_0 \in \Sigma_{L/K}$ l'unique élément fixant $u$, alors
  $$ P_{\exp(X'_u)u} = P_{A'u} = \prod_{\sigma \in \Sigma_{L/K}} P_{A'u|L}^\sigma, $$
  où $P_{A'u|L}$ est le polynôme caractéristique de $A'u$ sur $L$, on a $\deg P_{A'u|L}=m'$.

  Si $\sigma \in \Sigma_{L/F}, \sigma \neq \sigma_0$, alors
  $$ P_{A'u|L}^\sigma(uA'') \approx (u-\sigma(u))^{m'} \neq 0 $$
  D'autre part
  \begin{align*}
    P_{A'u|L}^{\sigma_0}(uA'') &= P_{A'u|L}(uA'') \approx u^{m'} P_{A'|L}(A'') \\
    & \approx u^{m'} P_{a'|L}(a'').
  \end{align*}

  Posons
  $$ \delta(u) := \prod_{\substack{\sigma \in \Sigma_{L/F} \\ \sigma \neq \sigma_0}}(u-\sigma(u)) \in L^\times . $$

  On vérifie que
  $$ \beta_u := ((-u)^{-[L:F]}u\delta(u))^{m'} \in L^\times $$
  satisfait à $\tau(\beta_u)=(-1)^{m'} \beta_u$; il en est de même pour $P_{X'_u|L}(a'')$. On a aussi $\beta_u \in \mathfrak{o}_L^\times$ dans le cas non ramifié. D'où
  \begin{align*}
    \Delta_0(\exp(X'_u)u,\exp(X''_u)u) &= \sgn_{K''/K''^\#} (\beta_u P_{X'_u|L}(a'') \det_F(u+1|W'_u)) \\
    & = \sgn_{K''/K''^\#} (\beta_u P_{X'_u|L}(a'')) \sgn_{K''/K''^\#}(N_{L/F}(u+1)).
  \end{align*}

  Soit $q''_u$ la $F$-forme quadratique sur $K''$ définie par $(\Tr_{L/F})_*(h''_u)$. Il y a une correspondance $X''_u \CVP Y''_u \in \mathfrak{u}(V''_u, h''_u)$, donc
  $$\sgn_{K''/K''^\#}(N_{L/F}(u+1)) = (N_{L/F}(u+1), (-1)^{m''[L^\#:F]} \det q''_u)_F$$
  d'après \ref{prop:signe-transfert}; notons-le $d_u$. C'est une bonne constante et on arrive à

  $$ \Delta_0(\exp(X'_u)u,\exp(X''_u)u) = d_u \cdot \sgn_{K''/K''^\#}(\beta_u P_{X'_u|L}(a'')) .$$

  Remarquons que les définitions de $d_u, \beta_u$ ont aussi un sens dans le cas $L=L^\# \times L^\#$.
\end{proof}

Traitons maintenant les sept autres termes dans \ref{prop:Delta-biadditif}.

\begin{proposition}\label{prop:descente-Delta0-residuelle}
  Il existe des bonnes constantes $c_{u',u''}$ (pour tous $u' \neq u''$), $c_{u,+}$, $c_{u,-}$, $c_{+,u}$, $c_{-,u}$, $c_{+,-}$, $c_{-,+}$ telles que
  \begin{gather*}
    \Delta_0(\exp(X'_{u'})u',\exp(X''_{u''})u'') = c_{u',u''}, \\
    \Delta_0(\exp(X'_u)u, \exp(X''_+)) = c_{u,+},\\
    \Delta_0(\exp(X'_u)u, -\exp(X''_-)) = c_{u,-},\\
    \Delta_0(\exp(X'_+), \exp(X''_u)u) = c_{+,u},\\
    \Delta_0(-\exp(X'_-), \exp(X''_u)u) = c_{-,u}, \\
    \Delta_0(\exp(X'_+), -\exp(X''_-)) = c_{+,-}, \\
    \Delta_0(-\exp(X'_-),\exp(X''_+)) = c_{-,+}.
  \end{gather*}
\end{proposition}

\begin{proof}
  Pour tout $\bullet \in \{u',u'',+,-\}$, on suppose que $X'_\bullet \in \mathcal{O}(K'/K'^\#, a', c')$, $X''_\bullet \in \mathcal{O}(K''/K''^\#, a'', c'')$. Vu \ref{prop:reciprocite}, il suffit d'établir les cas $(u',u'')$, $(u,+)$, , $(+,u)$ et $(+,-)$.

  \paragraph{Le cas $\mathbf{(u',u'')}$.} Posons
  \begin{align*}
    m' &:= \frac{1}{2}\dim_F W'_{u'},\\
    m'' &:= \frac{1}{2}\dim_F W''_{u''},\\
    L'' &:= F[u''].
  \end{align*}

  Alors
  $$ P_{\exp(a')u'}(\exp(a'')u'')(-\exp(a'')u'')^{-m'} \approx P_{u'}(u'')(-u')^{-m'},$$
  où $P_{u'}$ est le polynôme caractéristique de $u' \in \End_F(W'_{u'})$. On vérifie que $P_{u'}(u'')(-u'')^{-m'} \in L^{\#\times}$. On en déduit que
  \begin{align*}
    \Delta_0(\exp(X'_{u'})u',\exp(X''_{u''})u'') & =  \sgn_{K''/K''^\#}(P_{u'}(u'')(-u'')^{-m'}) \cdot \\
    & \cdot \sgn_{K''/K''^\#}(\det_F(u'+1|W'_{u'})).
  \end{align*}

  Soit $q'' := (\Tr_{L''/L''^\#})_*(h''_{u''})$. Il y a une correspondance $X''_{u''} \CVP Y''_{u''} \in \so(V''_{u''}, q'')$. On déduit par \ref{prop:signe-transfert} que
  \begin{align*}
    \sgn_{K''/K''^\#}(P_{u'}(u'')(-u'')^{-m'}) & = (P_{u'}(u'')(-u'')^{-m'}, (-1)^{\frac{m''}{[L''^\#:F]}} \det q'' )_{L''^\#} \\
    & =: c_1 .
  \end{align*}

  Toujours par \ref{prop:signe-transfert}, on a aussi
  \begin{align*}
    \sgn_{K''/K''^\#}(\det_F(u'+1|W'_{u'})) & = (\det_F(u'+1|W'_{u'}), (-1)^{m''} \det((\Tr_{L''^\#/F})_* q''))_F \\
    & = (N_{L'/F}(u'+1), (-1)^{m''} \det((\Tr_{L''^\#/F})_* q''))_F \\
    & =: c_2 .
  \end{align*}
  Prenons $c_{u',u''} := c_1 c_2$, c'est une bonne constante car $c_1$ et $c_2$ le sont.

  \paragraph{Le cas $\mathbf{(u,+)}$.} Posons $m' := \frac{1}{2}\dim_F W'_u$, $m'' := \frac{1}{2} \dim_F W''_+$. On a $P_{\exp(a')u}(\exp(a'')) \approx P_u(1) \in F^\times$. D'où
  \begin{align*}
    \Delta_0(\exp(X'_u)u, X''_+) & = \sgn_{K''/K''^\#}(P_{u}(1)(-1)^{-m'}) \sgn_{K''/K''^\#}(\det_F(u+1|W'_u)) \\
    & = \sgn_{K''/K''^\#}(P_{u}(1)(-1)^{-m'}) \sgn_{K''/K''^\#}(N_{L/F}(u+1)) \\
    & = \sgn_{K''/K''^\#}(\det_F(1-u|W'_u)(-1)^{-m'}) \sgn_{K''/K''^\#}(N_{L/F}(u+1)).
  \end{align*}

  Il y a une correspondance $X''_+ \CVP Y''_- \in \so(V''_-,q''_-)$. On déduit par \ref{prop:signe-transfert} que
  \begin{align*}
    \sgn_{K''/K''^\#}(\det_F(1-u|W'_u)(-1)^{-m'}) & = (\det_F(1-u|W'_u)(-1)^{-m'}, (-1)^{m''} \det q''_-)_F \\
    & =: c_1,
  \end{align*}
  qui est une bonne constante.

  De même, on a
  \begin{align*}
    \sgn_{K''/K''^\#}(N_{L/F}(u+1)) & = (N_{L/F}(u+1) ,(-1)^{m''} \det q''_-)_F \\
    & =: c_2,
  \end{align*}
  c'est aussi une bonne constante. Prenons $c_{u,+} := c_1 c_2$.

  \paragraph{Le cas $\mathbf{(+,u)}$.} Posons
  \begin{align*}
    m' &:= \frac{1}{2}\dim_F W'_+, \\
    m'' &:= \frac{1}{2} \dim_F W''_u, \\
    L & := F[u].
  \end{align*}
  Comme dans le cas $(u,+)$, on arrive à
  \begin{align*}
    \Delta_0(\exp(X'_+), \exp(X''_u)u) & = \sgn_{K''/K''^\#}(P_1(u)(-u)^{-m'} 2^{2m'}) \\
    & = \sgn_{K''/K''^\#}(P_1(u)(-u)^{-m'}),
  \end{align*}
  où $P_1(T) = (T-1)^{2m'}$ est le polynôme caractéristique de $1 \in \End_F(W'_+)$.

  Soit $q'' := (\Tr_{L''/L''^\#})_*(h''_{u''})$. Comme dans le cas $(u',u'')$, on conclut par \ref{prop:signe-transfert} que
  $$ \sgn_{K''/K''^\#}(P_1(u)(-u)^{-m'}) = ((u-1)^{2m'}(-u)^{-m'}, (-1)^{\frac{m''}{[L:F]}} \det q'')_{L^\#}. $$

  Prenons-le comme la bonne constante $c_{+,u}$.

  \paragraph{Le cas $\mathbf{(+,-)}$.} Posons $m' := \frac{1}{2} \dim_F W'_+$. Alors
  \begin{align*}
    \Delta_0(\exp(X'_+), -\exp(X''_-)) &= \sgn_{K''/K''^\#}(P_1(-1) 2^{2m'}) \\
    & = \sgn_{K''/K''^\#}((-2)\cdot 2)^{2m'} = 1.
  \end{align*}
  Prenons donc $c_{+,-}=1$.
\end{proof}

\subsection{Comparaison avec les facteurs de transfert des groupes classiques}
Supposons $F$ non archimédien. Les facteurs de transfert pour les algèbres de Lie des groupes classiques quasi-déployés sont décrits dans \cite{Wa01}, Chapitre X. Comme dans \cite{Wa01}, la correspondance de points est la correspondance par valeurs propres.

Relions maintenant \ref{prop:descente-facteur-transfert} et les facteurs de transfert pour les groupes classiques.

\begin{theorem}\label{prop:comparaison-transfert-classique}
  Supposons que $H_\epsilon$ est quasi-déployé. Les facteurs dans \ref{prop:descente-facteur-transfert} satisfont à\index{facteur de transfert!pour les groupes classiques}:
  \begin{itemize}
   \item pour tout $u$, $\Delta_u((Y'_u,Y''_u), X_u)$ est le facteur de transfert pour le groupe endoscopique $U(V'_u, h'_u) \times U(V''_u, h''_u)$ de $U(W_u, h_u)$;
   \item $\Delta_+((Y'_+,Y''_-), X_+)$ est le facteur de transfert pour le groupe endoscopique \newline$\Sp(W'_+) \times \SO(V''_-,q''_-)$ de $\Sp(W_+)$ évalué en $((X'_+, Y''_-),X_+)$;
   \item $\Delta_-((Y'_-,Y''_+), X_-)$ est le facteur de transfert pour le groupe endoscopique $\SO(V'_-,q'_-) \times \Sp(W''_-)$ de $\Sp(W_-)$ évalué en $((Y'_-, X''_-),X_-)$.
  \end{itemize}

  De plus, ce sont les facteurs de transfert normalisés au sens de \cite{Wa08} \S 4.7 dans le cas non ramifié.
\end{theorem}
\begin{proof}
  Prenons garde (cf. \ref{rem:c-normalisation}) que notre paramétrage de classes de conjugaison est différent que celui de \cite{Wa01} (cf. \ref{rem:c-normalisation}). Il faut aussi les observations ci-dessous.

  \begin{itemize}
    \item Il s'agira de groupes classiques sur les corps $L := F[u]$. Cela ne gène pas car le formalisme de l'endoscopie est compatible avec la restriction des scalaires.
    \item Nos formes $(W_u,h_u)$ sont anti-hermitiennes, pourtant celles de \cite{Wa01} sont hermitiennes; cela n'affecte pas le groupe unitaire, mais cela change le choix de paramètres ($\tau(c)=-c$ au lieu de $\tau(c)=c$) et la description de formes dans \cite{Wa01} X.3.
    \item Afin d'étendre les formules pour le facteur de transfert de \cite{Wa01} aux groupes classiques non quasi-déployés, on utilise \cite{LS1} (4.2) (il faut l'adapter à l'algèbre de Lie). Soit $M$ un groupe unitaire ou orthogonal sur $F$, fixons une forme intérieure quasi-déployée $M^*$ et un torseur intérieur $\psi: M \times_F \bar{F} \rightiso M^* \times_F \bar{F}$. Il faut calculer un invariant défini comme dans \cite{LS1} (3.4), avec des notations évidentes:
    $$ \inv\left( \dfrac{Y,X}{Y',X'} \right). $$

    Pour des tels groupes, on peut choisir $(M^*, \psi)$ dont la classe de cohomologie provient de $H^1(F, M)$ au lieu de $H^1(F, M_\text{ad})$. Cela permet d'éviter les constructions de doublements de \cite{LS1}. Un calcul explicite analogue à celui dans \cite{Wa01} X permet de conclure.
      
    \item Les cas où $\SO(V'_-,q'_-)$ ou $\SO(V''_-,q''_-)$ est isomorphe à $\Gm$ sont exclus dans \cite{Wa01} car ils rendent l'endoscopie non elliptique; pour la même raison, on ne considère pas le cas $L \simeq L^\# \times L^\#$ (qui revient à l'endoscopie pour $\GL$.) Or les formules de Waldspurger restent valables dans ces cas-là: elles valent la constante $1$, et il en est de même pour nos formules.
    \item A cause de notre définition d'intégrales orbitales, nous avons supprimé le facteur $\Delta_{IV}$ dans le facteur de transfert de \cite{Wa01}.
  \end{itemize}

  Après des modifications, l'identification de $\Delta_\pm$ résultent immédiatement. Quant à $\Delta_u$, il suffit de considérer le cas où $L=F[u]$ est un corps et $U(W_u,h_u)$ est quasi-déployé. Posons $d := \dim_L W_u$, on peut choisir $\gamma \in L^\times$ et une base de $\{e_j : 1 \leq j \leq d\}$ de $W_u$, de sorte que
  \begin{align*}
    \tau(\gamma) &= (-1)^d \gamma, \\
    h_u(e_j, e_k) &= \begin{cases}
      0, & \text{ si } j+k \neq d+1, \\
      2\gamma(-1)^{j+1}, & \text{ si } j+k=d+1.
    \end{cases}
  \end{align*}
  On peut prendre $\gamma \in \mathfrak{o}_L^\times$ dans le cas non ramifié.

  Supposons que $X_u \in \mathcal{O}(K/K^\#, a, c)$, $X_u=(X'_u, X''_u)$ et $(K/K^\#,a,c)=(K'/K'^\#,a',c') \oplus (K''/K''^\#,a'',c'')$ la décomposition de paramètres correspondante. Alors le facteur de transfert pour le groupe endoscopique $U(V'_u, h'_u) \times U(V''_u, h''_u)$ de $U(W_u, h_u)$ est, à une constante multiplicative près,
  $$ \sgn_{K''/K''^\#}(\gamma c''^{-1} \dot{P}_{X_u}(a'')). $$

  D'autre part, il existe $\gamma_u \in L^\times$, $\gamma_u \in \mathfrak{o}_L^\times$ dans le cas non ramifié tel que $\tau(\gamma_u)=(-1)^d \gamma_u$ et $\Delta_u((Y'_u,Y''_u), X_u)$ est égal à
  $$ (\text{bonne constante}) \cdot \sgn_{K''/K''^\#}(\gamma_u c''^{-1} \dot{P}_{X_u}(a'')). $$

  Il suffit de montrer que $\sgn_{K''/K''^\#}(\frac{\gamma}{\gamma_u})$ est une bonne constante. Cela résulte de \ref{prop:signe-transfert} car $X''_u \CVP Y''_u$ et $U(V''_u,h''_u) \subset \SO(V''_u, (\Tr_{L/L^\#})_* h''_u)$.

  Quant à la normalisation, voir la remarque à la fin de \cite{Wa01} X.
\end{proof}

\section{Transfert: le cas non archimédien}
Dans cette section, $F$ est toujours un corps local non archimédien de caractéristique nulle, $\Mp(W)$ désigne le revêtement métaplectique à huit feuillets $\MMp{8}(W)$ de $\Sp(W)$.

\subsection{Voisinages d'un élément semi-simple}\label{sec:voisinage-ss}
Soit $M$ un $F$-groupe réductif connexe ou un revêtement d'un tel groupe vérifiant \ref{hyp:revetement-commutant}. Dans le cas d'un revêtement $\rev: M \to \underline{M}(F)$, si $\eta \in M_\text{ss}$, posons $M^\eta := \rev^{-1}(\underline{M}^{p(\eta)}(F))$ et $M_\eta := \rev^{-1}(\underline{M}_{p(\eta)}(F))$. Dans le cas d'un groupe réductif, on confond systématiquement $M$ et $M(F)$.

Pour tout $\eta \in M_\text{ss}$ et tout ouvert $M^\eta$-invariant $\mathfrak{U}' \subset \mathfrak{m}_\eta$ contenant $0$, il existe un ouvert $M^\eta$-invariant $\mathfrak{U} \subset \mathfrak{U}'$ tel que

\begin{enumerate}
  \item $\exp: \mathfrak{U} \to M_\eta$ est défini et est un homéomorphisme sur son image, qui est ouvert dans $M_\eta$;
  \item l'application $(X,x) \mapsto x^{-1} \exp(X)\eta x$ de $\mathfrak{U} \times M$ sur $M$ est partout submersive, son image $\mathfrak{U}^\natural$ est un ouvert $M^\eta$-invariant de $M$;
  \item si $x \in M$ et s'il existe $X \in \mathfrak{U}$ tel que $x^{-1} \exp(X) \eta x \in \exp(\mathfrak{U})\eta$, alors $x \in M^\eta$.
%  \item posons $\mathfrak{U}_\text{reg} := \mathfrak{U} \cap \mathfrak{m}_{\eta,\text{reg}}$, si $X \in \mathfrak{U}_\text{reg}$ alors $\exp (X)\eta$ est fortement régulier.
\end{enumerate}

De plus, tout ouvert $M$-invariant dans $M$ contenant $\eta$ contient un ouvert de la forme $\mathfrak{U}^\natural$\index{$\mathfrak{U},\mathfrak{U}^\natural$}. Lorsque $M$ est un $F$-groupe réductif et $\mathfrak{U}'$ est invariant par conjugaison géométrique, on peut supposer de plus que
\begin{itemize}
  \item $\mathfrak{U}$ est invariant par conjugaison géométrique sous $M^\eta$;
  \item si $x \in M(\bar{F})$ et s'il existe $X \in \mathfrak{U}$ tel que $x^{-1} \exp(X) \eta x \in \exp(\mathfrak{U})\eta$, alors $x \in M^\eta(\bar{F})$.
\end{itemize}

Posons $\mathfrak{U}_\text{reg} := \mathfrak{U} \cap \mathfrak{m}_{\eta,\text{reg}}$.

\begin{proposition}[Descente d'intégrales orbitales]\label{prop:descente-int}
  Soit $\mathfrak{U}$ un ouvert comme ci-dessus. Supposons $\mathfrak{U}$ suffisamment petit.
  \begin{itemize}
    \item Soit $f^\natural \in C_c^\infty(\mathfrak{U}^\natural)$, alors il existe $f \in C_c^\infty(\mathfrak{U})$ tel que
    $$ J_{M_\eta}(X,f) = J_M(\exp(X)\eta, f^\natural) $$
    pour tout $X \in \mathfrak{U}_\text{reg}$.
    \item Inversement, soit $f \in C_c^\infty(\mathfrak{U})$ tel que $X \mapsto J_{M_\eta}(X,f)$ est invariant par $M^\eta(F)$, alors il existe $f^\natural \in C_c^\infty(\mathfrak{U}^\natural)$ qui satisfait à l'égalité ci-dessus.
  \end{itemize}

  Supposons que $M$ est un $F$-groupe réductif connexe et l'ouvert $\mathfrak{U}$ est invariant par conjugaison géométrique sous $M^\eta$, alors les énoncés précédents restent vrais pour l'égalité
  $$ J_{M_\eta}^\text{st}(X,f) = J_M^\text{st}(\exp(X)\eta, f^\natural), $$
  et la condition sur $f$ pour l'existence de $f^\natural$ est que $X \mapsto J_{M_\eta}(X,f)$ soit invariant par conjugaison géométrique par $M^\eta$.
\end{proposition}

Ces propriétés sont bien connues, voir par exemple \cite{Wa08} 2.3.

\subsection{Un triplet endoscopique non standard}\index{endoscopie non standard}
Rappelons la définition d'un triplet endoscopique non standard $(G_1, G_2, j_*)$  dans \cite{Wa08} 1.7. Ici $G_1, G_2$ sont des groupes semi-simples connexes et simplement connexes, quasi-déployés sur $F$. Fixons des tores maximaux $T_i \subset G_i$ qui font partie d'une paire de Borel définie sur $F$. Posons $X_{i,*} := X_*(T_i)$, $X_i^* := X^*(T_i)$ et posons $X_{i,*,\Q} := X_{i,*} \otimes_\Z \Q$, $X_{i,\Q}^* := X_i^* \otimes_\Z \Q$. Notons $\Sigma_i \subset X_i^*$ les racines et $\check{\Sigma}_i \subset X_i^*$ les coracines. La donnée $j_*$ est un isomorphisme
$$j_*: X_{1,*,\Q} \rightiso X_{2,*,\Q}. $$

Notons $j^*: X_{2,\Q}^* \rightiso X_{1,\Q}^*$ l'isomorphisme transposé. Ces données sont soumises aux conditions suivantes:

\begin{itemize}
  \item il existe des bijections $\check{\tau}: \check{\Sigma}_1 \to \check{\Sigma}_2$, $\tau: \Sigma_2 \to \Sigma_1$ et des applications $\check{b}: \check{\Sigma}_1 \to \Q_{>0}$, $b: \Sigma_2 \to \Q_{>0}$ telles que $j_*(\check{\alpha}_1) = \check{b}(\check{\alpha}_1)\check{\tau}(\check{\alpha}_1)$ et $j^*(\alpha_2) = b(\alpha_2) \tau(\alpha_2)$ pour tous $\check{\alpha}_1 \in \check{\Sigma}_1$, $\alpha_2 \in \Sigma_2$. De plus, le diagramme suivant est commutatif
  $$\xymatrix{
    \Sigma_1 \ar[d] & \Sigma_2 \ar[d] \ar[l]^{\tau} \\
    \check{\Sigma}_1 \ar[r]^{\check{\tau}} & \check{\Sigma}_2
  }$$
  où $\Sigma_i \to \check{\Sigma}_i$ ($i=1,2$) sont les bijections naturelles de données radicielles.
  \item $j_*$ et $j^*$ sont équivariants pour les actions de $\Gamma_F$.
\end{itemize}

Observons que la définition est symétrique en $G_1$ et $G_2$.

L'application $\check{\tau}$ induit un isomorphisme de groupes de Weyl $W^{G_1} \rightiso W^{G_2}$. On vérifie que ces données donnent naissance à un $F$-isomorphisme naturel
$$ \mathfrak{t}_1/W^{G_1} \to \mathfrak{t}_2/W^{G_2},$$
ce qui permet de définir la correspondance de classes de conjugaison semi-simples dans les algèbres de Lie pour l'endoscopie non standard. Soient $X_i \in \mathfrak{t}_{i,\text{reg}}(F)$ qui se correspondent et notons $T_i$ le commutant de $X_i$ dans $G_i$ ($i=1,2$), comme pour l'endoscopie standard, il y a aussi une correspondance de mesures de Haar sur $T_1(F)$ et $T_2(F)$.

\begin{theorem}[Transfert non standard, \cite{Wa08} 1.8]\label{prop:transfert-non-standard}
  Soit $(G_1, G_2, j_*)$ un triplet endoscopique non standard. Fixons des mesures de Haar sur $G_1(F), G_2(F)$. Si $X_i \in \mathfrak{g}_{i, \text{reg}}$ ($i=1,2$) se correspondent, alors pour tout $f_2 \in C_c^\infty(\mathfrak{g}_2(F))$ il existe $f_1 \in C_c^\infty(\mathfrak{g}_1(F))$ telle que
  $$ J_{G_1}^\text{st}(X_1, f_1) = J_{G_2}^\text{st}(X_2, f_2), $$
  où les intégrales orbitales sont définies par rapport à des mesures compatibles sur les commutants. On dit que $f_1$ est un transfert de $f_2$.
\end{theorem}

On a aussi une version non standard du lemme fondamental.

\begin{definition}
  On dit qu'un triplet $(G_1, G_2, j_*)$ est non ramifié si
  \begin{itemize}
   \item $G_1$, $G_2$ sont non ramifiés,
   \item les fonctions $b, \check{b}$ prennent valeurs dans $\Q_{>0} \cap \Z_p^\times$.
  \end{itemize}
\end{definition}

\begin{theorem}[Lemme fondamental non standard, \cite{Wa08} 4.10]\label{prop:LF-non-standard}
  Supposons que $(G_1,G_2,j_*)$ est un triplet endoscopique non standard non ramifié. Soient $\mathfrak{k}_1 \subset \mathfrak{g}_1(F)$ et $\mathfrak{k}_2 \subset \mathfrak{g}_2(F)$ des réseaux hyperspéciaux. Alors $\mathbbm{1}_{\mathfrak{k}_1}$ est un transfert de $\mathbbm{1}_{\mathfrak{k}_2}$ si l'on utilise des mesures non ramifiées sur $G_1(F)$ et $G_2(F)$.
\end{theorem}
\begin{remark}
  Les intégrales orbitales dans \cite{Wa08} ne sont pas normalisées. Cependant on voit aisément qu'il existe une constante $c \in F^\times$ telle que $D_{G_1}(X_1) = c D_{G_2}(X_2)$ si $X_1$ et $X_2$ se correspondent; de plus, $c \in \mathfrak{o}_F^\times$ si $(G_1, G_2, j_*)$ est non ramifié. Donc notre formulation est équivalente à celle de \cite{Wa08}.
\end{remark}

Le transfert et le lemme fondamental sont énoncés pour les groupes simplement connexes. En pratique, on utilise une variante de ce théorème dans laquelle $G_2$ (ou $G_1$) est remplacé par un quotient.

\begin{lemma}\label{prop:int-isogenie}
  Soit $\sigma: G_2 \to \underline{G_2}$ une $F$-isogénie. Il existe une constante $c>0$ dépendant de mesures de Haar sur $G_2(F), \underline{G}_2(F)$, telle que pour tous $X_2 \in \mathfrak{g}_{2,\text{reg}}(F)$, $f_2 \in C_c^\infty(\mathfrak{g}_2)(F)$, on a
  $$ J_{G_2}^\text{st}(X_2, f_2) = c J_{\underline{G_2}}^\text{st}(\underline{X_2}, \underline{f_2}) $$
  où $\underline{X_2} := \sigma_*(X_2)$, $\underline{f_2} = (\sigma^*)^{-1}(f_2)$.
\end{lemma}
\begin{proof}
  Les intégrales orbitales stables $J_{\underline{G_2}}^\text{st}(\underline{X_2}, \underline{f_2})$ et $J_{G_2}^\text{st}(X_2, f_2)$ sont prises sur le même espace $((G_2)_{X_2} \backslash G_2)(F) = ((\underline{G_2})_{\underline{X_2}} \backslash \underline{G_2})(F)$, l'identification respectant les mesures.
\end{proof}

\begin{proposition}\label{prop:transfert-LF-isogenie}
  Soient $(G_1, G_2, j_*)$ un triplet endoscopique non standard et $\sigma: G_2 \to \underline{G_2}$ une $F$-isogénie. Identifions $\mathfrak{g}_2, \underline{\mathfrak{g}}_2$ et $C_c^\infty(\mathfrak{g}_2), C_c^\infty(\underline{\mathfrak{g}}_2)$ à l'aide de $\sigma$, alors l'assertion de \ref{prop:transfert-non-standard} reste valable pour $G_1$ et $\underline{G}_2$.

  Supposons $(G_1, G_2, j_*)$ non ramifié. Soient $\mathfrak{k}_1 \subset \mathfrak{g}_1(F)$, $\mathfrak{k}_2 \subset \mathfrak{g}_2(F)$ et $\underline{\mathfrak{k}}_2 \subset \underline{\mathfrak{g}}_2(F)$ des réseaux hyperspéciaux, qui correspondent aux modèles $\mathbf{G}_1$, $\mathbf{G}_2$ et $\underline{\mathbf{G}}_2$ définis sur $\mathfrak{o}_F$. Si $\sigma$ provient d'une $\mathfrak{o}_F$-isogénie $\mathbf{G}_2 \to \underline{\mathbf{G}}_2$, alors l'assertion de \ref{prop:LF-non-standard} reste valable pour $G_1$ et $\underline{G}_2$ si l'on utilise des mesures non ramifiées.
\end{proposition}
\begin{proof}
  La première assertion résulte de \ref{prop:transfert-non-standard} et \ref{prop:int-isogenie}. Pour la deuxième, \ref{prop:LF-non-standard} et \ref{prop:int-isogenie} fournit une constante $c>0$ telle que
  \begin{equation}\label{eqn:transfert-LF-isogenie}
    J_{G_1}(X_1, \mathbbm{1}_{\mathfrak{k}_2}) = c \cdot J_{\underline{G}_2}(\underline{X}_2, \mathbbm{1}_{\underline{\mathfrak{k}}_2})
  \end{equation}
  si $X_1 \in \mathfrak{g}_{1,\text{reg}}(F)$ et $X_2 \in \mathfrak{g}_{2,\text{reg}}(F)$ qui se correspondent et si l'on pose $\underline{X}_2 = \sigma_*(X_2)$.

  On peut prendre $X_1 \in \mathfrak{k}_1$, $X_2 \in \mathfrak{k}_2$ de réductions régulières, alors $\underline{X}_2 \in \underline{\mathfrak{k}}_2$ l'est aussi. Un résultat de Kottwitz adapté aux algèbres de Lie (\cite{Ko86} 7.3) montre que les deux intégrales orbitales stables dans \eqref{eqn:transfert-LF-isogenie} valent $1$ avec les mesures non ramifiées. D'où $c=1$, ce qu'il fallait démontrer.
\end{proof}

Nous ne considérons qu'une seule famille de triplets endoscopiques non standards. Prenons $G_1 = \Sp(2n)$, $G_2 = \Spin(2n+1)$. Identifions $X_{1,*}$ à $\Z^n$ avec la base standard $\{e_1, \ldots, e_n\}$. La base duale pour $X_1^*$ est notée par $\{f_1, \ldots, f_n \}$. Identifions $X_{2,*}$ au groupe des $(x_1, \ldots, x_n) \in \Z^n$ tels que $x_1 + \ldots + x_n \in 2\Z$. Alors
\begin{align*}
  \Sigma_1 & = \{\pm f_i \pm f_j : 1 \leq i \neq j \leq n\} \cup \{\pm 2 f_i : 1 \leq i \leq n \}, \\
  \check{\Sigma}_1 & = \{\pm e_i \pm e_j : 1 \leq i \neq j \leq n\} \cup \{\pm e_i : 1 \leq i \leq n \}, \\
  \Sigma_2 & = \{\pm f_i \pm f_j : 1 \leq i \neq j \leq n\} \cup \{\pm f_i : 1 \leq i \leq n \}, \\
  \check{\Sigma}_2 & = \{\pm e_i \pm e_j : 1 \leq i \neq j \leq n\} \cup \{\pm 2e_i : 1 \leq i \leq n \}.
\end{align*}

On a $X_{2,*} \subset X_{1,*}$, $X_{2,*,\Q} = X_{1,*,\Q}$ et $j_* = \identity$. Prenons des bijections
\begin{align*}
  \tau: & \Sigma_2 \to \Sigma_1 \\
  & \pm f_i \pm f_j \mapsto \pm f_i \pm f_j, \\
  & \pm f_i \mapsto \pm 2f_i; \\
  \check{\tau}: & \check{\Sigma}_1 \to \check{\Sigma}_2 \\
  & \pm e_i \pm e_j \mapsto \pm e_i \pm e_j,\\
  & \pm e_i \mapsto \pm 2e_i.
\end{align*}

Définissons $b: \Sigma_2 \to \Q_{>0}$ par $b(\pm f_i \pm f_j)=1$, $b(\pm f_i)=\frac{1}{2}$; définissons $\check{b}: \check{\Sigma}_1 \to \Q_{>0}$ par $\check{b}(\pm e_i \pm e_j) = 1$, $\check{b}(\pm e_i) = \frac{1}{2}$. Ces données fournissent un triplet endoscopique non standard $(G_1, G_2, j_*)$. La correspondance de classes de conjugaison est la suivante: $X \in \syp(2n)_\text{reg}(F)$ et $Y \in \spin(2n+1)_{\text{reg}}(F)=\so(2n+1)_{\text{reg}}(F)$ se correspondent si et seulement si $X \CVP Y$ (la notation dans \ref{def:CVP}).

Remarquons que ce triplet est non ramifié lorsque $p>2$. Les groupes $\Sp(2n)$, $\Spin(2n+1)$ et $\SO(2n+1)$ sont tous définis sur $\Z$ et $\Spin(2n+1) \to \SO(2n+1)$ est une isogénie sur $\Z[\frac{1}{2}]$. Cela permet d'appliquer \ref{prop:transfert-LF-isogenie} dans le cas non ramifié.

\subsection{Démonstration du transfert}
Soient $\epsilon = (\epsilon', \epsilon'') \in H(F)_\text{ss}$. On commence par établir le transfert local en $\epsilon$.

\begin{lemma}\label{prop:transfert-local}
  Supposons $H_\epsilon$ quasi-déployé. Il existe un voisinage $\mathfrak{U} \subset \mathfrak{h}_\epsilon(F)$ vérifiant les propriétés du paragraphe \ref{sec:voisinage-ss} tel que si $f \in C_{c,\asp}^\infty(\tilde{G})$, alors il existe $f^H_\epsilon \in C_c^\infty(\mathfrak{U}^\natural)$ telle que
  $$ J_{H,\tilde{G}}(\gamma, f) = J_H^\text{st}(\gamma, f^H_\epsilon) $$
  pour tout $\gamma \in \mathfrak{U}^\natural \cap H_{G-\text{reg}}(F)$.
\end{lemma}
\begin{proof}
  Soit $\mathfrak{U} \in \mathfrak{h}_\epsilon(F)$ un voisinage vérifiant les propriétés du paragraphe \ref{sec:voisinage-ss}. Montrons d'abord qu'il existe $g^H_\epsilon \in C_c^\infty(\mathfrak{U})$ telle que
  \begin{equation}\label{eqn:descente-1}
    J_{H,\tilde{G}}(\exp(Y)\epsilon, f) = J_{H_\epsilon}^\text{st}(Y, g^H_\epsilon)
  \end{equation}
  pour tout $Y \in \mathfrak{U}$ tel que $\exp(Y)\epsilon \in \mathfrak{U}^\natural \cap H_{G-\text{reg}}(F)$.

  Soient $\eta_1, \ldots, \eta_m$ des représentants des classes de conjugaison semi-simples qui correspondent à $\epsilon$. Quitte à rétrécir $\mathfrak{U}$, on peut supposer qu'il existe des voisinages $\mathfrak{V}_i \subset \mathfrak{g}_{\eta_i}$ vérifiant les propriétés du paragraphe \ref{sec:voisinage-ss}, tels que l'ensemble des éléments correspondant à des éléments dans $\mathfrak{U}^\natural$ est inclus dans $\mathfrak{V}_1^\natural \sqcup \cdots \sqcup \mathfrak{V}_m^\natural$. Alors
  $$ J_{H,\tilde{G}}(\gamma, f) = \sum_{i=1}^m J_{H,\tilde{G}}^{(i)}(\gamma, f) $$
  pour tout $\gamma \in \mathfrak{U}^\natural \cap H_{G-\text{reg}}(F)$ et tout $f \in C_{c,\asp}^\infty(\tilde{G})$, où $J_{H,\tilde{G}}^{(i)}(\gamma, f)$ est défini de la même façon que \eqref{eqn:integrale-endoscopique} sauf que la somme est restreinte aux $\delta \in \mathfrak{V}_i$. Écrivons $\gamma = \exp(Y)\epsilon$. On se ramène à démontrer que pour tout $1 \leq i \leq m$, il existe $g^{H,(i)}_\epsilon \in C_c^\infty(\mathfrak{U})$ tel que
  $$ J_{H,\tilde{G}}^{(i)}(\gamma, f) = J_{H_\epsilon}^\text{st}(Y, g^{H,(i)}_\epsilon). $$

  Fixons un $1 \leq i \leq m$. Conservons le formalisme de \S\ref{sec:descente-formalisme} et paramétrons $\mathcal{O}(\epsilon)$ et $\mathcal{O}(\eta_i)$ par

\begin{align*}
  \eta_i & \in \mathcal{O}(\bigoplus_u (L/L^\#, u, (W_u,h_u)) \oplus (W_+, \angles{\cdot|\cdot}_+) \oplus (W_-, \angles{\cdot|\cdot}_-)), \\
  \epsilon' & \in \mathcal{O}(\bigoplus_u (L/L^\#, u, (V'_u,h'_u)) \oplus (V'_+, q'_+) \oplus (V'_-, q'_-)),\\
  \epsilon'' & \in \mathcal{O}(\bigoplus_u (L/L^\#, -u, (V''_u,h''_u)) \oplus (V''_+, q''_+) \oplus (V''_-, q''_-)).
\end{align*}

  Alors
  \begin{align*}
    G_{\eta_i} & = \prod_u U(W_u, h_u) \times \Sp(W_+) \times \Sp(W_-), \\
    H'_{\epsilon'} & = \prod_u U(V'_u, h'_u) \times \SO(V'_+, q'_+) \times \SO(V'_-,q'_-), \\
    H''_{\epsilon''} & = \prod_u U(V''_u, h''_u) \times \SO(V''_+, q''_+) \times \SO(V''_-,q''_-),\\
    H_\epsilon &= H'_{\epsilon'} \times H''_{\epsilon''}.
  \end{align*}

  Prenons des images réciproques $\tilde{\eta}_i \in \rev^{-1}(\tilde{\eta}_i)$ telles que $\tilde{\eta}_i = \pm 1$ lorsque $\eta_i = \pm 1$. Par la descente des intégrales orbitales \ref{prop:descente-int}, il existe une fonction $f_i^\flat \in C_c^\infty(\mathfrak{V}_i)$ telle que
  \begin{equation}\label{eqn:descente-2}
    J_{H,\tilde{G}}^{(i)}(\gamma, f) = \sum_X \Delta(\exp(Y)\epsilon, \exp(X)\tilde{\eta}_i) J_{G_{\eta_i}}(X, f_i^\flat),
  \end{equation}
  où $X$ parcourt les représentants des classes de conjugaison dans $\mathfrak{g}_{\eta_i}(F)$ telles que $\exp(X)\eta_i$ corresponde à $\gamma$, ou ce qui revient au même, $X \CVP Y$.

  Effectuons les décompositions dans \S\ref{sec:descente-formalisme}
  \begin{align*}
    X & = ((X_u)_u, X_+, X_+), \\
    Y & = ((Y_u)_u, Y_+, Y_-),\\
    X_\bullet & = (X'_\bullet, X''_\bullet) \in \mathfrak{u}(W_u,h_u),\\
    Y_u & = (Y'_u, Y''_u) \in \mathfrak{u}(V'_u, h'_u) \times \mathfrak{u}(V''_u, h''_u),\\
    Y_\pm & = (Y'_\pm, Y''_\mp) \in \so(V'_\pm,q'_\pm) \times \so(V''_\mp, q''_\mp)
  \end{align*}
  où $\bullet \in \{u,+,-\}$ comme d'habitude.

  Il suffit de considérer le cas $f_i^\flat = (\prod_u f_u) \cdot f_+ \cdot f_-$ dont $f_u \in C_c^\infty(\mathfrak{u}(W_u,h_u))$, $f_\pm \in C_c^{\infty}(\syp(W_\pm))$. D'après la descente du facteur de transfert \ref{prop:descente-facteur-transfert}, quitte à rétrécir $\mathfrak{U}$ le côté à droite de \eqref{eqn:descente-2} s'exprime comme le produit $(\prod_u J_u) \cdot J_+ \cdot J_-$, où
  \begin{align*}
    J_u = J_u(Y_u) & := \sum_{X_u} \Delta_u(Y_u, X_u) J_{U(W_u,h_u)}(X_u, f_u), \\
    J_+ = J_+(Y_+) & := \sum_{X_+} \Delta_+(Y_+, X_+) J_{\Sp(W_+)}(X_+, f_+), \\
    J_- = J_-(Y_-) & := \sum_{X_-} \Delta_-(Y_-, X_-) J_{\Sp(W_-)}(X_-, f_-),
  \end{align*}
  où les éléments $X_\bullet$ parcourent les représentants de classes de conjugaison telles que $X_\bullet \CVP Y_\bullet$.

  On se ramène à démontrer que $Y_\bullet \mapsto J_\bullet(Y_\bullet)$ est une intégrale orbitale stable. Pour $J_u$, cela découle de \ref{prop:comparaison-transfert-classique} et du transfert sur l'algèbre de Lie (\cite{Wa08} 1.6) pour l'endoscopie des groupes unitaires. Pour $J_+$, fixons un $X_+$ dans la somme. Le transfert pour le groupe endoscopique $\Sp(W'_+) \times \SO(V''_-, q''_-)$ de $\Sp(W_+)$ fournit une fonction
  $$ g_+ \in C_c^\infty(\syp(W'_+) \times \so(V''_-, q''_-)) $$
  telle que
  $$ J_+(Y_+) = J_{\Sp(W'_+) \times \SO(V''_-, q''_-)}^\text{st}((X'_+, Y''_-), g_+). $$

  En décomposant $g_+ = \sum g'_+ \cdot g''_+$ où $g'_+ \in C_c^\infty(\syp(W'_+))$, $g''_+ \in C_c^\infty(\so(V''_-,q''_-))$ et en appliquant le transfert non standard \ref{prop:transfert-non-standard}, \ref{prop:transfert-LF-isogenie} au triplet $(\Sp(W'_+), \Spin(V'_+, q'_+), \ldots)$ et à l'isogénie $\Spin(V'_+, q'_+) \to \SO(V'_+, q'_+)$, on déduit que $Y_+ \mapsto J_+(Y_+)$ est une intégrale orbitale stable. Le même argument montre que $Y_- \mapsto J_-(Y_-)$ l'est aussi. Cela établit \eqref{eqn:descente-2}.
  
  Déduisons maintenant ce lemme de \eqref{eqn:descente-1} en remontant les intégrales orbitales. En effet, $Y \mapsto J_{H_\epsilon}^\text{st}(Y, g_\epsilon^H)$ est invariante par conjugaison géométrique par $H^\epsilon$ car $J_{H,\tilde{G}}(\cdot,f)$ l'est. D'après \ref{prop:descente-int}, il existe $f^H_\epsilon \in C_c^\infty(\mathfrak{U}^\natural)$ tel que $J_{H_\epsilon}^\text{st}(Y, g^H_\epsilon) = J_H^\text{st}(\exp(Y)\epsilon, f^H_\epsilon)$. Cela achève la démonstration.
\end{proof}

Pour démontrer \ref{prop:transfert}, nous ferons usage d'une caractérisation locale des intégrales orbitales stables due à Langlands et Shelstad. Adoptons la convention de \cite{LS1} concernant les mesures de Haar.

\begin{theorem}[\cite{LS2} 2.2A]
  Soit $M$ un $F$-groupe réductif quasi-déployé. Soit $J$ une fonction sur $M_\text{reg}(F)$. Supposons que:
  \begin{itemize}
    \item $J$ est stablement invariante;
    \item il existe un ouvert compact $C \subset M(F)$ tel que $J$ est à support dans $\bigcup_{m \in M(F)} mCm^{-1} \cap M_\text{reg}(F)$; on dit que $J$ est à support compact modulo conjugaison;
    \item pour tout $\epsilon \in M(F)_\text{ss}$, il existe un ouvert $\mathfrak{W}$ contenant $\epsilon$ et $f_\epsilon \in C_c^\infty(M(F))$ tels que
    $$ J(\gamma) = J_M^\text{st}(\gamma, f_\epsilon) $$
    pour tout $\gamma \in \mathfrak{W} \cap M_\text{reg}(F)$; on dit que $J$ est une intégrale orbitale stable locale en $\epsilon$.
  \end{itemize}

  Alors il existe $f \in C_c^\infty(M(F))$ tel que $J(\gamma) = J_M^\text{st}(\gamma, f)$ pour tout $\gamma \in M_\text{reg}(F)$.
\end{theorem}

\begin{proof}[Démonstration de \ref{prop:transfert}]
  On veut appliquer le théorème de caractérisation à $J_{H,\tilde{G}}(\cdot, f)$ sur $H(F)$. A priori, cette fonction stablement invariante est définie sur $H_{G-\text{reg}}(F)$. Or le transfert local \ref{prop:transfert-local} appliqué aux éléments réguliers permet de la prolonger sur $H_\text{reg}(F)$. Comme dans  l'endoscopie pour les groupes réductifs, on montre que $J_{H,\tilde{G}}$ est à support compact modulo conjugaison. Soit $\epsilon \in H(F)_\text{ss}$. Pour étudier le comportement local en $\epsilon$, on peut supposer $H_\epsilon$ quasi-déployé d'après les arguments dans \cite{LS2} 1.3. Maintenant \ref{prop:transfert-local} montre que $J_{H,\tilde{G}}$ est une intégrale orbitale stable locale en $\epsilon$. Cela permet de conclure.
\end{proof}

\subsection{Démonstration du lemme fondamental pour les unités}
Nous nous plaçons dans le cas non ramifié \ref{hyp:non-ramifie}. Notons $p$ la caractéristique résiduelle de $F$. Fixons $G=\Sp(W)$ et $H = H_{n',n''}$ un groupe endoscopique de $\tilde{G}$. Fixons un réseau autodual $L \subset W$ et posons $K = \text{Stab}_G(L)$. Identifions $K$ à un sous-groupe compact ouvert de $\tilde{G}$ à l'aide du modèle latticiel.

Conservons les notations de \ref{prop:LF-unite}. Cette section est consacrée à la démonstration de l'égalité

\begin{equation}\label{eqn:LF-unite}
  J_{H,\tilde{G}}(\gamma, f_K) = J_H^\text{st}(\gamma, \mathbbm{1}_{K_H})
\end{equation}
pour tout $\gamma \in H_{G-\text{reg}}(F)$, où on utilise les mesures non ramifiées sur $\tilde{G}, H(F)$ et des mesures compatibles sur les commutants.

Comme le transfert, ce lemme fondamental sera démontré par la méthode de descente. Pour ce faire, effectuons des réductions.

\begin{lemma}
  Si $\gamma$ n'est pas un élément compact, alors $J_{H,\tilde{G}}(\gamma, f_K) = J_H^\text{st}(\gamma, \mathbbm{1}_{K_H})=0$.
\end{lemma}
\begin{proof}
  La compacité ne dépend que de la classe de conjugaison géométrique (\cite{Wa08} 5.2). Supposons que $\gamma$ n'est pas compact, on a alors $\mathcal{O}^\text{st}(\gamma) \cap K_H = \emptyset$, d'où $J_H^\text{st}(\gamma, \mathbbm{1}_{K_H})=0$. D'autre part, si $\delta \in G(F)$ qui correspond à $\gamma$ n'est pas compact, alors $\mathcal{O}(\delta) \cap K = \emptyset$, d'où $J_{\tilde{G}}(\delta, f_K)=0$. Or la correspondance de classes de conjugaison semi-simples de \S\ref{sec:donnees-endoscopiques} préserve la compacité. On en déduit que $J_{H,\tilde{G}}(\gamma,f_K) = 0$.
\end{proof}

Supposons dès maintenant que $\gamma=(\gamma',\gamma'')$ est compact, alors on a la décomposition de Jordan topologique
\begin{align}
  \label{eqn:descente-gamma-1}\gamma &= \exp(Y)\epsilon,\\
  \label{eqn:descente-gamma-2} \epsilon &= (\epsilon',\epsilon''),\\
  \label{eqn:descente-gamma-3} Y &= (Y',Y''),
\end{align}
telle que $\epsilon$ est d'ordre fini premier à $p$.

\begin{lemma}
  Avec les hypothèses ci-dessus, il existe $\gamma_1 \in H_{G-\text{reg}}(F)$ tel que
  \begin{itemize}
    \item $\gamma, \gamma_1$ sont stablement conjugués;
    \item soit $\gamma_1 = \exp(Y_1)\epsilon_1$ la décomposition de Jordan topologique, alors $H_{\epsilon_1}$ est quasi-déployé.
  \end{itemize}

  De plus, supposons $H_\epsilon$ quasi-déployé, alors $\epsilon$ est stablement conjugué à un élément de $K_H$ si et seulement si $H_\epsilon$ est non ramifié.

  Les mêmes assetions restent valides pour un élément compact $\delta \in G_\text{reg}(F)$ et sa décomposition de Jordan topologique.
\end{lemma}
\begin{proof}
  Voir \cite{Wa08}, 5.13 (2) et 5.3 (v).
\end{proof}

\begin{lemma}\label{prop:existence-diagramme-nr}
  Supposons que $\gamma \in H_{G-\text{reg}}(F)$ est compact avec $\epsilon \in K_H$, alors il existe $\delta \in G(F)$ qui correspond à $\gamma$ avec la décomposition de Jordan topologique $\delta=\exp(X)\eta$ telle que $\eta \in K$.
\end{lemma}
\begin{proof}
  Cf. \cite{Wa08} 5.7 (ii).
\end{proof}

Traitons maintenant la descente d'une intégrale orbitale sur $\tilde{G}$.

\begin{lemma}\label{prop:int-orb-f_K}
  Soit $\eta \in K$ d'ordre fini premier à $p$, alors $G_\eta$ est non ramifié et $K_\eta := K \cap G_\eta(F)$ est encore un sous-groupe hyperspécial de $G_\eta(F)$. Notons $\mathfrak{k}_\eta \subset \mathfrak{g}_\eta(F)$ le sous-réseau hyperspécial associé. Pour tout $X \in \mathfrak{g}_\eta(F)$, on a
  \begin{align*}
    J_{\tilde{G}}(\exp(X)\eta, f_K) & = J_{G_\eta}(X, \mathbbm{1}_{\mathfrak{k}_\eta})\\
    & = J_{G_\eta}(\exp(X), \mathbbm{1}_{K_\eta}),
  \end{align*}
  si l'on utilise les mesures non ramifiées.
\end{lemma}
\begin{proof}
  D'après \cite{Wa08} 5.3 (iii), $G_\eta$ est non ramifié et $K_\eta := K \cap G_\eta(F)$ est un sous-groupe hyperspécial de $G_\eta(F)$.
  
  Posons $T^\flat := (G_\eta)_X$. Les arguments dans \cite{Wa08} 5.11 montrent que
  \begin{multline}
    \int_{G_{\exp(X)\eta}(F) \backslash G(F)} f_K(\tilde{x}^{-1} \exp(X_j)\eta \tilde{x}) \dd \dot{x} = \\
    \sum_{\dot{x} \in T^\flat(F) \backslash G_\eta(F) / K_\eta} \mes(T^\flat(F) \backslash T^\flat(F) x K_\eta) f_K(\tilde{x}^{-1} \exp(X)\eta \tilde{x}),
  \end{multline}
  où $\tilde{x}$ est une image réciproque quelconque de $x$. Montrons que

  \begin{equation}\label{eqn:int-orb-f_K}
    \forall \tilde{x} \in \rev^{-1}(G_\eta(F)),\quad f_K(\tilde{x}^{-1}\exp(X)\eta \tilde{x}) = \mathbbm{1}_K(x^{-1}\exp(X)\eta x).
  \end{equation}
  En effet, si $x^{-1}\exp(X)\eta x \notin K$, alors les deux côtés valent $0$. S'il appartient à $K$, le côté à droite vaut $1$. Supposons donc $\tilde{x}^{-1}\exp(X)\eta \tilde{x} \in \noyau K$ où $\noyau \in \Ker(\rev)$. En prenant la limite des $(p^{n_k})$-ièmes puissances avec $n_k \to +\infty$ une suite convenable, on obtient $\tilde{x}^{-1} \eta \tilde{x} \in \noyau K$. Or $\tilde{x}$ centralise $\eta$, d'où $\noyau=1$.

  Maintenant on peut reprendre les arguments dans \cite{Wa08} 5.11 et on arrive à
  $$ J_{\tilde{G}}(\exp(X)\eta, f_K) = J_{G_\eta}(X, \mathbbm{1}_{\mathfrak{k}_\eta}). $$
\end{proof}

\begin{proof}[Démonstration de \ref{prop:LF-unite}]
  Vu les résultats précédents, on peut supposer que $\gamma$ est compact avec les décompositions \eqref{eqn:descente-gamma-1}-\eqref{eqn:descente-gamma-3} et que $H_\epsilon$ est quasi-déployé. Soit $\delta \in G(F)$ qui correspond à $\gamma$, alors $\delta$ est compact avec la décomposition de Jordan topologique $\delta=\exp(X)\eta$. S'il n'existe pas $\delta$ comme ci-dessus satisfaisant à $\eta \in K$, alors $J_{H,\tilde{G}}(\gamma, f_K)=0$, et $H_\epsilon$ n'est pas non ramifié d'après \ref{prop:existence-diagramme-nr}. Dans ce cas-là $\mathcal{O}^\text{st}(\gamma)$ ne coupe aucun sous-groupe hyperspécial de $H(F)$ (\cite{Wa08} 5.3 (v)), d'où $J_H^\text{st}(\gamma, \mathbbm{1}_{K_H})=0$ et \ref{prop:LF-unite} est vérifié.

  Supposons donc $\eta \in K$, alors $G_\eta$ est non ramifié et $K_\eta := K \cap G_\eta(F)$ est encore un sous-groupe hyperspécial de $G_\eta(F)$ par \ref{prop:int-orb-f_K}.

  Prenons un ensemble de représentants $\{\delta_j\}_{j \in J}$ de $\mathcal{O}^\text{st}(\delta)/\text{conj}$. Soit $J_0$ l'ensemble des $j \in J$ tel que $\mathcal{O}(\delta_j) \cap K \neq \emptyset$. D'après \cite{Wa08} 5.11, on peut prendre les $\delta_j$ ($j \in J_0$) avec décompositions de Jordan topologiques de la forme
  $$ \delta_j = \exp(X_j)\eta, $$
  où $\{X_j\}_{j \in J_0}$ forment un ensemble de représentants des classes de conjugaison dans $\mathfrak{g}_\eta$ coupant $\mathfrak{k}_\eta$ dans $\mathcal{O}^\text{st}(X)$. On a
  \begin{equation}\label{eqn:descente-nr-G}
    J_{H,\tilde{G}}(\gamma, f_K) = \sum_{j \in J_0} \Delta(\exp(X_j)\eta,\exp(Y)\epsilon) J_{\tilde{G}}(\exp(X_j)\eta, f_K).
  \end{equation}

  Notons comme précédemment $\mathfrak{k}_\eta \subset \mathfrak{g}_\eta(F)$ le sous-réseau hyperspécial associé à $K_\eta$. Grâce à \ref{prop:int-orb-f_K}, pour tout $j \in J_0$ on a
  $$ J_{\tilde{G}}(\exp(X_j)\eta, f_K) = J_{G_\eta}(X_j, \mathbbm{1}_{\mathfrak{k}_\eta}). $$

  Prenons des paramètres pour $\mathcal{O}(\eta)$ et $\mathcal{O}(\epsilon)$ comme dans la preuve de \ref{prop:transfert-local}. On a
\begin{align*}
  G_\eta & = \prod_u U(W_u, h_u) \times \Sp(W_+) \times \Sp(W_-), \\
  H'_{\epsilon'} & = \prod_u U(V'_u, h'_u) \times \SO(V'_+, q'_+) \times \SO(V'_-,q'_-), \\
  H''_{\epsilon''} & = \prod_u U(V''_u, h''_u) \times \SO(V''_+, q''_+) \times \SO(V''_-,q''_-),\\
  H_\epsilon &= H'_{\epsilon'} \times H''_{\epsilon''}.
\end{align*}

  Introduisons aussi les objets $X_\pm$, $X_u$, $Y_\pm$, $Y_u$, etc... dans \S \ref{sec:descente-formalisme}. Les mêmes arguments qu'en \ref{prop:transfert-local} permettent d'exprimer $J_{H,\tilde{G}}(\gamma, f_K)$ comme un produit $(\prod_u J_u) J_+ J_-$ avec
  \begin{align*}
    J_u = J_u(Y_u) & := \sum_{X_u} \Delta_u(Y_u, X_u) J_{U(W_u,h_u)}(X_u, f_u), \\
    J_+ = J_+(Y_+) & := \sum_{X_+} \Delta_+(Y_+, X_+) J_{\Sp(W_+)}(X_+, f_+), \\
    J_- = J_-(Y_-) & := \sum_{X_-} \Delta_-(Y_-, X_-) J_{\Sp(W_-)}(X_-, f_-).
  \end{align*}
  Ces expressions sont, pour l'essentiel, composées des intégrales endoscopiques pour les groupes classiques et des intégrales orbitales stables auxquels le transfert non standard s'applique (cf. la démonstration de \ref{prop:transfert-local}).

  Supposons que $H_\epsilon$ n'est pas non ramifié. On a vu que $J_H^\text{st}(\gamma, \mathbbm{1}_{K_H})=0$ dans ce cas. D'autre part, un résultat de Kottwitz (\cite{Ko86} 7.5) adapté aux algèbres de Lie montre qu'au moins l'un des $J_u, J_+, J_-$ s'annule, d'où $J_{H,\tilde{G}}(\gamma,f_K)=0$ et \ref{prop:LF-unite} est vérifié.

  Supposons désormais $H_\epsilon$ non ramifié. C'est loisible de supposer $\epsilon \in K_H$ et dans ce cas $K_{H,\epsilon} := K_H \cap H_\epsilon(F)$ est un sous-groupe hyperspécial de $H_\epsilon(F)$ (\cite{Wa08} 5.3 (iii),(v)).  Notons $\mathfrak{k}_\epsilon \subset \mathfrak{h}_\epsilon(F)$ le sous-réseau hyperspécial associé. On le décompose en des réseaux hyperspéciaux:
  \begin{align*}
    \mathfrak{k}_{H,\epsilon} &= \bigoplus_u \mathfrak{k}_{H,u} \oplus \mathfrak{k}_{H,+} \oplus \mathfrak{k}_{H,-} , \\
    \mathfrak{k}_{H,u} & \subset \mathfrak{u}(V'_u,h'_u) \oplus \mathfrak{u}(V''_u,h''_u),\\
    \mathfrak{k}_{H,\pm} & \subset \mathfrak{so}(V'_\pm, q'_\pm) \oplus \mathfrak{so}(V'_\mp, q'_\mp).
  \end{align*}

  Alors la descente des intégrales orbitales stables (\cite{Wa08} 5.11) donne
  \begin{align*}
    J_H^\text{st}(\gamma, \mathbbm{1}_{K_H}) &= J_{H_\epsilon}(Y, \mathbbm{1}_{\mathfrak{k}_H}) \\
    &= \prod_u J_{H,u} \cdot J_{H,+} \cdot J_{H,-},
  \end{align*}
  où
  \begin{align*}
    J_{H,u} & := J_{U(V'_u, h'_u) \times U(V''_u, h''_u)}^\text{st} (Y_u, \mathbbm{1}_{\mathfrak{k}_{H,u}}), \\
    J_{H,+} & := J_{\SO(V'_+,q'_+) \times \SO(V''_-,q''_-)}^\text{st}(Y_+, \mathbbm{1}_{\mathfrak{k}_{H,+}}), \\
    J_{H,-} & := J_{\SO(V'_-,q'_-) \times \SO(V''_+,q''_+)}^\text{st}(Y_-, \mathbbm{1}_{\mathfrak{k}_{H,-}}).
  \end{align*}

  Comme les hypothèses dans \S\ref{sec:descente-nr} sont satisfaites, les facteurs de transfert $\Delta_\bullet$ ($\bullet \in \{u,+,-\}$) sont normalisés au sens de \cite{Wa08} \S 4.7 d'après \ref{prop:comparaison-transfert-classique}. Les arguments qui restent sont analogues à ceux pour \ref{prop:transfert-local}, sauf que l'on utilise le lemme fondamental sur les algèbres de Lie pour les groupes classiques et le lemme fondamental non standard sous la forme de \ref{prop:transfert-LF-isogenie}. Le bilan est que $J_u = J_{H,u}$, $J_+ = J_{H,+}$ et $J_{H,-}$, ce qui achève la démonstration.
\end{proof}

\bibliographystyle{abbrv-fr}
\bibliography{endoscopie}

\printindex

\end{document}